\documentclass{amsart}

\usepackage{amsmath}

\usepackage{amsfonts,amssymb,amsthm,enumerate}

\usepackage{mathpazo}

\newcommand{\set}[2]{\left\{#1\,:\,#2\right\}}

\def\DD{\mathbb{D}}

\theoremstyle{definition}

\newtheorem*{defi}{Definition}

\theoremstyle{plain}

\newtheorem{thm}{Theorem}

\newtheorem{prop}[thm]{Proposition}

\newtheorem{lem}[thm]{Lemma}

\numberwithin{equation}{section}
\numberwithin{thm}{section}

\begin{document}



\title{A new approach to velocity averaging lemmas in Besov spaces}

\author{Diogo Ars\'enio}
\author{Nader Masmoudi}

\address{Paris 7 - CNRS}
\email{arsenio@math.jussieu.fr}

\address{Courant Institute}
\email{masmoudi@cims.nyu.edu}

\date{\today}

\begin{abstract}
	We develop a new approach to velocity averaging lemmas based on the dispersive properties of the kinetic transport operator. This method yields unprecedented sharp results, which display, in some cases, a gain of one full derivative. Moreover, the study of dispersion allows to treat the case of $L_x^rL^p_v$ integrability with $r\leq p$. We also establish results on the control of concentrations in the degenerate $L_{x,v}^1$ case, which is fundamental in the study of the hydrodynamic limit of the Boltzmann equation.
\end{abstract}


\thanks{The two authors were partially supported by the NSF grant DMS-0703145.}

\maketitle

\tableofcontents


\section{Introduction}

The regularizing properties of the kinetic transport equation were first established in \cite{GPS85} for the basic Hilbertian case. Essentially, the main result therein states that if $f(x,v),g(x,v)\in L^2\left(\mathbb{R}^D_x\times\mathbb{R}^D_v\right)$, where $D\geq 1$ is the dimension, satisfy the stationary transport relation
\begin{equation}\label{transport 1}
	v\cdot\nabla_xf(x,v)=g(x,v),
\end{equation}
in the sense of distributions, then, for any given $\phi(v)\in C_0^\infty\left(\mathbb{R}^D\right)$, the velocity average verifies that
\begin{equation}
	\int_{\mathbb{R}^D}f(x,v)\phi(v)dv\in H^\frac{1}{2}\left(\mathbb{R}^D_x\right).
\end{equation}

This kind of smoothing property was then further investigated in \cite{GLPS88}, where the Hilbertian case was refined and some extensions to the non-Hilbertian case $f,g\in L^p\left(\mathbb{R}^D_x\times\mathbb{R}^D_v\right)$, with $p\neq 2$, were obtained, but these were not optimal. The first general results in the non-Hilbertian case establishing an optimal gain of regularity were obtained in \cite{diperna2} using Besov spaces and interpolation theory. The methods employed therein were not optimal in all aspects, though, but they were robust and thus allowed to further extend the averaging lemmas to settings bearing much more generality and still exhibiting an optimal gain of regularity. To be precise, the results from \cite{diperna2} were able to treat the case
\begin{equation}\label{transport 2}
	v\cdot\nabla_xf(x,v)=\left(1-\Delta_x\right)^\alpha\left(1-\Delta_v\right)^\beta g(x,v),
\end{equation}
where $0\leq \alpha<1$, $\beta\geq 0$, $f\in L^p\left(\mathbb{R}^D_x\times\mathbb{R}^D_v\right)$ and $g\in L^q\left(\mathbb{R}^D_x\times\mathbb{R}^D_v\right)$, with $1< p,q<\infty$. Another interesting generalization from \cite{diperna2} concerned the case where $f,g\in L^q\left(\mathbb{R}^D_v;L^p\left(\mathbb{R}^D_x\right)\right)$ satisfy the transport equation \eqref{transport 1} with $1<q\leq p<\infty$.

The corresponding results in standard Sobolev spaces were then obtained in \cite{bezard}, but the methods of proof were complicated and relied on harmonic analysis on product spaces. Unfortunately, as pointed out in \cite{westdickenberg}, the proofs in \cite{bezard} were flawed in the non-homogeneous cases, i.e. in the case $f\in L^p\left(\mathbb{R}^D_x\times\mathbb{R}^D_v\right)$ and $g\in L^q\left(\mathbb{R}^D_x\times\mathbb{R}^D_v\right)$, with $1< p,q<\infty$, $p\neq q$, and in the case $f,g\in L^q\left(\mathbb{R}^D_v;L^p\left(\mathbb{R}^D_x\right)\right)$, with $1<q< p<\infty$.

Note that another general approach, which yielded similar results in abstract interpolation spaces resulting from the real interpolation of Besov spaces, was developed in \cite{jabin}.

The interesting case of velocity averaging where $f(x,v)$ and $g(x,v)$ have more local integrability in $v$ than in $x$, i.e. $f,g\in L^p\left(\mathbb{R}^D_x;L^q\left(\mathbb{R}^D_v\right)\right)$ with $1<p< q<\infty$, was not addressed until \cite{westdickenberg}. To be precise, the simple question raised therein was: is it possible to improve, at least locally, the properties of the velocity averages in the case $f,g\in L^p\left(\mathbb{R}^D_x;L^q\left(\mathbb{R}^D_v\right)\right)$, with $p\leq q$, with respect to the case $f,g\in L^p\left(\mathbb{R}^D_x\times\mathbb{R}^D_v\right)$? The answer from \cite{westdickenberg} was definitely affirmative, even though it did not provide a general approach to this setting. Some other similar but very specific cases were treated in \cite{jabin}.

In the present work, we provide a very general and robust method to treat some of the cases where the local integrability in velocity is improved. Our approach has its own limitations, though, and doesn't apply to the whole range of parameters $1\leq p\leq q\leq\infty$. However, as discussed later on, it exhibits some optimality in the range of parameters where it is valid.

The main idea in our analysis consists in utilizing the so-called dispersive properties of the kinetic transport equation, together with standard averaging methods. This type of dispersive estimates was first used in \cite{BD85} and further developed in \cite{CP96, ovcharov}. It is at first uncertain whether these dispersive properties are even related to velocity averaging lemmas. However, our work clearly establishes a rather strong link between these two properties.

Most of the above-mentioned developments came to include more cases that were imposed by the underlying applications. In particular, one of the last improvements of averaging lemmas from \cite{golse3} was motivated by the hydrodynamic limit of the Boltzmann equation \cite{golse}, where it was necessary to have a refined averaging lemma in $L^1\left(\mathbb{R}^D_x\times\mathbb{R}^D_v\right)$. This was maybe the first time that some dispersive estimates were used in the proof of an averaging lemma. This paper is also somewhat motivated by the hydrodynamic limit of the Boltzmann equation but in the case of long-range interactions \cite{arsenio, arsenio3}. Thus, our work allows to obtain an extension of the averaging lemma in $L^1\left(\mathbb{R}^D_x\times\mathbb{R}^D_v\right)$ from \cite{golse3}, which is crucially employed in \cite{arsenio, arsenio3}. The main novelty in the $L^1$ averaging lemmas compared to \cite{golse3} is that we are able to include derivatives in the right-hand side of the transport equation, i.e. we consider the relation \eqref{transport 2} rather than \eqref{transport 1}. This is actually crucial when dealing with the Boltzmann equation in the presence of long-range interactions, since the Boltzmann collision operator is only defined (after renormalization) in the sense of distributions as a singular integral operator. We present this application of our main results to the theory of hydrodynamic limits in Section \ref{L1}. See also \cite{arsenio2} for a different approach to the same problem based on hypoellipticity.


\section{Technical tools}

The results in this article rely on the theory of function spaces and interpolation. Thus, we present in this preliminary section the main tools and notations that will be utilized throughout this work and we refer the reader to \cite{bergh, triebel} for more details on the subjects.

\subsection{Besov spaces}\label{LP decomposition}

We will denote the Fourier transform
\begin{equation}
	\hat f(\xi)=\mathcal{F}f\left(\xi\right)=\int_{\mathbb{R}^D} e^{- i \xi \cdot v} f(v) dv
\end{equation}
and its inverse
\begin{equation}
	\tilde g(v)=\mathcal{F}^{-1} g\left(v\right)=\frac{1}{\left(2\pi\right)^D}\int_{\mathbb{R}^D} e^{i v \cdot \xi} g(\xi) d\xi.
\end{equation}

We introduce now a standard Littlewood-Paley decomposition of the frequency space into dyadic blocks. To this end, let $\psi(\xi),\varphi(\xi)\in C_0^\infty\left(\mathbb{R}^D\right)$ be such that
\begin{equation}
	\begin{gathered}
		\psi,\varphi\geq 0\text{ are radial},
		\quad \mathrm{supp}\psi\subset\left\{|\xi|\leq 1\right\},
		\quad\mathrm{supp}\varphi\subset\left\{\frac{1}{2}\leq |\xi|\leq 2\right\}\\
		\text{and}\quad1= \psi(\xi)+\sum_{k=0}^\infty \varphi\left(2^{-k}\xi\right),\quad\text{for all }\xi\in\mathbb{R}^D.
	\end{gathered}
\end{equation}
Defining the scaled functions $\psi_{\delta}(\xi)=\psi\left(\frac{\xi}{\delta}\right)$ and $\varphi_{\delta}(\xi)=\varphi\left(\frac{\xi}{\delta}\right)$, for any $\delta>0$, one has then that
\begin{equation}
	\begin{gathered}
		\mathrm{supp}\psi_{\delta}\subset\left\{ |\xi|\leq \delta\right\},
		\quad \mathrm{supp}\varphi_{\delta}\subset\left\{\frac{\delta}{2}\leq |\xi|\leq 2\delta\right\}\\
		\text{and}\quad 1\equiv \psi_\delta+\sum_{k=0}^\infty \varphi_{\delta 2^k}.
	\end{gathered}
\end{equation}

Furthermore, we will use the Fourier multiplier operators
\begin{equation}
	S_\delta ,\Delta_\delta:\mathcal{S}'\left(\mathbb{R}^D\right)\rightarrow\mathcal{S}'\left(\mathbb{R}^D\right)
\end{equation}
(here $\mathcal{S}'$ denotes the space of tempered distributions) defined by
\begin{equation}\label{dyadic block}
	S_\delta f= \left(\mathcal{F}^{-1}\psi_\delta\right)*f
	\quad\text{and}\quad
	\Delta_\delta f= \left(\mathcal{F}^{-1}\varphi_\delta\right)*f,
\end{equation}
so that
\begin{equation}
	S_\delta f+\sum_{k=0}^\infty\Delta_{\delta 2^k}f=f,
\end{equation}
where the series is convergent in $\mathcal{S}'$.

For notational convenience, we also introduce, for every $0<\delta_1<\delta_2$, the operators
\begin{equation}\label{dyadic block 2}
	\Delta_0=S_1
	\qquad\text{and}\qquad
	\Delta_{\left[\delta_1,\delta_2\right]}=S_{2\delta_2}-S_{\delta_1},
\end{equation}
so that $\mathcal{F} \Delta_{\left[\delta_1,\delta_2\right]} f$ coincides with $\mathcal{F}f$ on $\left\{\delta_1\leq \xi\leq \delta_2\right\}$ and is supported on the domain $\left\{\frac{\delta_1}{2}\leq \xi\leq 2\delta_2\right\}$.


Now, we may define the standard Besov spaces $B^{s}_{p,q}\left(\mathbb{R}^D\right)$, for any $s \in \mathbb{R}$ and $1\leq p,q\leq \infty$, as the subspaces of tempered distributions endowed with the norm
\begin{equation}
	\left\|f\right\|_{B^{s}_{p,q}\left(\mathbb{R}^D\right)}=
	\left(\left\|\Delta_0 f\right\|_{L^p\left(\mathbb{R}^D\right)}^q
	+\sum_{k=0}^\infty 2^{ksq}
	\left\|\Delta_{2^k}f\right\|_{L^p\left(\mathbb{R}^D\right)}^q\right)^\frac{1}{q},
\end{equation}
if $q<\infty$, and with the obvious modifications in case $q=\infty$.


We also introduce the homogeneous Besov spaces $\dot B^{s}_{p,q}\left(\mathbb{R}^D\right)$, for any $s \in \mathbb{R}$ and $1\leq p,q\leq \infty$, as the subspaces of tempered distributions endowed with the semi-norm
\begin{equation}
	\left\|f\right\|_{\dot B^{s}_{p,q}\left(\mathbb{R}^D\right)}=
	\left(
	\sum_{k=-\infty}^\infty 2^{ksq}
	\left\|\Delta_{2^k}f\right\|_{L^p\left(\mathbb{R}^D\right)}^q\right)^\frac{1}{q},
\end{equation}
if $q<\infty$, and with the obvious modifications in case $q=\infty$.

For functions depending on two variables $x\in\mathbb{R}^D$ and $v\in\mathbb{R}^D$, we will use the Littlewood-Paley decomposition on each variable. That is, denoting $\DD = \{0\} \cup \set{2^k\in\mathbb{N}}{k\in\mathbb{N}} $, we can write
\begin{equation}
	f(x,v) = \sum_{i,j \in \DD} \Delta_i^x \Delta_j^v f(x,v),
\end{equation}
for any $f(x,v)\in\mathcal{S}'\left(\mathbb{R}^D_x\times\mathbb{R}^D_v\right)$, where we employ the superscripts to emphasize that the multipliers $\Delta_i^x$ and $\Delta_j^v$ solely act on the respective variables $x$ and $v$. Thus, we define now the mixed Besov spaces $B^{t,s}_{r,p,q} \left(\mathbb{R}^D_x\times\mathbb{R}^D_v\right)$, for any $s,t \in \mathbb{R}$ and $1\leq p,q,r\leq \infty$, as the subspaces of tempered distributions endowed with the norm
\begin{equation}
	\left\|f\right\|_{B^{t,s}_{r,p,q}\left(\mathbb{R}^D_x\times\mathbb{R}^D_v\right)}=
	\left(
	\sum_{i,j\in\DD} \left(1\vee i\right)^{tq}\left(1\vee j\right)^{sq}
	\left\|\Delta_i^x \Delta_j^v f\right\|_{L^r\left(\mathbb{R}^D_x;L^p\left(\mathbb{R}^D_v\right)\right)}^q\right)^\frac{1}{q},
\end{equation}
if $q<\infty$, and with the obvious modifications in case $q=\infty$, where the symbol $a\vee b$, for any $a,b\in\mathbb{R}$, stands for the maximum between $a$ and $b$.

Next, in addition to the spaces $ L^r\left( \mathbb{R}^D_x ; B^{s}_{p,q}\left(\mathbb{R}^D_v\right)  \right)  $, for any $s \in \mathbb{R}$ and $1\leq p,q,r\leq \infty$, which are defined as $L^r$ spaces with values in the Banach spaces $B^{s}_{p,q}$, we further define the spaces $ \widetilde L^r\left( \mathbb{R}^D_x ; B^{s}_{p,q}\left(\mathbb{R}^D_v\right)  \right)  $ as the subspaces of tempered distributions endowed with the norm
\begin{equation}
	\left\|f\right\|_{ \widetilde L^r \left( \mathbb{R}^D_x ;   B^{s}_{p,q}\left(\mathbb{R}^D_v\right) \right)}=
	\left(\left\|\Delta_0^v  f\right\|_{ L^r\left(\mathbb{R}^D_x;L^p\left(\mathbb{R}^D_v\right)\right) }^q
	+\sum_{k=0}^\infty 2^{ksq}
	\left\|\Delta_{2^k}^vf\right\|_{L^r\left(\mathbb{R}^D_x;L^p\left(\mathbb{R}^D_v\right)\right) }^q\right)^\frac{1}{q},
\end{equation}
if $q<\infty$, and with the obvious modifications in case $q=\infty$. This kind of spaces were first introduced by Chemin and Lerner in \cite{CL95} with $(x,v)$ replaced by $(t,x)$ (time and space) and were used in many problems related to the Navier-Stokes equations (cf. \cite{CM01} for instance).

One can check easily that, if $q \geq r $, then
\begin{equation}
	 L^r\left( \mathbb{R}^D_x ; B^{s}_{p,q}\left(\mathbb{R}^D_v\right)  \right)
	\subset  \widetilde L^r\left( \mathbb{R}^D_x ;   B^{s}_{p,q}\left(\mathbb{R}^D_v\right) \right),
\end{equation}
and that, if $q \leq r $, then
\begin{equation}
	\widetilde L^r\left( \mathbb{R}^D_x ; B^{s}_{p,q}\left(\mathbb{R}^D_v\right)  \right)
	\subset  L^r\left( \mathbb{R}^D_x; B^{s}_{p,q}\left(\mathbb{R}^D_v\right) \right).
\end{equation}
Furthermore, it holds that
\begin{equation}
	B^{0,s}_{r,p,1} \left(\mathbb{R}^D_x\times\mathbb{R}^D_v\right)
	\subset \widetilde L^r\left( \mathbb{R}^D_x ;   B^{s}_{p,q}\left(\mathbb{R}^D_v\right) \right)
	\subset B^{0,s}_{r,p,\infty} \left(\mathbb{R}^D_x\times\mathbb{R}^D_v\right),
\end{equation}
for all $s\in \mathbb{R}$ and $1\leq p,q,r\leq \infty$.


\subsection{Real interpolation theory}\label{interpolation}

We present now some useful elements from real interpolation theory. We will merely discuss properties that will be useful in the sequel and we refer to \cite{bergh} for more details on the subject.

First of all, we briefly recall the $K$-method of real interpolation. To this end, we consider any couple of normed spaces $A_0$ and $A_1$ compatible in the sense that they are embedded in a common topological vector space. The sum $A_1+A_0$ is the normed space defined by the norm
\begin{equation}
		\left\|a\right\|_{A_1+A_0} = \inf_{\substack{ a_0\in A_0, a_1\in A_1 \\ a=a_0+a_1 }}
		\left\|a_0\right\|_{A_0} + \left\|a_1\right\|_{A_1}.
\end{equation}
For any $0<\theta<1$, $1\leq q\leq \infty$, we define the normed space $\left[A_0,A_1\right]_{\theta,q}\subset A_0+A_1$ by the norm
\begin{equation}
	\left\|a\right\|_{\left[A_0,A_1\right]_{\theta,q}}=
	\left(
	\int_0^\infty \left(t^{-\theta}K(t,a)\right)^q \frac{dt}{t}
	\right)^\frac{1}{q},
\end{equation}
if $q<\infty$, and
\begin{equation}
	\left\|a\right\|_{\left[A_0,A_1\right]_{\theta,\infty}}=
	\sup_{t>0}t^{-\theta}K(t,a),
\end{equation}
if $q=\infty$, where the $K$-functional is defined, for any $a\in A_0+A_1$, $t>0$, by
\begin{equation}
	K(t,a) = \inf_{\substack{ a_0\in A_0, a_1\in A_1 \\ a=a_0+a_1 }}
	\left\|a_0\right\|_{A_0} + t \left\|a_1\right\|_{A_1}.
\end{equation}

Then, the normed space $\left[A_0,A_1\right]_{\theta,q}$ is an exact interpolation space of exponent $\theta$. In other words, this means that, considering another couple of compatible normed spaces $B_0$ and $B_1$, for any operator $T$ bounded from $A_0$ into $B_0$ and from $A_1$ into $B_1$, the operator $T$ is bounded from $\left[A_0,A_1\right]_{\theta,q}$ into $\left[B_0,B_1\right]_{\theta,q}$ as well, with an operator norm satisfying
\begin{equation}
	\left\|T\right\|_{\left[A_0,A_1\right]_{\theta,q}\to\left[B_0,B_1\right]_{\theta,q}}
	\leq
	\left\|T\right\|_{A_0\to B_0}^{1-\theta}
	\left\|T\right\|_{A_1\to B_1}^\theta.
\end{equation}

It turns out that there are other equivalent methods yielding the same interpolation spaces as the $K$-method. In particular, we are now briefly presenting the construction of interpolation spaces known as \emph{espaces de moyennes} (as coined by Lions and Peetre \cite{lions6}), which is equivalent to the $K$-method of interpolation but has the advantage of being slightly more general. We refer to \cite{bergh} for more details on these methods.

Thus, for any $0<\theta<1$ and $1\leq q, q_0,q_1<\infty$ such that $\frac 1q=\frac{1-\theta}{q_0}+\frac{\theta}{q_1}$, and any compatible couple of normed spaces $A_0$ and $A_1$, we define the following norms
\begin{equation}
	\inf_{\substack{ a_0(s)\in A_0, a_1(s)\in A_1 \\ a=a_0(s)+a_1(s) }} \left( \left\|s^{-\theta} a_0(s)\right\|_{L^{q_0}\left((0,\infty),\frac{ds}{s};A_0\right)} + \left\|s^{1-\theta} a_1(s)\right\|_{L^{q_1}\left((0,\infty),\frac{ds}{s};A_1\right)} \right)
\end{equation}
and
\begin{equation}\label{espace moyenne}
	\inf_{\substack{ a_0(s)\in A_0, a_1(s)\in A_1 \\ a=a_0(s)+a_1(s) }} \left( \left\|s^{-\theta} a_0(s)\right\|^{q_0}_{L^{q_0}\left((0,\infty),\frac{ds}{s};A_0\right)} + \left\|s^{1-\theta} a_1(s)\right\|^{q_1}_{L^{q_1}\left((0,\infty),\frac{ds}{s};A_1\right)} \right)^\frac{1}{q}.
\end{equation}
It is possible to show (cf. \cite[Thm 3.12.1]{bergh}) that both norms above are equivalent to $\left\|a\right\|_{\left[A_0,A_1\right]_{\theta,q}}$ and thus define the same interpolation space.

%


\section{Main results}\label{main results}

In this section, we present our main results.

\subsection{Velocity averaging in $L^1_xL^p_v$, inhomogeneous case}

Our first result concerns the endpoint case $L^1_xL^p_v$.

\begin{thm}\label{pre averaging lemma p}
	Let $f(x,v) \in \widetilde L^1\left(\mathbb{R}^D_x;B^\alpha_{p,q}\left(\mathbb{R}^D_v\right)\right)$, where $1\leq p, q \leq \infty$ and $\alpha>-D\left(1-\frac 1p\right)>-1 $, be such that
	\begin{equation}\label{pre averaging lemma p.0}
		v \cdot \nabla_x f = g
	\end{equation}
	for some $g(x,v) \in \widetilde L^1\left(\mathbb{R}^D_x;B^\beta_{p,q}\left(\mathbb{R}^D_v\right)\right)$, where $\beta\in\mathbb{R}$.
	
	If $\beta + D \left(1-\frac1p\right)  <1$, then
	\begin{equation}
		\int_{\mathbb{R}^D} f(x,v) dv \in B^s_{p,q}\left(\mathbb{R}^D_x\right),
	\end{equation}
	where $s = \frac{\alpha+D \left( 1 - \frac 1 p \right)}{1+\alpha-\beta}  - D \left( 1 - \frac 1 p \right) $, and the following estimate holds
	\begin{equation}\label{pre averaging lemma p.1}
		\begin{gathered}
			\left\|\int_{\mathbb{R}^D} f(x,v) dv\right\|_{B^s_{p,q}\left(dx\right)}\leq
			 C \left(\left\|f\right\|_{\widetilde L^1\left(dx;B^\alpha_{p,q}\left(dv\right)\right)}+ \left\|g\right\|_{\widetilde L^1\left(dx;B^\beta_{p,q}\left(dv\right)\right)}\right),
		\end{gathered}
	\end{equation}
	where the constant $C>0$ only depends on fixed parameters.
	
	If $\beta + D \left(1-\frac1p\right)  > 1$ or, if $\beta + D \left(1-\frac1p\right) = 1$ and $q=1$, then
	\begin{equation}
		\int_{\mathbb{R}^D} f(x,v) dv \in B^{s}_{p,\infty}\left(\mathbb{R}^D_x\right),
	\end{equation}
	where $s = 1- D \left( 1 - \frac 1 p \right) $, and the following estimate holds
	\begin{equation}\label{pre averaging lemma p.2}
		\begin{gathered}
			\left\|\int_{\mathbb{R}^D} f(x,v) dv\right\|_{B^{s}_{p,\infty}\left(dx\right)}\leq
			 C \left(\left\|\Delta_0^{x,v}f\right\|_{L^1\left(dx;L^p\left(dv\right)\right)}+ \left\|g\right\|_{\widetilde L^1\left(dx;B^\beta_{p,q}\left(dv\right)\right)}\right),
		\end{gathered}
	\end{equation}
	where the constant $C>0$ only depends on fixed parameters.
\end{thm}

Notice that, with the sole bound $f(x,v) \in \widetilde L^1\left(\mathbb{R}^D_x;B^\alpha_{p,q}\left(\mathbb{R}^D_v\right)\right)$, and in particular without any information on the transport equation \eqref{pre averaging lemma p.0}, it is only possible to deduce, by Sobolev embedding, that the velocity average satisfies, for any $\phi(v)\in C_0^\infty\left(\mathbb{R}^D\right)$,
\begin{equation}
	\int_{\mathbb{R}^D} f(x,v)\phi(v) dv \in L^1\left(\mathbb{R}^D_x\right)\subset B_{p,\infty}^{-D\left(1-\frac 1p\right)}\left(\mathbb{R}^D_x\right).
\end{equation}
Therefore, in the above theorem, the gain on the velocity average can be measured by the difference between the regularity index $s$ in \eqref{pre averaging lemma p.1} and \eqref{pre averaging lemma p.2}, and the regularity index $-D\left(1-\frac 1p\right)$ obtained by Sobolev embedding. Thus, as long as $\beta + D \left(1-\frac1p\right)  <1$, the above theorem yields a gain of regularity of $\frac{\alpha+D \left( 1 - \frac 1 p \right)}{1+\alpha-\beta}$.

Note that this gain approaches one full derivative as $\beta$ tends to $1-D \left(1-\frac1p\right)$, or as $\alpha$ tends to infinity. However, since the transport operator is a differential operator of order one, it can never yield a net gain of regularity greater than one. This is precisely the reason why, when $\beta + D \left(1-\frac1p\right)  \geq 1$, the averaging lemma saturates and only yields a maximal gain of one full derivative, independently of $\alpha$.

Very loosely speaking, this result shows that the transport operator $v\cdot\nabla_x$ is fully invertible when $g$ is very regular in velocity. Thus, it becomes an elliptic operator through velocity averaging. This is also the reason why, quite remarkably, only the low frequencies of $f$ are involved in this case.

Notice also that the condition $\alpha>-D\left(1-\frac 1p\right)$ above is very natural, since otherwise the gain $\frac{\alpha+D \left( 1 - \frac 1 p \right)}{1+\alpha-\beta}$ becomes negative and thus the averaging lemma turns out to be weaker than the Sobolev embedding.

However, the condition $D\left(1-\frac 1p\right)<1$ seems less natural. Actually, its necessity comes from the handling of the low velocity frequencies of $g(x,v)$ (it can be interpreted as $\beta + D \left(1-\frac1p\right)  <1$ with $\beta=0$ for those low frequencies). Thus, it is possible to remove this condition by considering a corresponding version of Theorem \ref{pre averaging lemma p} for homogeneous Besov spaces, which is the content of Theorem \ref{pre averaging lemma p h} below.

The next result extends the preceding theorem to the case including spatial derivatives in the right-hand side.

\begin{thm}\label{pre averaging lemma p 2}
	Let $f(x,v) \in B^{a,\alpha}_{1,p,q}\left(\mathbb{R}^D_x\times\mathbb{R}^D_v\right)$, where $1\leq p, q \leq \infty$, $a\in\mathbb{R}$ and $\alpha> - D \left( 1-\frac{1}{p} \right)  > -1$, be such that
	\begin{equation}
		v \cdot \nabla_x f = g
	\end{equation}
	for some $g(x,v) \in B^{b,\beta}_{1,p,q}\left(\mathbb{R}^D_x\times\mathbb{R}^D_v\right)$, where $\beta\in\mathbb{R}$ and $b\geq a-1$.

	If $\beta + D \left(1-\frac1p\right)  <1$, then
	\begin{equation}
		\int_{\mathbb{R}^D} f(x,v) dv \in B^s_{p,q}\left(\mathbb{R}^D_x\right),
	\end{equation}
	where $ s = \left(1+b-a\right)\frac{\alpha+ D \left( 1 - \frac 1 p \right)}{1+\alpha-\beta} + a - D \left( 1 - \frac 1 p \right) $, and the following estimate holds
	\begin{equation}
		\begin{gathered}
			\left\|\int_{\mathbb{R}^D} f(x,v) dv\right\|_{B^s_{p,q}\left(dx\right)}\leq
			C \left( \left\| f \right\|_{B^{a,\alpha}_{1,p,q}\left(\mathbb{R}^D_x\times\mathbb{R}^D_v\right)}
			+
			\left\| g \right\|_{B^{b,\beta}_{1,p,q}\left(\mathbb{R}^D_x\times\mathbb{R}^D_v\right)} \right),
		\end{gathered}
	\end{equation}
	where the constant $C>0$ only depends on fixed parameters.
	
	If $\beta + D \left(1-\frac1p\right) > 1$ or, if $\beta + D \left(1-\frac1p\right) = 1$ and $q=1$, then
	\begin{equation}
		\int_{\mathbb{R}^D} f(x,v) dv \in B^s_{p,q}\left(\mathbb{R}^D_x\right),
	\end{equation}
	where $ s = \left(1+b-a\right) + a - D \left( 1 - \frac 1 p \right) $, and the following estimate holds
	\begin{equation}
		\begin{gathered}
			\left\|\int_{\mathbb{R}^D} f(x,v) dv\right\|_{B^s_{p,q}\left(dx\right)}\leq
			C \left( \left\|\Delta_0^{x,v}f\right\|_{L^1\left(dx;L^p\left(dv\right)\right)}
			+
			\left\| g \right\|_{B^{b,\beta}_{1,p,q}\left(\mathbb{R}^D_x\times\mathbb{R}^D_v\right)} \right),
		\end{gathered}
	\end{equation}
	where the constant $C>0$ only depends on fixed parameters.
	
	If $\beta + D \left(1-\frac1p\right) = 1$ and $q\neq 1$, then, for every $\epsilon>0$,
	\begin{equation}
		\int_{\mathbb{R}^D} f(x,v) dv \in B^{s-\epsilon}_{p,q}\left(\mathbb{R}^D_x\right),
	\end{equation}
	where $ s = \left(1+b-a\right) + a - D \left( 1 - \frac 1 p \right) $, and the following estimate holds
	\begin{equation}
		\begin{gathered}
			\left\|\int_{\mathbb{R}^D} f(x,v) dv\right\|_{B^{s-\epsilon}_{p,q}\left(dx\right)}\leq
			C \left( \left\| f \right\|_{B^{a,\alpha}_{1,p,q}\left(\mathbb{R}^D_x\times\mathbb{R}^D_v\right)}
			+
			\left\| g \right\|_{B^{b,\beta}_{1,p,q}\left(\mathbb{R}^D_x\times\mathbb{R}^D_v\right)} \right),
		\end{gathered}
	\end{equation}
	where the constant $C>0$ only depends on fixed parameters, in particular on $\epsilon>0$.
\end{thm}

The remarks formulated above about Theorem \ref{pre averaging lemma p} are still valid here regarding Theorem \ref{pre averaging lemma p 2}.

Thus, as long as $\beta + D \left(1-\frac1p\right)  <1$, the above theorem yields a gain of regularity of $(1+b-a)\frac{\alpha+D \left( 1 - \frac 1 p \right)}{1+\alpha-\beta}$ (compared to the Sobolev embedding). Notice that this gain approaches a derivative of order $(1+b-a)$ as $\beta$ tends to $1-D \left(1-\frac1p\right)$, or as $\alpha$ tends to infinity, which is optimal for a differential operator of order one. Therefore, the averaging lemma saturates beyond the value $\beta + D \left(1-\frac1p\right) = 1$ and only yields at most a gain of $1+b-a$ derivatives.

We do not know whether it is possible to achieve a full gain of $1+b-a$ derivatives in the case $\beta + D \left(1-\frac1p\right) = 1$ and $q\neq 1$. Nevertheless, the cases $\beta + D \left(1-\frac1p\right) > 1$ or $\beta + D \left(1-\frac1p\right) = 1$ and $q=1$ do yield an exact full gain of $1+b-a$ derivatives, independently of $\alpha$, which, again, is largely optimal. Quite remarkably, this is the first time that a velocity averaging result achieves exactly the maximal gain of regularity, i.e. one full derivative in the case $a=b=0$, say. Moreover, it is worth noting that only the low frequencies of $f$ are involved in this case, which, very loosely speaking, shows that the transport operator $v\cdot\nabla_x$ is fully invertible when $g$ is very regular in velocity.

Notice also that the conditions $\alpha>-D\left(1-\frac 1p\right)$ and $b\geq a-1$ above are very natural, since otherwise the gain $(1+b-a)\frac{\alpha+D \left( 1 - \frac 1 p \right)}{1+\alpha-\beta}$ becomes negative and thus the averaging lemma turns out to be weaker than the Sobolev embedding.

Finally, as for Theorem \ref{pre averaging lemma p}, the condition $D\left(1-\frac 1p\right)<1$ seems less natural. Actually, its necessity comes from the handling of the low velocity frequencies of $g(x,v)$ (it can be interpreted as $\beta + D \left(1-\frac1p\right)  <1$ with $\beta=0$ for those low frequencies). Thus, it is possible to remove this condition by considering a corresponding version of Theorem \ref{pre averaging lemma p 2} for homogeneous Besov spaces, which is the content of Theorem \ref{pre averaging lemma p 2 h} below.


\subsection{Velocity averaging in $L^1_xL^p_v$, homogeneous case}

This section contains the results in the endpoint case $L^1_xL^p_v$ formulated with homogeneous Besov spaces and corresponding to Theorems \ref{pre averaging lemma p} and \ref{pre averaging lemma p 2}.

\begin{thm}\label{pre averaging lemma p h}
	Let $f(x,v) \in \widetilde L^1\left(\mathbb{R}^D_x;\dot B^\alpha_{p,q}\left(\mathbb{R}^D_v\right)\right)$, where $1\leq p, q \leq \infty$ and $\alpha>-D\left(1-\frac 1p\right) $, be such that
	\begin{equation}
		v \cdot \nabla_x f = g
	\end{equation}
	for some $g(x,v) \in \widetilde L^1\left(\mathbb{R}^D_x;\dot B^\beta_{p,q}\left(\mathbb{R}^D_v\right)\right)$, where $\beta <  1-D \left(1-\frac1p\right)$.
	
	Then,
	\begin{equation}
		\int_{\mathbb{R}^D} f(x,v) dv \in \dot B^s_{p,q}\left(\mathbb{R}^D_x\right),
	\end{equation}
	where $s = \frac{\alpha+D \left( 1 - \frac 1 p \right)}{1+\alpha-\beta}  - D \left( 1 - \frac 1 p \right) $, and the following estimate holds
	\begin{equation}
		\begin{gathered}
			\left\|\int_{\mathbb{R}^D} f(x,v) dv\right\|_{\dot B^s_{p,q}\left(dx\right)}\leq
			 C \left\|f\right\|_{\widetilde L^1\left(dx;\dot B^\alpha_{p,q}\left(dv\right)\right)}^{\frac{1-\beta-D\left(1-\frac 1p\right)}{1+\alpha-\beta}}
			\times
			\left\|g\right\|_{\widetilde L^1\left(dx;\dot B^\beta_{p,q}\left(dv\right)\right)}^{\frac{\alpha+D\left(1-\frac 1p\right)}{1+\alpha-\beta}},
		\end{gathered}
	\end{equation}
	where the constant $C>0$ only depends on fixed parameters.
\end{thm}

\begin{thm}\label{pre averaging lemma p 2 h}
	Let $f(x,v) \in \dot B^{a,\alpha}_{1,p,q}\left(\mathbb{R}^D_x\times\mathbb{R}^D_v\right)$, where $1\leq p, q \leq \infty$, $a\in\mathbb{R}$ and $\alpha> - D \left( 1-\frac{1}{p} \right)$, be such that
	\begin{equation}
		v \cdot \nabla_x f = g
	\end{equation}
	for some $g(x,v) \in \dot B^{b,\beta}_{1,p,q}\left(\mathbb{R}^D_x\times\mathbb{R}^D_v\right)$, where $b\geq a-1$ and $\beta <  1-D \left(1-\frac1p\right)$.

	Then,
	\begin{equation}
		\int_{\mathbb{R}^D} f(x,v) dv \in \dot B^s_{p,q}\left(\mathbb{R}^D_x\right),
	\end{equation}
	where $ s = \left(1+b-a\right)\frac{\alpha+ D \left( 1 - \frac 1 p \right)}{1+\alpha-\beta} + a - D \left( 1 - \frac 1 p \right) $, and the following estimate holds
	\begin{equation}
		\left\|\int_{\mathbb{R}^D} f(x,v) dv\right\|_{\dot B^s_{p,q}\left(dx\right)}\leq
		C
		\left\| f \right\|_{\dot B^{a,\alpha}_{1,p,q}\left(\mathbb{R}^D_x\times\mathbb{R}^D_v\right)}^{\frac{1-\beta-D\left(1-\frac 1p\right)}{1+\alpha-\beta}}
		\times
		\left\| g \right\|_{\dot B^{b,\beta}_{1,p,q}\left(\mathbb{R}^D_x\times\mathbb{R}^D_v\right)}^{\frac{\alpha+D\left(1-\frac 1p\right)}{1+\alpha-\beta}},
	\end{equation}
	where the constant $C>0$ only depends on fixed parameters.
	
	Furthermore, if $\beta + D \left(1-\frac1p\right) = 1$ and $q=1$, then the above estimate remains valid.
\end{thm}


\subsection{The classical $L^2_xL^2_v$ case revisited}

We give now a new very general version of the classical velocity averaging lemma in $L^2_xL^2_v$ in terms of Besov spaces. The proofs of this formulation have the advantage of employing the same principles and ideas as in the $L^1_xL^p_v$ cases. In particular, they are based on the same decompositions and operators (provided by Proposition \ref{crucial} below), which will be crucial in order to carry out interpolation arguments between the $L^1_xL^p_v$ and $L^2_xL^2_v$ cases later on.

Note that, as usual, exploiting the Hilbertian structure of Besov spaces (i.e. choosing $q=2$), the theorem below can readily be reformulated in terms of more standard Sobolev spaces.

\begin{thm}\label{classical}
	Let $f \in B_{2,2,q}^{a,\alpha}\left(\mathbb{R}^D_x\times\mathbb{R}^D_v\right)$, where $a\in\mathbb{R}$, $\alpha> -\frac 12$ and $1\leq q\leq \infty$, be such that
	\begin{equation}
		v\cdot \nabla_x f= g
	\end{equation}
	for some $g\in B_{2,2,q}^{b,\beta}\left(\mathbb{R}^D_x\times\mathbb{R}^D_v\right)$, where $b\geq a-1$ and $\beta\in\mathbb{R}$.
	
	If $\beta < \frac 12$, then, for any $\phi\in C_0^\infty\left(\mathbb{R}^D\right)$,
	\begin{equation}
		\int_{\mathbb{R}^D} f(x,v)\phi(v) dv \in B^s_{2,q}\left(\mathbb{R}^D_x\right),
	\end{equation}
	where $ s = (1+b-a) \frac{\alpha+\frac 12}{1+\alpha-\beta} + a $, and the following estimate holds
	\begin{equation}
		\left\|\int_{\mathbb{R}^D} f(x,v)\phi(v) dv\right\|_{B^s_{2,q}\left(\mathbb{R}^D_x\right)}\leq
		C_\phi \left( \left\| f \right\|_{B^{a,\alpha}_{2,2,q}\left(\mathbb{R}^D_x\times\mathbb{R}^D_v\right)}
		+
		\left\| g \right\|_{B^{b,\beta}_{2,2,q}\left(\mathbb{R}^D_x\times\mathbb{R}^D_v\right)} \right),
	\end{equation}
	where the constant $C_\phi>0$ only depends on $\phi$ and other fixed parameters.

	If $\beta > \frac 12$ or, if $\beta=\frac 12$ and $q=1$, then, for any $\phi\in C_0^\infty\left(\mathbb{R}^D\right)$,
	\begin{equation}
		\int_{\mathbb{R}^D} f(x,v)\phi(v) dv \in B^s_{2,q}\left(\mathbb{R}^D_x\right),
	\end{equation}
	where $ s = 1+b $, and the following estimate holds
	\begin{equation}
		\left\|\int_{\mathbb{R}^D} f(x,v)\phi(v) dv\right\|_{B^s_{2,q}\left(\mathbb{R}^D_x\right)}\leq
		C_\phi \left( \left\| \Delta_0^{x,v}f \right\|_{L^2\left(\mathbb{R}^D_x\times\mathbb{R}^D_v\right)}
		+
		\left\| g \right\|_{B^{b,\beta}_{2,2,q}\left(\mathbb{R}^D_x\times\mathbb{R}^D_v\right)} \right),
	\end{equation}
	where the constant $C_\phi>0$ only depends on $\phi$ and other fixed parameters.

	If $\beta=\frac 12$ and $q\neq 1$, then, for any $\phi\in C_0^\infty\left(\mathbb{R}^D\right)$ and every $\epsilon>0$,
	\begin{equation}
		\int_{\mathbb{R}^D} f(x,v)\phi(v) dv \in B^{s-\epsilon}_{2,q}\left(\mathbb{R}^D_x\right),
	\end{equation}
	where $ s = 1+b $, and the following estimate holds
	\begin{equation}
		\left\|\int_{\mathbb{R}^D} f(x,v)\phi(v) dv\right\|_{B^{s-\epsilon}_{2,q}\left(\mathbb{R}^D_x\right)}\leq
		C_\phi \left( \left\| f \right\|_{B^{a,\alpha}_{2,2,q}\left(\mathbb{R}^D_x\times\mathbb{R}^D_v\right)}
		+
		\left\| g \right\|_{B^{b,\beta}_{2,2,q}\left(\mathbb{R}^D_x\times\mathbb{R}^D_v\right)} \right),
	\end{equation}
	where the constant $C_\phi>0$ only depends on $\phi$ and other fixed parameters, in particular on $\epsilon>0$.
\end{thm}

Notice that the theorem above provides a net gain of regularity of $(1+b-a) \frac{\alpha+\frac 12}{1+\alpha-\gamma}$ derivatives, where $\gamma=\min\left\{\beta,\frac 12\right\}$. Therefore, the restrictions $\alpha> -\frac 12$ and $1+b-a\geq 0$ on the parameters are, in fact, quite natural since the gain of regularity would possibly be negative otherwise.

Furthermore, the threshold at the value $\beta=\frac 12$ stems from the fact that, since the transport operator is a differential operator of order one, it cannot yield, in the case $a=b=0$ say, a gain of regularity which would be superior to one full derivative. In other words, necessarily $\frac{\alpha+\frac 12}{1+\alpha-\gamma}\leq 1$, which implies $\gamma\leq \frac 12$.

Quite remarkably, as for Theorem \ref{pre averaging lemma p 2}, the above theorem does achieve the maximal gain of regularity of $1+b-a$ derivatives in the cases $\beta>\frac 12$ or $\beta=\frac 12$ and $q=1$, independently of $\alpha$, which is unprecedented. Moreover, it is worth noting that only the low frequencies of $f$ are involved in this case, which, very loosely speaking, shows that the transport operator $v\cdot\nabla_x$ is fully invertible when $g$ is very regular in velocity.


\subsection{The $L_x^1L_v^p$ and $L_x^2L_v^2$ cases reconciled}

The following theorem results from a simple interpolation between Theorems \ref{pre averaging lemma p 2} and \ref{classical}. A more general, but far more complicated, interpolation procedure will yield the more general Theorem \ref{main averaging lemma 2} below.

\begin{thm}\label{main averaging lemma}
	Let $f(x,v) \in B^{a,\alpha}_{r,p,q}\left(\mathbb{R}^D_x\times\mathbb{R}^D_v\right)$, where $1\leq r\leq p\leq\infty$, $1\leq q<\infty$, $a\in\mathbb{R}$ and $\alpha> \frac 1r -1 - D \left( \frac 1r-\frac{1}{p} \right)  > -\frac 1r$, be such that
	\begin{equation}
		v \cdot \nabla_x f = g
	\end{equation}
	for some $g(x,v) \in B^{b,\beta}_{r,p,q}\left(\mathbb{R}^D_x\times\mathbb{R}^D_v\right)$, where $\beta\in\mathbb{R}$ and $b\geq a-1$.

	If $\beta< \frac 1r - D \left(\frac 1r-\frac1p\right) $, then, for any $\phi\in C_0^\infty\left(\mathbb{R}^D\right)$,
	\begin{equation}
		\int_{\mathbb{R}^D} f(x,v)\phi(v) dv \in B^s_{p,q}\left(\mathbb{R}^D_x\right),
	\end{equation}
	where $ s = \left(1+b-a\right) \frac{1+\alpha -\left(\frac 1r - D \left( \frac 1r - \frac 1 p \right)\right)}{1+\alpha-\beta} +a - D \left( \frac 1r - \frac 1 p \right) $, and the following estimate holds
	\begin{equation}
		\begin{gathered}
			\left\|\int_{\mathbb{R}^D} f(x,v)\phi(v) dv\right\|_{B^s_{p,q}\left(dx\right)}\leq
			C_\phi \left( \left\| f \right\|_{B^{a,\alpha}_{r,p,q}\left(\mathbb{R}^D_x\times\mathbb{R}^D_v\right)}
			+
			\left\| g \right\|_{B^{b,\beta}_{r,p,q}\left(\mathbb{R}^D_x\times\mathbb{R}^D_v\right)} \right),
		\end{gathered}
	\end{equation}
	where the constant $C_\phi>0$ only depends on $\phi$ and other fixed parameters.
	
	If $\beta> \frac 1r - D \left(\frac 1r-\frac1p\right) $ or, if $\beta= \frac 1r - D \left(\frac 1r-\frac1p\right) $ and $q=1$, then, for any $\phi\in C_0^\infty\left(\mathbb{R}^D\right)$,
	\begin{equation}
		\int_{\mathbb{R}^D} f(x,v)\phi(v) dv \in B^s_{p,q}\left(\mathbb{R}^D_x\right),
	\end{equation}
	where $ s = 1+b - D \left( \frac 1r - \frac 1 p \right) $, and the following estimate holds
	\begin{equation}
		\begin{gathered}
			\left\|\int_{\mathbb{R}^D} f(x,v)\phi(v) dv\right\|_{B^s_{p,q}\left(dx\right)}\leq
			C_\phi \left( \left\| \Delta_0^{x,v} f \right\|_{B^{a,\alpha}_{r,p,q}\left(\mathbb{R}^D_x\times\mathbb{R}^D_v\right)}
			+
			\left\| g \right\|_{B^{b,\beta}_{r,p,q}\left(\mathbb{R}^D_x\times\mathbb{R}^D_v\right)} \right),
		\end{gathered}
	\end{equation}
	where the constant $C_\phi>0$ only depends on $\phi$ and other fixed parameters.
	
	If $\beta= \frac 1r - D \left(\frac 1r-\frac1p\right) $ and $q\neq1$, then, for any $\phi\in C_0^\infty\left(\mathbb{R}^D\right)$ and every $\epsilon>0$,
	\begin{equation}
		\int_{\mathbb{R}^D} f(x,v)\phi(v) dv \in B^{s-\epsilon}_{p,q}\left(\mathbb{R}^D_x\right),
	\end{equation}
	where $ s = 1+b - D \left( \frac 1r - \frac 1 p \right) $, and the following estimate holds
	\begin{equation}
		\begin{gathered}
			\left\|\int_{\mathbb{R}^D} f(x,v)\phi(v) dv\right\|_{B^{s-\epsilon}_{p,q}\left(dx\right)}\leq
			C_\phi \left( \left\| f \right\|_{B^{a,\alpha}_{r,p,q}\left(\mathbb{R}^D_x\times\mathbb{R}^D_v\right)}
			+
			\left\| g \right\|_{B^{b,\beta}_{r,p,q}\left(\mathbb{R}^D_x\times\mathbb{R}^D_v\right)} \right),
		\end{gathered}
	\end{equation}
	where the constant $C_\phi>0$ only depends on $\phi$ and other fixed parameters, in particular on $\epsilon>0$.
\end{thm}

Notice that the theorem above corresponds exactly to Theorems \ref{pre averaging lemma p 2} and \ref{classical} in the limiting cases $r=1$ and $r=2$, respectively. It provides a net gain of regularity, compared to the Sobolev embedding, of $\left(1+b-a\right) \frac{1+\alpha -\left(\frac 1r - D \left( \frac 1r - \frac 1 p \right)\right)}{1+\alpha-\gamma}$ derivatives, where $\gamma=\left\{\beta, \frac 1r - D \left(\frac 1r-\frac1p\right) \right\}$. Therefore, the restrictions $\alpha> \frac 1r -1 - D \left( \frac 1r-\frac{1}{p} \right)$ and $1+b-a\geq 0$ on the parameters are, in fact, quite natural since the gain of regularity would possibly be negative otherwise.

Furthermore, the threshold at the value $\beta= \frac 1r - D \left(\frac 1r-\frac1p\right) $ stems from the fact that, since the transport operator is a differential operator of order one, it cannot yield, in the case $a=b=0$ say, a gain of regularity which would be superior to one full derivative. In other words, necessarily $\frac{1+\alpha -\left(\frac 1r - D \left( \frac 1r - \frac 1 p \right)\right)}{1+\alpha-\gamma}\leq 1$, which implies $\gamma\leq \frac 1r - D \left(\frac 1r-\frac1p\right)$.

Quite remarkably, as for Theorems \ref{pre averaging lemma p 2} and \ref{classical}, the above theorem does achieve the maximal gain of regularity of $1+b-a$ derivatives in the cases $\beta>\frac 1r - D \left(\frac 1r-\frac1p\right)$ or $\beta=\frac 1r - D \left(\frac 1r-\frac1p\right)$ and $q=1$, independently of $\alpha$, which is unprecedented. Moreover, it is worth noting that only the low frequencies of $f$ are involved in this case, which, very loosely speaking, shows that the transport operator $v\cdot\nabla_x$ is fully invertible when $g$ is very regular in velocity.

The following theorem is the most general result presented in this work. However, it does not contain all the previous theorems. It follows from a general abstract interpolation procedure of the preceding results.

\begin{thm}\label{main averaging lemma 2}
	Let $f(x,v) \in B^{a,\alpha}_{r_0,p_0,q_0}\left(\mathbb{R}^D_x\times\mathbb{R}^D_v\right)$, where $1\leq r_0\leq p_0 \leq r_0' \leq \infty$, $1\leq q_0 < \infty$, $a\in\mathbb{R}$ and $\alpha>\frac{1}{r_0}-1-D\left(\frac{1}{r_0}-\frac{1}{p_0}\right)$, be such that
	\begin{equation}
		v \cdot \nabla_x f = g
	\end{equation}
	for some $g(x,v) \in B^{b,\beta}_{r_1,p_1,q_1}\left(\mathbb{R}^D_x\times\mathbb{R}^D_v\right)$, where $1\leq r_1\leq p_1 \leq r_1' \leq \infty$, $1\leq q_1 < \infty$, $b\in\mathbb{R}$ and $\beta<\frac{1}{r_1}-D\left(\frac{1}{r_1}-\frac{1}{p_1}\right)$ satisfy
	\begin{equation}
		\frac{2}{r_1}-1-D\left(\frac{1}{r_1}-\frac{1}{p_1}\right)>0
		\qquad\text{or}\qquad
		p_1=r_1=2
	\end{equation}
	and
	\begin{equation}
		(1-\theta)\frac{1}{p_0}+\theta\frac{1}{p_1} =(1-\theta)\frac{1}{q_0}+\theta\frac{1}{q_1},
	\end{equation}
	where
	\begin{equation}
		\theta
		= \frac{\left[\alpha + 1 -\frac{1}{r_0} + D\left(\frac{1}{r_0}-\frac{1}{p_0}\right)\right]}
		{\left[\alpha + 1 -\frac{1}{r_0} + D\left(\frac{1}{r_0}-\frac{1}{p_0}\right)\right]
		+\left[
		-\beta+\frac{1}{r_1}-D\left(\frac{1}{r_1}-\frac{1}{p_1}\right)\right]}\in (0,1).
	\end{equation}

	Then, for any $\chi,\phi\in C_0^\infty\left(\mathbb{R}^D\right)$,
	\begin{equation}
		\int_{\mathbb{R}^D} f(x,v)\chi(x)\phi(v) dv \in B^s_{p,p}\left(\mathbb{R}^D_x\right),
	\end{equation}
	where
	\begin{equation}
		\begin{aligned}
			s & = (1-\theta)\left(a-D\left(\frac 1{r_0}-\frac 1{p_0}\right)\right)
			+\theta\left(b-D\left(\frac 1{r_1}-\frac 1{p_1}\right)\right) + \theta, \\
			\frac{1}{p} & = (1-\theta)\frac{1}{p_0}+\theta\frac{1}{p_1} = (1-\theta)\frac{1}{q_0}+\theta\frac{1}{q_1},
		\end{aligned}
	\end{equation}
	and the following estimate holds
	\begin{equation}\label{general estimate}
		\begin{aligned}
			& \left\|\int_{\mathbb{R}^D} f(x,v)\chi(x)\phi(v) dv\right\|_{B^s_{p,p}\left(dx\right)} \\
			& \leq
			C_\phi \left( \left\| f \right\|_{B^{a,\alpha}_{r_0,p_0,q_0}\left(\mathbb{R}^D_x\times\mathbb{R}^D_v\right)}
			+
			\left\| g \right\|_{B^{b,\beta}_{r_1,p_1,q_1}\left(\mathbb{R}^D_x\times\mathbb{R}^D_v\right)} \right),
		\end{aligned}
	\end{equation}
	where the constant $C_\phi>0$ only depends on $\phi$ and other fixed parameters.
	
	Furthermore, if $p\geq r_0$, the space localization through a cutoff $\chi(x)$ is not necessary and it is possible to take $\chi(x)\equiv 1$ in the above statements.
\end{thm}

The interpretation of net gain of regularity is not as straightforward as it is for the preceding theorems. Thus, we provide now a somewhat alternative analysis of the regularity index $s$.

The above theorem essentially establishes an estimate on the velocity average which stems from an interpolation of order $\theta$ between the controls
\begin{equation}
	\int_{\mathbb{R}^D} f(x,v)\phi(v) dv \in B^{a}_{r_0,q_0}\left(\mathbb{R}^D_x\right)
	\subset B^{a-D\left(\frac{1}{r_0}-\frac{1}{p_0}\right)}_{p_0,q_0}\left(\mathbb{R}^D_x\right)
\end{equation}
and
\begin{equation}
	\int_{\mathbb{R}^D} g(x,v)\phi(v) dv \in B^{b}_{r_1,q_1}\left(\mathbb{R}^D_x\right)
	\subset B^{b-D\left(\frac{1}{r_1}-\frac{1}{p_1}\right)}_{p_1,q_1}\left(\mathbb{R}^D_x\right),
\end{equation}
which follow from standard Sobolev embeddings. If the functions $f(x,v)$ and $g(x,v)$ were linked by a relation $f=Tg$, where $T$ is some bounded and invertible operator of differential order $r\in\mathbb{R}$ acting only on $x$, then it would be natural to expect, by interpolation of order $\theta$, a control on the velocity average in the Besov space
\begin{equation}
	B^{(1-\theta)\left(a-D\left(\frac 1{r_0}-\frac 1{p_0}\right)\right)
	+\theta\left(b-D\left(\frac 1{r_1}-\frac 1{p_1}\right)+r\right)}_{p,p}\left(\mathbb{R}^D_x\right),
\end{equation}
where $\frac{1}{p} = (1-\theta)\frac{1}{p_0}+\theta\frac{1}{p_1} = (1-\theta)\frac{1}{q_0}+\theta\frac{1}{q_1}$. This formal reasoning shows that the net gain of regularity given by a differential relation $f=Tg$ of order $r$ through an interpolation of order $\theta$ is at most $\theta r$, when compared to a differential operator of zero (e.g. the identity).

From that viewpoint, the above theorem asserts that, the transport operator $T=v\cdot\nabla_x$ being a differential operator of order one, it is possible to obtain a maximal net gain of regularity $\theta$ through velocity averaging. We insist that here the net gain of regularity is found by comparing the actual regularity index $s$ with the interpolation of the indices obtained by Sobolev embeddings.

Therefore, the restriction $\alpha>\frac{1}{r_0}-1-D\left(\frac{1}{r_0}-\frac{1}{p_0}\right)$ on the parameters is, in fact, quite natural since the gain of regularity $\theta$ would possibly be negative otherwise.

Furthermore, the constraint $\beta<\frac{1}{r_1}-D\left(\frac{1}{r_1}-\frac{1}{p_1}\right)$ stems from the fact that, since the transport operator is a differential operator of order one, it cannot yield a gain of regularity $\theta$ which would be superior to one full derivative. In other words, necessarily $\theta< 1$, which implies $\beta< \frac{1}{r_1}-D\left(\frac{1}{r_1}-\frac{1}{p_1}\right)$.

It is quite interesting to note that some kind of space localization is definitely necessary in the case $p<r_0$ of the above theorem. It is in fact explicit from the proofs that this restriction comes from the control of low frequencies. Indeed, let us suppose that the estimate \eqref{general estimate} holds with $\chi\equiv 1$ for some given choice of parameters in the one dimensional case $D=1$. Further consider $f(x,v)\in \mathcal{S}\left(\mathbb{R}_x\times\mathbb{R}_v\right)$ such that its space and velocity frequencies are localized in a bounded domain. In other words, we suppose that $\Delta_{2^k}^xf=\Delta_{2^k}^vf=0$, for every $k\geq 0$, say. Therefore, in virtue of estimate \eqref{general estimate}, it holds that
\begin{equation}
		\left\|\int_{\mathbb{R}^D} f(x,v)\phi(v) dv\right\|_{L^p_x}\leq
		C_\phi \left( \left\| f \right\|_{L^{r_0}_xL^{p_0}_v}
		+
		\left\| v\partial_x f \right\|_{L^{r_1}_xL^{p_1}_v} \right).
\end{equation}
In particular, since the transformation $f_R(x,v)=f\left(\frac{x}{R},v\right)$ preserves de localization of low frequencies for any $R>1$, we deduce that it must also hold that
\begin{equation}
	\begin{aligned}
		R^{\frac 1p} \left\|\int_{\mathbb{R}^D} f(x,v)\phi(v) dv\right\|_{L^p_x}
		& = \left\|\int_{\mathbb{R}^D} f_R(x,v)\phi(v) dv\right\|_{L^p_x} \\
		& \leq
		C_\phi \left( \left\| f_R \right\|_{L^{r_0}_xL^{p_0}_v}
		+
		\left\| v\partial_x f_R \right\|_{L^{r_1}_xL^{p_1}_v} \right) \\
		& =
		C_\phi \left( R^{\frac{1}{r_0}}\left\| f \right\|_{L^{r_0}_xL^{p_0}_v}
		+ R^{\frac{1}{r_1}-1}
		\left\| v\partial_{x} f \right\|_{L^{r_1}_xL^{p_1}_v}
		\right).
	\end{aligned}
\end{equation}
It follows that $R^{\frac 1p-\frac 1{r_0}}$ must remain bounded as $R$ tends towards infinity, which forces $r_0\leq p$.

Finally, we would like to emphasize that we have chosen to present, in this work, cases of velocity averaging lemmas dealing with $L^r_xL^p_v$ integrability only, where $r\leq p$, principally because our analysis of dispersion allowed to handle these previously unsettled cases. But we insist that this is by no means a restriction of our method, which is, in fact, very robust and enables to also treat the actually easier setting of $L^p_vL_x^r$ integrability, where $p\leq r$, and thus, to recover most of previously known results in sharper Besov spaces. Our precise interpolation techniques can even reach settings where the left-hand side $f$ enjoys $L^{r_0}_xL^{p_0}_v$ integrability, where $r_0\leq p_0$, while the right-hand side $g$ displays $L^{p_1}_vL_x^{r_1}$ integrability, where $p_1\leq r_1$, and vice versa.


\section{Ellipticity, dispersion and averaging}\label{methods}

Here, we explain the concepts which will lead to the proofs of the main results in this work.

The classical theory of velocity averaging lemmas in $L^2_{x,v}$, first developed in \cite{GLPS88,GPS85}, is based on a simple but ingenious microlocal decomposition. More precisely, if $f(x,v),g(x,v)\in L^2\left(\mathbb{R}^D_x\times\mathbb{R}^D_v\right) $ satisfy the transport relation \eqref{transport 1} then, considering the Fourier transforms $\hat f(\eta,v)$ and $\hat g(\eta,v)$ in the space variable only, it holds that
\begin{equation}
	i v\cdot\eta \hat f(\eta,v) = \hat g(\eta,v).
\end{equation}
Therefore, it is possible to exploit some ellipticity of the transport operator as long as one remains on an appropriate microlocal domain. In other words, we may invert the transport operator as long as the quantity $\left|v\cdot\eta\right|$ remains uniformly bounded away from zero:
\begin{equation}
	\hat f(\eta,v) = \frac{1}{i v\cdot\eta} \hat g(\eta,v),\qquad \text{on $\left\{|v\cdot\eta|>1\right\}$, say.}
\end{equation}
Thus, introducing some cutoff function $\rho\in\mathcal{S}\left(\mathbb{R}\right)$ ($\mathcal{S}$ denotes the Schwartz space of rapidly decaying functions) such that $\rho(0)=1$ and an interpolation parameter $t>0$, we may decompose
\begin{equation}\label{elliptic}
	\hat f(\eta,v)= \rho\left(t v\cdot\eta\right) \hat f(\eta,v)
	+ \frac{1 - \rho\left(t v\cdot\eta\right)}{i v\cdot\eta} \hat g(\eta,v).
\end{equation}
It follows that each term in the right-hand side may then be estimated locally in $L^2_v$ and the remainder of the proof simply consists in choosing the optimal value for the interpolation parameter $t$ (which will depend on $\eta$). The conclusion of this method yields that, for every test function $\phi(v)\in C^\infty_0$, the velocity average $\int f(x,v)\phi(v)dv$ belongs to $H^\frac{1}{2}_x$. This approach yields optimal results and exhibits the crucial regularizing properties of the transport operator, which, we insist, is based on exploiting some partial ellipticity.

As mentioned before, several extensions of this method are possible (cf. \cite{bezard, diperna2}), in particular, the $L^p_{x,v}$ case of velocity averaging lemmas is obtained by interpolating the preceding $L^2_{x,v}$ result with the degenerate case in $L^1_{x,v}$. Indeed, if $f(x,v),g(x,v)\in L^1_{x,v}$, then absolutely no regularity may be gained on the velocity averages from the transport equation, which is unfortunately optimal as far as the gain of regularity is concerned.

In this work, we obtain refined velocity averaging results by further exploiting the dispersive properties of the transport operator discovered by Castella and Perthame in \cite{CP96}. This requires the development of a suitable interpolation formula, more refined than \eqref{elliptic}, and the study of its properties. Thus, introducing an interpolation parameter $t>0$, it trivially holds, from \eqref{transport 1}, that
\begin{equation}
	\begin{cases}
		(\partial_t+v\cdot\nabla_x)f=g,\\
		f(t=0)=f.
	\end{cases}
\end{equation}
Hence the interpolation formula, 
\begin{equation}\label{interpolation formula}
	f(x,v)=f(x-tv,v)+\int_0^tg(x-sv,v) ds,
\end{equation}
which is in fact dual to the interpolation formula employed in \cite{golse3} and is merely Duhamel's representation formula for the time dependent transport equation.

Furthermore, considering any $f\in\mathcal{S}\left(\mathbb{R}^D_x\times\mathbb{R}^D_v\right)$, $\chi\in\mathcal{S}\left(\mathbb{R}^D\right)$ and denoting $\chi_\lambda(\cdot)=\frac{1}{\lambda^D}\chi\left(\frac{\cdot}{\lambda}\right)$, where $\lambda>0$, one easily verifies the following rule of action of convolutions on velocity averages,
\begin{equation}
	\begin{aligned}
		\chi_\lambda*_x\int_{\mathbb{R}^D} f(x-tv,v) dv & =
		\int_{\mathbb{R}^D\times\mathbb{R}^D}\frac{1}{\lambda^D}\chi\left(\frac{x-y}{\lambda}\right)f(y-tv,v) dvdy\\
		& = \int_{\mathbb{R}^D\times\mathbb{R}^D} \frac{1}{\left(\frac{\lambda}{t}\right)^D}
		\chi\left(\frac{w-v}{\frac{\lambda}{t}}\right)f(x-tw,v) dwdv\\
		& = \int_{\mathbb{R}^D}\left(\chi_{\frac{\lambda}{t}}*_vf\right)(x-tw,w) dw,
	\end{aligned}
\end{equation}
where we used the change of variables $\left(v,y\right) \mapsto \left(v,w = \frac{x-y}t + v\right) $.

In the notation of Section \ref{LP decomposition} and in particular \eqref{dyadic block} and \eqref{dyadic block 2}, one then checks employing the above rule of action of convolutions that the dyadic frequency blocks act on velocity averages according to the identities, where $\delta>0$,
\begin{equation}  \label{x-to-v}
	\begin{aligned}
		\Delta_0^x\int f(x-tv,v) dv & =
		\int \left(S_{t}^vf\right)(x-tv,v) dv,\\
		\Delta_\delta^x\int f(x-tv,v) dv & =
		\int \left(\Delta_{t\delta}^vf\right)(x-tv,v) dv.
	\end{aligned}
\end{equation}
Next, we apply this identity to the velocity averages of the above interpolation formula \eqref{interpolation formula} to deduce
\begin{equation} \label{D-x}
	\begin{aligned}
		\Delta_{0}^x\int_{\mathbb{R}^D} f(x,v) dv & =
		\int_{\mathbb{R}^D} \left(S_{t}^v f\right)(x-tv,v) dv \\
		& +
		\int_0^t\int_{\mathbb{R}^D} \left(S_{s}^v g\right)(x-sv,v) dvds,\\
		\Delta_{2^k}^x\int_{\mathbb{R}^D} f(x,v) dv & =
		\int_{\mathbb{R}^D} \left(\Delta_{t2^k}^v f\right)(x-tv,v) dv \\
		& +
		\int_0^t\int_{\mathbb{R}^D} \left(\Delta_{s2^k}^v g\right)(x-sv,v) dvds.
	\end{aligned}
\end{equation}

The above representation formula can be used explicitly to establish the simplest forms of our main results. It is the first key idea in our approach. Indeed, it shows that the space frequencies of the averages of $f(x,v)$ may be controlled with the velocity frequencies of $f(x,v)$ and $g(x,v)$. This property of transfer of frequencies is linked to the hypoellipticity of the transport operator and we refer to \cite{arsenio2} for recent developments on this matter.

Moreover, it can be used in some case to give a sense to the velocity average $\int_{\mathbb{R}^D}f(x,v)dv$ even when $f(x,v)$ is a priori not globally integrable in velocity. Indeed, supposing that we can show (using dispersive estimates for instance) that $\left(\Delta_{t2^k}^v f\right)(x-tv,v)$ and $\int_0^t \left(\Delta_{s2^k}^v g\right)(x-sv,v) ds$ actually are globally integrable in velocity for each given $t>0$, then we may nevertheless define the velocity average $\Delta_{2^k}^x\int_{\mathbb{R}^D} f(x,v) dv$ by the identities \eqref{D-x}, and similarly for the low frequencies component. This principle is used implicitly in the statements of Theorems \ref{pre averaging lemma p}, \ref{pre averaging lemma p 2}, \ref{pre averaging lemma p h} and \ref{pre averaging lemma p 2 h}, but we will never render this argument explicit for the sake of simplicity.

A very simple but useful refinement of the above formulas \eqref{D-x} follows from the fact that $S_2^x \Delta_0^x=\Delta_0^x$ and $\Delta_{\left[2^{k-1},2^{k+1}\right]}^x \Delta_{2^k}^x = \Delta_{2^k}^x$. Thus, we obtain
\begin{equation} \label{D-x 2}
	\begin{aligned}
		\Delta_{0}^x\int_{\mathbb{R}^D} f(x,v) dv & =
		\int_{\mathbb{R}^D} \left(S_2^xS_t^v f\right)(x-tv,v) dv \\
		& +
		\int_0^t\int_{\mathbb{R}^D} \left(S_2^xS_s^v g\right)(x-sv,v) dvds,\\
		\Delta_{2^k}^x\int_{\mathbb{R}^D} f(x,v) dv & =
		\int_{\mathbb{R}^D} \left(\Delta_{\left[2^{k-1},2^{k+1}\right]}^x\Delta_{t2^k}^v f\right)(x-tv,v) dv \\
		& +
		\int_0^t\int_{\mathbb{R}^D} \left(\Delta_{\left[2^{k-1},2^{k+1}\right]}^x\Delta_{s2^k}^v g\right)(x-sv,v) dvds,
	\end{aligned}
\end{equation}
which considerably decreases the set of frequencies of $f(x,v)$ and $g(x,v)$ required to control the frequencies of the velocity averages.

The formulas \eqref{D-x} and \eqref{D-x 2} will be used to treat the velocity averages in the $L^1_xL^p_v$ and the $L^1_vL^p_x$ settings, which are endpoint cases. As usual, the more general cases will then be obtained by interpolation with the classical $L_{x,v}^2$ case of velocity averaging. However, a significant obstruction to this interpolating strategy lies in that the representation formulas \eqref{elliptic}, which exploits the elliptic properties of the transport operator, and \eqref{interpolation formula}, which is based on the dispersion of the transport operator, are of different nature and thus seem at first to be incompatible. There is however an elementary but crucial link between them which we establish now. Similar ideas are used in \cite{arsenio2}.

To this end, we first notice, recalling the Fourier inversion formula $\rho(r)=\frac{1}{2\pi}\int_\mathbb{R} e^{i rs} \hat \rho(s) ds$, that it trivially holds, for any $t\in\mathbb{R}$, that
\begin{equation}
	\hat f\left(\eta,v\right) \rho\left(t \eta\cdot v\right)
	=\frac{1}{2\pi}
	\int_\mathbb{R} \hat f\left(\eta,v\right) e^{i st \eta\cdot v} \hat \rho(s) ds,
\end{equation}
and similarly, since $1=\rho(0)=\frac{1}{2\pi}\int_\mathbb{R} \hat \rho(s) ds$ and further noticing
\begin{equation}
	-\int_0^{st} e^{i \eta\cdot v \sigma} d\sigma
	=
	\frac{1-e^{i st \eta\cdot v}}{i  \eta\cdot v},
\end{equation}
that
\begin{equation}
	\begin{aligned}
		\hat g\left(\eta,v\right) \frac{1-\rho\left(t \eta\cdot v\right)}{i\eta\cdot v}
		=&\frac{1}{2\pi}
		\int_\mathbb{R} \hat g\left(\eta,v\right) \frac{ 1 - e^{i st \eta\cdot v} }{i\eta\cdot v} \hat \rho(s) ds\\
		=&
		- \frac{1}{2\pi} \int_\mathbb{R} \int_0^{st} \hat g\left(\eta,v\right) e^{i \eta\cdot v \sigma} d\sigma \hat \rho(s) ds.
	\end{aligned}
\end{equation}
Then, simply taking the inverse Fourier transform of the above identities, we obtain
\begin{equation}\label{equiv f}
	\begin{aligned}
		\mathcal{F}_x^{-1}
		\rho\left(t \eta\cdot v\right)
		\mathcal{F}_x f\left(x,v\right)
		=&\frac{1}{2\pi}
		\int_\mathbb{R}
		\mathcal{F}_x^{-1}
		e^{i st \eta\cdot v}
		\mathcal{F}_x
		f\left(x,v\right) \hat \rho(s) ds\\
		=&\frac{1}{2\pi}
		\int_\mathbb{R}
		f\left(x+stv,v\right) \hat \rho(s) ds,
	\end{aligned}
\end{equation}
and
\begin{equation}\label{equiv g}
	\begin{aligned}
		\mathcal{F}_x^{-1}
		\frac{1-\rho\left(t \eta\cdot v\right)}{i \eta\cdot v}
		\mathcal{F}_x
		g\left(x,v\right)
		=&
		-\frac{1}{2\pi}
		\int_\mathbb{R} \int_0^{st}
		\mathcal{F}_x^{-1}
		e^{i \sigma \eta\cdot v }
		\mathcal{F}_x g\left(x,v\right)  d\sigma \hat \rho(s) ds\\
		=&
		- \frac{1}{2\pi} \int_\mathbb{R} \int_0^{st}
		g\left(x+\sigma v,v\right)  d\sigma \hat \rho(s) ds.
	\end{aligned}
\end{equation}

Therefore, incorporating identities \eqref{equiv f} and \eqref{equiv g} into the interpolation formulas \eqref{elliptic} or \eqref{interpolation formula}, we obtain, for each $t\in\mathbb{R}$, the following refined decomposition
\begin{equation}\label{refined interpolation formula}
	f(x,v)=T_A^t f(x,v) + t T_B^t g(x,v),
\end{equation}
where
\begin{equation}\label{T 1}
	\begin{aligned}
		T_A^tf(x,v)&=\frac{1}{2\pi}\int_\mathbb{R}
		f\left(x+stv,v\right) \hat \rho(s) ds\\
		&=\int_\mathbb{R}
		f\left(x-stv,v\right) \tilde \rho(s) ds,
		\\
		T_B^tg(x,v)&=- \frac 1t \frac{1}{2\pi}\int_\mathbb{R} \int_0^{st}
		g\left(x+\sigma v,v\right)  d\sigma \hat \rho(s) ds\\
		& =-\frac{1}{2\pi} \int_\mathbb{R} \int_0^{s}
		g\left(x+\sigma t v,v\right)  d\sigma \hat \rho(s) ds\\
		& = \int_\mathbb{R} \int_0^{s}
		g\left(x-\sigma t v,v\right)  d\sigma \tilde \rho(s) ds,
	\end{aligned}
\end{equation}
and
\begin{equation}\label{T 2}
	\begin{aligned}
		\mathcal{F}_xT_A^tf(\eta,v)&=
		\rho\left(t \eta\cdot v\right)
		\mathcal{F}_x f\left(\eta,v\right),
		\\
		\mathcal{F}_xT_B^tg(\eta,v)&=
		\frac{1-\rho\left(t \eta\cdot v\right)}{ i t \eta\cdot v}
		\mathcal{F}_x
		g\left(\eta,v\right),
	\end{aligned}
\end{equation}
which shows that formulas \eqref{elliptic} and \eqref{interpolation formula} are in fact equivalent. Indeed, it is readily seen that \eqref{elliptic} follows from \eqref{refined interpolation formula} by use of the Fourier transform, while \eqref{interpolation formula} is deduced from \eqref{refined interpolation formula} by setting $\tilde\rho(s)$ equal to an arbitrary approximation of the Dirac mass at $s=1$, which makes sense since we have merely imposed on the cutoff that $\rho(0)=\frac{1}{2\pi}\int_\mathbb{R}\hat\rho(s)ds=\int_\mathbb{R}\tilde\rho(s)ds=1$. Equivalently, formula \eqref{refined interpolation formula} can be derived from \eqref{interpolation formula} by replacing $t$ by $st$ and then integrating against $\tilde\rho(s)ds$.

Notice also that, by setting $\tilde \rho(s)=e^{-s}\mathbb{1}_{\left\{s\geq 0\right\}}$ in \eqref{refined interpolation formula}, one arrives at the standard representation formula, for any $t>0$,
\begin{equation}
	\begin{aligned}
		f(x,v) & =\int_0^\infty
		f\left(x-stv,v\right) e^{-s}  + t  \int_0^{s}
		g\left(x-\sigma t v,v\right)  d\sigma e^{-s} ds\\
		& =\int_0^\infty
		\left[\frac{1}{t}f\left(x-sv,v\right)  +
		g\left(x-s v,v\right)\right]  e^{-\frac st} ds,
	\end{aligned}
\end{equation}
which is usually obtained by directly solving the equivalent transport equation
\begin{equation}
	\left[\frac 1t +v\cdot\nabla_x\right] f= \frac 1t f+g,
\end{equation}
but we will not make any use of this decomposition (cf. \cite{golse3, jabin} for uses of this formula). Essentially, this particular choice of cutoff is not appropriate because its frequencies are not well localized, even though its Fourier transform decays exponentially. The importance of having a strong frequencies cutoff is made fully explicit in the localization identities \eqref{crucial 3} of Proposition \ref{crucial} below.

Much more importantly, this refined decomposition \eqref{refined interpolation formula} shows that the elliptic and dispersive properties of the transport operator may be exploited through the same interpolation formula. Thus, employing formula \eqref{x-to-v} on the transfer of frequencies for velocity averages with the above decomposition, we obtain the following crucial proposition, which will systematically be the starting point of the proofs for the averaging lemmas presented in this work.

\begin{prop}\label{crucial}
	Let $f(x,v),g(x,v)\in\mathcal{S}\left(\mathbb{R}^D\times\mathbb{R}^D\right)$ be such that
	\begin{equation}
		v\cdot\nabla_x f = g.
	\end{equation}
	For all $t>0$, $\delta\geq 0$ and for every cutoff function $\rho\in\mathcal{S}\left(\mathbb{R}\right)$ such that $\rho(0)=\frac{1}{2\pi}\int_\mathbb{R}\hat\rho(s)ds=\int_\mathbb{R}\tilde\rho(s)ds=1$, we consider the decomposition
	\begin{equation}\label{crucial 0}
		\Delta_{\delta}^x\int_{\mathbb{R}^D} f(x,v) dv = A_{\delta}^tf(x) + t B_{\delta}^tg(x),
	\end{equation}
	where the operators $A_\delta^t$ and $B_\delta^t$ are defined by
	\begin{equation}
		\begin{aligned}
			A_\delta^t f(x) &= \Delta_{\delta}^x \int_{\mathbb{R}^D} T_A^t f(x,v) dv,\\
			B_\delta^t g(x) &= \Delta_{\delta}^x \int_{\mathbb{R}^D} T_B^t g(x,v) dv.
		\end{aligned}
	\end{equation}
	
	Then it holds that
	\begin{equation}\label{crucial 1}
		\begin{aligned}
			A_{\delta}^t f(x)&=
			\int_{\mathbb{R}^D}
			\mathcal{F}_x^{-1}
			\rho\left(t \eta\cdot v\right)
			\mathcal{F}_x
			\Delta_{\delta}^x
			f(x,v)dv,\\
			B_{\delta}^t g(x)&=
			\int_{\mathbb{R}^D}
			\mathcal{F}_x^{-1}
			\tau\left(t \eta\cdot v\right)
			\mathcal{F}_x
			\Delta_{\delta}^x
			g(x,v)dv,
		\end{aligned}
	\end{equation}
	where $\tau(s)=\frac{1-\rho\left( s \right)}{i s}$ is smooth, and
	\begin{equation}\label{crucial 2}
		\begin{aligned}
			A_{\delta}^t f(x)
			& = \int_\mathbb{R} \left[ \int_{\mathbb{R}^D}
			\Delta_{\delta}^x f \left(x-stv,v\right) dv \right] \tilde \rho(s) ds , \\
			B_{\delta}^t g(x)
			& = \int_\mathbb{R} \left[ \int_0^1 \int_{\mathbb{R}^D}
			\Delta_{\delta}^x g \left(x-\sigma st v,v\right)  dv d\sigma \right] s\tilde \rho(s) ds.
		\end{aligned}
	\end{equation}
	
	Furthermore, if the cutoff $\rho\in\mathcal{S}\left(\mathbb{R}\right)$ is such that $\tilde\rho$ is compactly supported inside $\left[1,2\right]$, then, for every $\delta>0$,
	\begin{equation}\label{crucial 3}
		\begin{aligned}
			A_{0}^t f
			& = A_{0}^t \left( S_2^xS_{4t}^v f \right), &
			B_{0}^t g
			& = B_{0}^t \left(S_2^xS_{4t}^v g\right),\\
			A_{\delta}^t f
			& = A_{\delta}^t \left( \Delta_{\left[\frac \delta 2,2\delta\right]}^x\Delta_{\left[{t\delta\over 2},4t\delta\right]}^v f \right), &
			B_{\delta}^t g
			& = B_{\delta}^t \left(\Delta_{\left[\frac \delta 2,2\delta\right]}^xS_{8t \delta}^v g\right).
		\end{aligned}
	\end{equation}
\end{prop}

\begin{proof}
	It is readily seen that identities \eqref{crucial 1} and \eqref{crucial 2} are a mere transcription in terms of velocity averages of properties \eqref{T 1} and \eqref{T 2} for the operators $T_A^t$ and $T_B^t$. Therefore, we only have to justify the crucial property \eqref{crucial 3}. To this end, notice first that property \eqref{x-to-v} on the transfer of spatial to velocity frequencies implies, according to \eqref{crucial 2}, that, when $\delta>0$,
	\begin{equation}
		\begin{aligned}
			A_{0}^t f(x)
			& = \int_\mathbb{R} \left[ \int_{\mathbb{R}^D}
			\left(S_{st}^v f\right) \left(x-stv,v\right) dv \right] \tilde \rho(s) ds , \\
			B_{0}^t g(x)
			& = \int_\mathbb{R} \left[ \int_0^1 \int_{\mathbb{R}^D}
			\left( S_{\sigma st}^v g \right) \left(x-\sigma st v,v\right)  dv d\sigma \right] s\tilde \rho(s) ds,\\
			A_{\delta}^t f(x)
			& = \int_\mathbb{R} \left[ \int_{\mathbb{R}^D}
			\left(\Delta_{st\delta}^v f\right) \left(x-stv,v\right) dv \right] \tilde \rho(s) ds , \\
			B_{\delta}^t g(x)
			& = \int_\mathbb{R} \left[ \int_0^1 \int_{\mathbb{R}^D}
			\left( \Delta_{\sigma st\delta}^v g \right) \left(x-\sigma st v,v\right)  dv d\sigma \right] s\tilde \rho(s) ds.
		\end{aligned}
	\end{equation}
	
	Then, since $t\leq st\leq 2t$ and $\sigma st\leq 2t$ on the support of $\tilde\rho(s)$, it holds that
	\begin{equation}
		\begin{aligned}
			S_{st}^v & = S_{st}^v S_{4t}^v, &
			S_{\sigma st}^v & = S_{\sigma st}^v S_{4t}^v,\\
			\Delta_{st\delta}^v & = \Delta_{st\delta}^v \Delta_{\left[{t\delta\over 2},4t\delta\right]}^v, &
			\Delta_{\sigma st\delta}^v & = \Delta_{\sigma st\delta}^v S_{8t \delta}^v.
		\end{aligned}
	\end{equation}
	Hence, employing identities \eqref{x-to-v} again, we find that
	\begin{equation}
		\begin{aligned}
			A_{0}^t f(x)
			& = \int_\mathbb{R} \left[ \int_{\mathbb{R}^D}
			\left(S_{st}^v S_{4t}^v f\right) \left(x-stv,v\right) dv \right] \tilde \rho(s) ds \\
			& = \int_\mathbb{R} \left[ \int_{\mathbb{R}^D}
			\left(\Delta_{0}^x S_{4t}^v f\right) \left(x-stv,v\right) dv \right] \tilde \rho(s) ds = A_{0}^t \left( S_{4t}^v f \right)(x) , \\
			B_{0}^t g(x)
			& = \int_\mathbb{R} \left[ \int_0^1 \int_{\mathbb{R}^D}
			\left( S_{\sigma st}^v S_{4t}^v g \right) \left(x-\sigma st v,v\right)  dv d\sigma \right] s\tilde \rho(s) ds\\
			& = \int_\mathbb{R} \left[ \int_0^1 \int_{\mathbb{R}^D}
			\left( \Delta_{0}^x S_{4t}^v g \right) \left(x-\sigma st v,v\right)  dv d\sigma \right] s\tilde \rho(s) ds = B_{0}^t \left(S_{4t}^v g\right)(x),\\
			A_{\delta}^t f(x)
			& = \int_\mathbb{R} \left[ \int_{\mathbb{R}^D}
			\left(\Delta_{st\delta}^v \Delta_{\left[{t\delta\over 2},4t\delta\right]}^v f\right) \left(x-stv,v\right) dv \right] \tilde \rho(s) ds \\
			& = \int_\mathbb{R} \left[ \int_{\mathbb{R}^D}
			\left(\Delta_{\delta}^x \Delta_{\left[{t\delta\over 2},4t\delta\right]}^v f\right) \left(x-stv,v\right) dv \right] \tilde \rho(s) ds = A_{\delta}^t \left( \Delta_{\left[{t\delta\over 2},4t\delta\right]}^v f \right)(x) , \\
			B_{\delta}^t g(x)
			& = \int_\mathbb{R} \left[ \int_0^1 \int_{\mathbb{R}^D}
			\left( \Delta_{\sigma st\delta}^v S_{8t \delta}^v g \right) \left(x-\sigma st v,v\right)  dv d\sigma \right] s\tilde \rho(s) ds\\
			& = \int_\mathbb{R} \left[ \int_0^1 \int_{\mathbb{R}^D}
			\left( \Delta_{\delta}^x S_{8t \delta}^v g \right) \left(x-\sigma st v,v\right)  dv d\sigma \right] s\tilde \rho(s) ds = B_{\delta}^t \left(S_{8t \delta}^v g\right)(x).
		\end{aligned}
	\end{equation}
	Finally, utilizing that $S_2^x \Delta_0^x=\Delta_0^x$ and $\Delta_{\left[\frac \delta 2,2\delta\right]}^x \Delta_{\delta}^x = \Delta_{\delta}^x$ concludes the proof of the proposition.
\end{proof}


\section{Control of concentrations in $L^1$}\label{L1}

The methods developed in Section \ref{methods}, which eventually led to the main results from Section \ref{main results}, were initially motivated by an important application to the viscous incompressible hydrodynamic limit of the Boltzmann equation with long-range interactions in \cite{arsenio, arsenio3}. We present in this section some technical lemmas which were crucially employed therein and are contained in our main results. Essentially, we show here how the methods from Section \ref{methods} can be used to gain control over the concentrations in kinetic transport equations through velocity averaging. This application relies on a refined regularity estimate for the Boltzmann equation without cutoff established in \cite{arsenio4}, wherein the reader will also find a complete discussion of this application to the hydrodynamic limit. An alternative method has been developed in \cite{arsenio2}.

The difficult problem of concentrations in the viscous incompressible hydrodynamic limit of the Boltzmann equation was first successfully handled in \cite{golse} for bounded collision kernels with cutoff and then, using the same approach, extended in \cite{golse5, levermore} to more general cross-sections with cutoff. In any case, the method consisted in an application of a simple, but subtle, velocity averaging lemma in $L^1$ established in \cite{golse3}, which we briefly present now. Recall first the following definition from \cite{golse3}.

\begin{defi}

	A bounded sequence $\left\{f_n(x,v)\right\}_{n=0}^\infty\subset L^1_{\textrm{loc}}\left(\mathbb{R}^D_x\times\mathbb{R}^D_v\right)$ is said to be \textbf{equi-integrable in $v$} if and only if for every $\eta>0$ and each compact $K\subset\mathbb{R}^D\times\mathbb{R}^D$, there exists $\delta>0$ such that for any measurable set $\Omega\subset\mathbb{R}^D\times\mathbb{R}^D$ with $\sup_{x\in\mathbb{R}^D}\int\mathbb{1}_\Omega(x,v) dv<\delta$, we have that
	\begin{equation}
		\int_{\Omega\cap K}|f_n(x,v)| dxdv<\eta \quad\text{for every $n$.}
	\end{equation}

\end{defi}

The crucial compactness lemma in $L^1$ obtained in \cite{golse3} and relevant to the rigorous derivations of the hydrodynamic limit in the cutoff case \cite{golse, golse5, levermore} is contained in the following theorem.

\begin{thm}[\cite{golse3}]\label{velocity averaging golse}

	Let the sequence $\{f_n(x,v)\}_{n=0}^\infty$ be bounded in $L^1_\mathrm{loc}\left(\mathbb{R}^D_x\times\mathbb{R}^D_v\right)$, equi-integrable in $v$, and such that
	\begin{equation}\label{kinetic}
		v\cdot\nabla_xf_n=g_n
	\end{equation}
	for some sequence $\{g_n\}_{n=0}^\infty$ bounded in $L^1_\mathrm{loc}\left(\mathbb{R}^D_x\times\mathbb{R}^D_v\right)$.
	
	Then,
	\begin{enumerate}

		\item $\{f_n\}_{n=0}^\infty$ is equi-integrable (in all variables $x$ and $v$) and,

		\item for each $\psi\in L^\infty(\mathbb{R}^D)$ with compact support, the family of velocity averages
			\begin{equation}
				\int_{\mathbb{R}^D} f_n(x,v)\psi(v) dv\qquad\text{is relatively compact in }L^1_{\mathrm{loc}}(\mathbb{R}^D).
			\end{equation}

	\end{enumerate}

\end{thm}

Unfortunately, the above result is not directly applicable to the viscous incompressible hydrodynamic limit of the Boltzmann equation with long-range interactions. Indeed, its use is prevented by the particular structure of the Boltzmann collision operator, which behaves as a nonlinear differential operator for collision kernels without cutoff (cf. \cite{arsenio, arsenio3, arsenio4}). Consequently, the strategy from \cite{golse, golse5, levermore} would require a generalization of the preceding theorem to kinetic transport equations \eqref{kinetic} with velocity derivatives in its right-hand side.

However, a simple counterexample shows that a straight generalization won't be possible. Indeed, consider the locally integrable functions
\begin{equation}
	\begin{aligned}
		f_n(x,v) & =n\varphi(nx_1)\cos(nv_1),\\
		g_n(x,v) & =v_1n\varphi'(nx_1)\sin(nv_1)\\
		\text{and}\quad h_n(x,v) & =n\varphi'(nx_1)\sin(nv_1),
	\end{aligned}
\end{equation}
where $\varphi\in C_0^\infty(\mathbb{R})$ has non-trivial mass. Then, one easily checks that
\begin{equation}
	\begin{gathered}
		v\cdot\nabla_xf_n=\partial_{v_1}g_n-h_n\\
		\text{and}\quad \{f_n\}_{n=0}^\infty \text{ is equi-integrable in }v.
	\end{gathered}
\end{equation}
However, $\{f_n\}_{n=0}^\infty$ is not equi-integrable in all variables, which shows that the first assertion of Theorem \ref{velocity averaging golse} doesn't hold if one has derivatives on the right-hand side of the transport equation \eqref{kinetic}.

Nevertheless, note that the velocity averages of $f_n$ do converge strongly to zero and thus, it is unclear whether there holds a direct extension of the second assertion of Theorem \ref{velocity averaging golse} or not. Furthermore, this counterexample uses crucially the fact that $f_n$ is allowed to alternate sign and so, it might still be possible that preventing this oscillatory behavior by simply imposing a nonnegativity condition on the $f_n$'s would allow for a generalization of the result.

We wish now to extend the above Theorem \ref{velocity averaging golse} in order to allow the use of derivatives in the right-hand side of the transport equation, which, again, is necessary for the rigorous derivation of the hydrodynamic limit in the non-cutoff case. It turns out that a slightly better control on the high velocity frequencies of the $f_n$'s will suffice to provide a relevant generalization. This is the content of the coming theorem.

\begin{thm}\label{averaging lemma}
	
	Let $\{f_n(x,v)\}_{n=0}^\infty$ be a bounded sequence in $L^1\left(\mathbb{R}^D_x;B^0_{1,1}\left(\mathbb{R}^D_v\right)\right)$, such that $f_n\geq0$ and
	\begin{equation}
		v\cdot\nabla_xf_n=\left(1-\Delta_{v}\right)^\frac{\gamma}{2}g_n
	\end{equation}
	for some sequence $\{g_n(x,v)\}_{n=0}^\infty$ bounded in $L^1\left(\mathbb{R}^D_x\times \mathbb{R}^D_v\right)$ and some $\gamma\in\mathbb{R}$.
	
	Then,
	\begin{equation}
		\text{$\{f_n\}_{n=0}^\infty$ is equi-integrable (in all variables $x$ and $v$).}
	\end{equation}
	
\end{thm}

The above theorem will be obtained as a corollary of the following proposition, which is, in fact, a particular case of Theorem \ref{pre averaging lemma p}. Nevertheless, we do provide below a complete justification for this result, because its proof is simpler and contains some of the essential ideas presented in Section \ref{methods}. Thus, we hope that it will also serve as a primer to the more convoluted demonstrations of Section \ref{proofs}.

\begin{prop}\label{pre averaging lemma}
	Let $f(x,v) \in  L^1\left(\mathbb{R}^D_x;B^0_{1,1}\left(\mathbb{R}^D_v\right)\right)$ be such that
	\begin{equation}
		v \cdot \nabla_x f = g
	\end{equation}
	for some $g(x,v) \in  L^1\left(\mathbb{R}^D_x;B^\beta_{1,1}\left(\mathbb{R}^D_v\right)\right)$, where $\beta\in\mathbb{R}$.
	
	Then,
	\begin{equation}
		\int_{\mathbb{R}^D} f(x,v) dv \in B^0_{1,1}\left(\mathbb{R}^D_x\right),
	\end{equation}
	and the following estimate holds
	\begin{equation}
		\begin{gathered}
			\left\|\int_{\mathbb{R}^D} f(x,v) dv\right\|_{B^0_{1,1}\left(dx\right)}\leq
			 C \left(\left\|f\right\|_{ L^1\left(dx;B^0_{1,1}\left(dv\right)\right)}+ \left\|g\right\|_{ L^1\left(dx;B^\beta_{1,1}\left(dv\right)\right)}\right),
		\end{gathered}
	\end{equation}
	where the constant $C>0$ only depends on fixed parameters.
\end{prop}

\begin{proof}
	Notice first that, in virtue of the inclusions $\widetilde L^1B^{\beta_1}_{1,1} \subset \widetilde L^1B^{\beta_2}_{1,1}$ whenever $\beta_1\geq\beta_2$, we may assume without any loss of generality that $\beta<1$. Furthermore, it is to be emphasized that $\widetilde L^1B^{\beta}_{1,1} = L^1B^{\beta}_{1,1}$, for any $\beta\in\mathbb{R}$.
	
	Integrating the dyadic interpolation formula \eqref{D-x} in space yields
	\begin{equation}
		\left\|\Delta_{2^k}^x\int_{\mathbb{R}^D} f(x,v) dv\right\|_{L^1_x}\leq \left\|\Delta_{t2^k}^v f\right\|_{L^1_{x,v}}
		+\int_0^t\left\|\Delta_{s2^k}^v g\right\|_{L^1_{x,v}}ds.
	\end{equation}
	For each value of $k$, we fix now the interpolation parameter $t$ as $t_k=2^{k\frac{\beta}{1-\beta}}$. Thus, summing the above estimate over $k$ yields
	\begin{equation}\label{besov averaging 1}
		\sum_{k=0}^\infty \left\|\Delta_{2^k}^x\int_{\mathbb{R}^D} f(x,v) dv\right\|_{L^1_x}
		\leq
		\sum_{k=0}^\infty \left\|\Delta_{t_k2^k}^v f\right\|_{L^1_{x,v}}
		+ \sum_{k=0}^\infty \int_0^{t_k}\left\|\Delta_{s2^k}^v g\right\|_{L^1_{x,v}}ds.
	\end{equation}
	
	The last term above is then handled through the following calculation, noticing that $\frac{t_{k-1}}{2}<t_k$,
	\begin{equation}
		\begin{aligned}
			\sum_{k=0}^\infty \int_0^{t_k}\left\|\Delta_{s2^k}^v g\right\|_{L^1_{x,v}}ds
			& =
			\sum_{k=0}^\infty \int_0^{\frac{t_{k-1}}{2}}\left\|\Delta_{s2^k}^v g\right\|_{L^1_{x,v}}ds
			+ \int_{\frac{t_{k-1}}{2}}^{t_k}\left\|\Delta_{s2^k}^v g\right\|_{L^1_{x,v}}ds
			\\
			& =
			\sum_{k=0}^\infty \frac 12 \int_0^{t_{k-1}}\left\|\Delta_{s2^{k-1}}^v g\right\|_{L^1_{x,v}}ds
			+ \int_{\frac{t_{k-1}}{2}}^{t_k}\left\|\Delta_{s2^k}^v g\right\|_{L^1_{x,v}}ds
			\\
			& =
			\frac 12 \sum_{k=-1}^\infty \int_0^{t_{k}}\left\|\Delta_{s2^{k}}^v g\right\|_{L^1_{x,v}}ds
			+ \sum_{k=0}^\infty \int_{\frac{t_{k-1}}{2}}^{t_k}\left\|\Delta_{s2^k}^v g\right\|_{L^1_{x,v}}ds,
		\end{aligned}
	\end{equation}
	from which we deduce
	\begin{equation}\label{besov averaging 2}
		\sum_{k=0}^\infty \int_0^{t_k}\left\|\Delta_{s2^k}^v g\right\|_{L^1_{x,v}}ds
		=
		\int_0^{t_{(-1)}}\left\|\Delta_{\frac s2}^v g\right\|_{L^1_{x,v}}ds
		+ 2 \sum_{k=0}^\infty \int_{\frac{t_{k-1}}{2}}^{t_k}\left\|\Delta_{s2^k}^v g\right\|_{L^1_{x,v}}ds.
	\end{equation}
	Consequently, combining \eqref{besov averaging 1} with \eqref{besov averaging 2}, we infer
	\begin{equation}\label{besov averaging 3}
		\begin{aligned}
			& \sum_{k=0}^\infty \left\|\Delta_{2^k}^x\int_{\mathbb{R}^D} f(x,v) dv\right\|_{L^1_x}\\
			& \leq
			\sum_{k=0}^\infty \left\|\Delta_{t_k2^k}^v f\right\|_{L^1_{x,v}}
			+  2 \sum_{k=0}^\infty \int_{\frac{t_{k-1}}{2}}^{t_k}\left\|\Delta_{s2^k}^v g\right\|_{L^1_{x,v}}ds
			+C\|g\|_{\widetilde L^1\left(dx;B^\beta_{1,1}\left(dv\right)\right)},
		\end{aligned}
	\end{equation}
	where $C>0$ is a fixed constant that only depends on fixed parameters.

	Now, let us recall a basic principle from the theory of Littlewood and Paley. Consider any $\delta>0$. Then, there exists a unique integer $l_\delta$ such that $2^{l_\delta}\leq \delta < 2^{l_\delta+1}$. In the notation of section \ref{LP decomposition} and specially \eqref{dyadic block}, it then holds that $\sum_{j=-1}^{2}\varphi_{2^{l_\delta+j}}\equiv 1$ on the support of $\varphi_\delta$, which therefore implies for any $f\in\mathcal{S}'\left(\mathbb{R}^D\right)$ and any $1\leq p\leq \infty$ that
	\begin{equation}
		\left\|\Delta_\delta f\right\|_{L^p}=\left\|\sum_{j=-1}^2\Delta_{2^{l_\delta+j}}\Delta_\delta f\right\|_{L^p}\leq
		C\sum_{j=-1}^2\left\|\Delta_{2^{l_\delta+j}}f\right\|_{L^p},
	\end{equation}
	for some fixed constant $C>0$ independent of $\delta$.
	
	We wish now to apply this principle to our situation. Thus, denoting by $[\cdot]$ the integer part of a number and noticing that $2^{\left[k\frac{1}{1-\beta}\right]}\leq t_k2^k=2^{k\frac{1}{1-\beta}}<2^{\left[k\frac{1}{1-\beta}\right]+1}$, we obtain, on the one hand, that
	\begin{equation}\label{besov averaging 4}
		\sum_{k=0}^\infty \left\|\Delta_{t_k2^k}^v f\right\|_{L^1_{x,v}}
		\leq C \sum _{j=-1}^2
		\sum_{k=0}^\infty \left\|\Delta_{2^{\left[k\frac{1}{1-\beta}\right]+j}}^v f\right\|_{L^1_{x,v}}
		\leq C \left\|f\right\|_{\widetilde L^1\left(dx;B^0_{1,1}\left(dv\right)\right)}.
	\end{equation}
	And, on the other hand, we have that
	\begin{equation}\label{besov averaging 5}
		\begin{aligned}
			\sum_{k=0}^\infty \int_{\frac{t_{k-1}}{2}}^{t_k}\left\|\Delta_{s2^k}^v g\right\|_{L^1_{x,v}}ds
			& \leq
			\sum_{j=-\left[\frac{1}{1-\beta}\right]-2}^2
			\sum_{k=0}^\infty \int_{\frac{t_{k-1}}{2}}^{t_k}\left\|
			\Delta_{2^{\left[k\frac{1}{1-\beta}\right]+j}}^v
			\Delta_{s2^k}^v g\right\|_{L^1_{x,v}}ds \\
			& \leq C
			\sum_{j=-\left[\frac{1}{1-\beta}\right]-2}^2
			\sum_{k=0}^\infty \left(t_k-\frac{t_{k-1}}{2}\right)
			\left\|
			\Delta_{2^{\left[k\frac{1}{1-\beta}\right]+j}}^v
			g\right\|_{L^1_{x,v}} \\
			& \leq C
			\sum_{j=-\left[\frac{1}{1-\beta}\right]-2}^2
			\sum_{k=0}^\infty
			\left(2^{\left[k\frac{1}{1-\beta}\right]+j}\right)^\beta
			\left\|
			\Delta_{2^{\left[k\frac{1}{1-\beta}\right]+j}}^v
			g\right\|_{L^1_{x,v}} \\
			& \leq C \left\|g\right\|_{\widetilde L^1\left(dx ; B^\beta_{1,1}\left(dv\right)\right)}.
		\end{aligned}
	\end{equation}

	Finally, combining \eqref{besov averaging 3} with \eqref{besov averaging 4} and \eqref{besov averaging 5} yields
	\begin{equation}
		\left\|\int_{\mathbb{R}^D} f(x,v) dv\right\|_{B^0_{1,1}\left(dx\right)}\leq
		 C \left(\left\|f\right\|_{\widetilde L^1\left(dx;B^0_{1,1}\left(dv\right)\right)}+ \left\|g\right\|_{\widetilde L^1\left(dx;B^\beta_{1,1}\left(dv\right)\right)}\right),
	\end{equation}
	where the constant $C>0$ only depends on fixed parameters, which concludes our proof.
\end{proof}

\begin{proof}[Proof of Theorem \ref{averaging lemma}]
	
	Notice first that there exists $\beta<1$ so that
	\begin{equation}
		\left(1-\Delta_{v}\right)^\frac{\gamma}{2}g_n \in L^1\left(\mathbb{R}^D_x;B^\beta_{1,1}\left(\mathbb{R}^D_v\right)\right).
	\end{equation}
	We may therefore straightforwardly apply Proposition \ref{pre averaging lemma} to infer that
	\begin{equation}
		\text{$\left\{\int_{\mathbb{R}^D} f_n(x,v) dv\right\}_{n=0}^\infty$ is bounded in $B^0_{1,1}\left(\mathbb{R}^D_x\right)$.}
	\end{equation}

	Let us recall now a few basic facts from the theory of functions spaces. We refer the reader to \cite{runst} or \cite{triebel} for more details on the subject. The norm that defines the Triebel-Lizorkin spaces $F^s_{p,q}\left(\mathbb{R}^D\right)$, for any $1\leq p,q<\infty$ and $s\in\mathbb{R}$, is given by (in the notation of Section \ref{LP decomposition})

	\begin{equation}
		\left\|f\right\|_{F^s_{p,q}}=\left(\int_{\mathbb{R}^D}\left(\left|\Delta_0 f(x)\right|^q+\sum_{k=0}^\infty 2^{ksq}\left|\Delta_{2^k}f(x)\right|^q\right)^\frac{p}{q} dx\right)^\frac{1}{p},
	\end{equation}
	where $f\in\mathcal{S}'\left(\mathbb{R}^D\right)$. Thus, it is clear that $B^0_{1,1}=F^0_{1,1}$ and therefore $B^0_{1,1}$ is continuously embedded in $F^0_{1,2}$, which is nothing but the local Hardy space $h_1$ (cf. \cite[Section 2.3.5]{triebel}).

	Therefore, it holds that
	\begin{equation}
		\begin{gathered}
			\text{$\left\{\int_{\mathbb{R}^D} f_n(x,v) dv\right\}_{n=0}^\infty$ is bounded in $h_1\left(\mathbb{R}^D_x\right)$}\\
			\text{and $\left\{f_n(x,v)\right\}_{n=0}^\infty$ is bounded in $L^1\left(\mathbb{R}^D_x;h_1\left(\mathbb{R}^D_v\right)\right)$.}
		\end{gathered}
	\end{equation}
	We recall now Stein's $L\log L$ result (cf. \cite{stein3}, \cite[Chapter 1, \S5.2]{stein2} and \cite[Chapter I, \S8.14, Chapter III, \S5.3]{stein}), which states that a sequence of positive functions uniformly bounded in $h_1$ necessarily satisfies locally a uniform $L\log L$ bound. Here, for a given compact subset $K\subset \mathbb{R}^D$, the Orlicz space $L\log L\left(K\right)$ is defined as the Banach space endowed with the norm
	\begin{equation}
		\left\|f\right\|_{L\log L\left(K\right)}=\inf\set{\lambda>0}{\int_Kh\left(\frac{|f(x)|}{\lambda}\right)dx\leq 1},
	\end{equation}
	where $h(z)=(1+z)\log(1+z)-z$ is a convex function. Thus, we conclude that, for any compact subsets $K_x,K_v\subset\mathbb{R}^D$,
	\begin{equation}
		\begin{gathered}
			\text{$\left\{\int_{\mathbb{R}^D} f_n(x,v) dv\right\}_{n=0}^\infty$ is bounded in $L\log L\left(K_x\right)$}\\
			\text{and $\left\{f_n(x,v)\right\}_{n=0}^\infty$ is bounded in $L^1\left(\mathbb{R}^D_x;L\log L\left(K_v\right)\right)$.}
		\end{gathered}
	\end{equation}
	It then follows that
	\begin{equation}
		\begin{gathered}
			\text{$\left\{\int_{\mathbb{R}^D} f_n(x,v) dv\right\}_{n=0}^\infty$ is equi-integrable}\\
			\text{and $\left\{f_n(x,v)\right\}_{n=0}^\infty$ is equi-integrable in $v$,}
		\end{gathered}
	\end{equation}
	and this is, together with the positiveness of the sequences, exactly what is required by Lemma 5.2 of \cite{golse3} in order to deduce that
	\begin{equation}
		\text{$\{f_n(x,v)\}_{n=0}^\infty$ is equi-integrable (in all variables $x$ and $v$),}
	\end{equation}
	which concludes the proof of the theorem.
\end{proof}

We provide now simple extensions of Theorem \ref{averaging lemma} and Proposition \ref{pre averaging lemma} to the time dependent transport equation with a vanishing temporal derivative. These variants of the preceding results are precisely the versions needed for their application to the viscous incompressible hydrodynamic limit of the Boltzmann equation without cutoff in \cite{arsenio3}. Moreover, the proofs provided below show how to smoothly adapt the general methods developed in this work to the time dependent kinetic transport equation.

\begin{thm}\label{averaging lemma 2}

	Let $\{f_n(t,x,v)\}_{n=0}^\infty$ be a bounded sequence in $L^\infty \left(\mathbb{R}_t;L^1\left(\mathbb{R}^D_x\times\mathbb{R}^D_v\right)\right)$ and in $L^1\left(\mathbb{R}_t\times\mathbb{R}^D_x;B^0_{1,1}\left(\mathbb{R}^D_v\right)\right)$, such that $f_n\geq0$ and that
	\begin{equation}
		\left(\delta_n\partial_t+v\cdot\nabla_x\right)f_n=\left(1-\Delta_{v}\right)^\frac{\gamma}{2}g_n
	\end{equation}
	for some sequence $\{g_n(t,x,v)\}_{n=0}^\infty$ bounded in $L^1\left(\mathbb{R}_t\times\mathbb{R}^D_x\times \mathbb{R}^D_v\right)$, some $\gamma\in\mathbb{R}$ and where the sequence $\delta_n>0$ vanishes as $n\rightarrow \infty$.
	
	Then,
	\begin{equation}
		\text{$\{f_n\}_{n=0}^\infty$ is equi-integrable (in all variables $t$, $x$ and $v$).}
	\end{equation}

\end{thm}

As before, the above theorem will be obtained as a corollary of the following proposition.

\begin{prop}\label{pre averaging lemma 2}

	Let $\{f_n(t,x,v)\}_{n=0}^\infty$ be a bounded sequence in $L^1\left(\mathbb{R}_t\times\mathbb{R}^D_x;B^0_{1,1}\left(\mathbb{R}^D_v\right)\right)$ such that
	\begin{equation}
		\left(\delta_n\partial_t+v\cdot\nabla_x\right)f_n=g_n
	\end{equation}
	for some bounded sequence $\{g_n(x,v)\}_{n=0}^\infty$ in $L^1\left(\mathbb{R}_t\times\mathbb{R}^D_x;B^\beta_{1,1}\left(\mathbb{R}^D_v\right)\right)$, where $\beta\in\mathbb{R}$ and the sequence $\delta_n>0$ vanishes as $n\rightarrow \infty$.
	
	Then,
	\begin{equation}
		\text{$\left\{\int_{\mathbb{R}^D} f_n(t,x,v) dv\right\}_{n=0}^\infty$ is bounded in $L^1\left(\mathbb{R}_t;B^0_{1,1}\left(\mathbb{R}^D_x\right)\right)$.}
	\end{equation}

\end{prop}

\begin{proof}
	We give here a short demonstration of this lemma by following closely the proof of Proposition \ref{pre averaging lemma} and emphasizing the necessary changes. Thus, we first introduce an interpolation parameter $s\in\mathbb{R}$. Then, it trivially holds that
	\begin{equation}
		\begin{cases}
			(\partial_s+\delta_n\partial_t+v\cdot\nabla_x)f_n=g_n,\\
			f_n(s=0)=f_n.
		\end{cases}
	\end{equation}
	Hence the interpolation formula
	\begin{equation}
		f_n(t,x,v)=f_n(t-\delta_n s,x-sv,v)+\int_0^sg_n\left(t-\delta_n\sigma,x-\sigma v,v\right) d\sigma.
	\end{equation}
	It follows that, in virtue of \eqref{x-to-v},
	\begin{equation}
		\begin{aligned}
			\Delta_{2^k}^x\int_{\mathbb{R}^D} f_n(t,x,v) dv & =\int_{\mathbb{R}^D} \left(\Delta_{s2^k}^v f_n\right)(t-\delta_n s,x-sv,v) dv\\
			& +\int_0^s\int_{\mathbb{R}^D} \left(\Delta_{\sigma 2^k}^v g_n\right)(t-\delta_n\sigma,x-\sigma v,v) dvd\sigma,
		\end{aligned}
	\end{equation}
	which substitutes the interpolation formula \eqref{D-x}.
	
	Then, integrating in time and space yields
	\begin{equation}
		\left\|\Delta_{2^k}^x\int_{\mathbb{R}^D} f_n(t,x,v) dv\right\|_{L^1_{t,x}}\leq \left\|\Delta_{s2^k}^v f_n\right\|_{L^1_{t,x,v}}
		+\int_0^s\left\|\Delta_{\sigma 2^k}^v g_n\right\|_{L^1_{t,x,v}}d\sigma.
	\end{equation}
	For each value of $k$, we fix now the interpolation parameter $s$ as $s_k=2^{k\frac{\beta}{1-\beta}}$, which yields, summing over $k$,
	\begin{equation}
		\begin{aligned}
			\sum_{k=0}^\infty \left\|\Delta_{2^k}^x\int_{\mathbb{R}^D} f_n(t,x,v) dv\right\|_{L^1_{t,x}}
			\leq
			\sum_{k=0}^\infty \left\|\Delta_{s_k2^k}^v f_n\right\|_{L^1_{t,x,v}}
			+\sum_{k=0}^\infty \int_0^{s_k}\left\|\Delta_{\sigma2^k}^v g_n\right\|_{L^1_{t,x,v}}d\sigma.
		\end{aligned}
	\end{equation}
	
	The remainder of the demonstration only consists in an analysis of the velocity frequencies and it is thus strictly identical to the proof of Proposition \ref{pre averaging lemma}. And so, we infer that
	\begin{equation}
		\left\|\int_{\mathbb{R}^D} f_n(t,x,v) dv\right\|_{L^1\left(dt;B^0_{1,1}\left(dx\right)\right)}\leq
		 C\left( \left\|f_n\right\|_{L^1\left(dtdx;B^0_{1,1}\left(dv\right)\right)}+
		\left\|g_n\right\|_{L^1\left(dtdx;B^\beta_{1,1}\left(dv\right)\right)}\right),
	\end{equation}
	which concludes our proof.
\end{proof}

\begin{proof}[Proof of Theorem \ref{averaging lemma 2}]
	Notice first that there exists $\beta<1$ so that
	\begin{equation}
		\left(1-\Delta_{v}\right)^\frac{\gamma}{2}g_n \in L^1\left(\mathbb{R}_t\times \mathbb{R}^D_x;B^\beta_{1,1}\left(\mathbb{R}^D_v\right)\right).
	\end{equation}
	We may therefore straightforwardly apply Proposition \ref{pre averaging lemma 2} to infer that
	\begin{equation}
		\text{$\left\{\int_{\mathbb{R}^D} f_n(t,x,v) dv\right\}_{n=0}^\infty$ is bounded in $L^1\left(\mathbb{R}_t;B^0_{1,1}\left(\mathbb{R}^D_x\right)\right)$.}
	\end{equation}
	
	Then, just as in the proof of lemma \ref{averaging lemma}, using Stein's $L\log L$ result, we conclude that, for any compact subsets $K_x,K_v\subset\mathbb{R}^D$,
	\begin{equation}
		\begin{gathered}
			\text{$\left\{\int_{\mathbb{R}^D} f_n(t,x,v) dv\right\}_{n=0}^\infty$ is bounded in $L^1\left(\mathbb{R}_t;L\log L\left(K_x\right)\right)$}\\
			\text{and $\left\{f_n(t,x,v)\right\}_{n=0}^\infty$ is bounded in $L^1\left(\mathbb{R}_t\times\mathbb{R}^D_x;L\log L\left(K_v\right)\right)$.}
		\end{gathered}
	\end{equation}
	Since $\{f_n(t,x,v)\}_{n=0}^\infty$ is also bounded in $L^\infty \left(\mathbb{R}_t;L^1\left(\mathbb{R}^D_x\times\mathbb{R}^D_v\right)\right)$, it then follows that
	\begin{equation}
		\begin{gathered}
			\text{$\left\{\int_{\mathbb{R}^D\times\mathbb{R}^D	} f_n(t,x,v) dxdv\right\}_{n=0}^\infty$ is equi-integrable in $t$,}\\
			\text{$\left\{\int_{\mathbb{R}^D} f_n(t,x,v) dv\right\}_{n=0}^\infty$ is equi-integrable in $x$}\\
			\text{and $\left\{f_n(t,x,v)\right\}_{n=0}^\infty$ is equi-integrable in $v$.}
		\end{gathered}
	\end{equation}
	By the positiveness of the sequences, it is then possible to apply here Lemma 5.2 of \cite{golse3} iteratively (one may also consult the proof of Proposition 3.3.5 of \cite{saintraymond}) in order to conclude that
	\begin{equation}
		\text{$\{f_n(t,x,v)\}_{n=0}^\infty$ is equi-integrable (in all variables $t$, $x$ and $v$),}
	\end{equation}
	which concludes the proof of the theorem.
\end{proof}


\section{Proofs of the main results}\label{proofs}

We proceed now to the demonstrations of the main results from Section \ref{main results}. Since all our proofs are based on the crucial decomposition \eqref{crucial 0} given by Proposition \ref{crucial}, we will systematically start the proofs by establishing sharp dispersive estimates on the operators $A_\delta^t$ and $B_\delta^t$ appearing in the right-hand side of \eqref{crucial 0}. In accordance with the hypotheses of Proposition \ref{crucial}, we will always consider the operators $A_\delta^t$ and $B_\delta^t$ as defined for some given cutoff $\rho\in\mathcal{S}\left(\mathbb{R}\right)$ such that $\rho(0)=\frac{1}{2\pi}\int_\mathbb{R}\hat\rho(s)ds=\int_\mathbb{R}\tilde\rho(s)ds=1$ and that $\tilde\rho$ is compactly supported inside $\left[1,2\right]$.

\subsection{Velocity averaging in $L^1_xL^p_v$, inhomogeneous case}

In this section, we show the proofs of the averaging lemmas in the endpoint case $L^1_xL^p_v$, i.e. of Theorems \ref{pre averaging lemma p} and \ref{pre averaging lemma p 2}.

For the sake of simplicity, we will only consider here the operators $A_\delta^t$ and $B_\delta^t$ as defined for the specific cutoff $\tilde\rho(s)=\delta_{\left\{s=1\right\}}$, so that
\begin{equation}
	\begin{aligned}
		A_{\delta}^t f(x)
		& = \int_{\mathbb{R}^D}
		\Delta_{\delta}^x f \left(x-tv,v\right) dv , \\
		B_{\delta}^t g(x)
		& = \int_0^1 \int_{\mathbb{R}^D}
		\Delta_{\delta}^x g \left(x-\sigma t v,v\right)  dv d\sigma.
	\end{aligned}
\end{equation}
This simplification is performed at absolutely no cost of generality, since it is then possible to easily extend the results to the case of a general smooth cutoff $\rho\in\mathcal{S}\left(\mathbb{R}\right)$ by considering
\begin{equation}
	\int_{\mathbb{R}} \left[A_{\delta}^{st} f(x)\right] \tilde \rho(s)ds
	\qquad\text{and}\qquad
	\int_{\mathbb{R}} \left[B_{\delta}^{st} g(x)\right] s\tilde \rho(s)ds.
\end{equation}

It is to be emphasized that the above simplification is possible only because the identities \eqref{crucial 1} from Proposition \ref{crucial} will not be employed here.

\begin{lem}\label{operators norms dispersive}
	For every $1\leq p, q\leq \infty$, $\alpha\in\mathbb{R}$, $k\in\mathbb{N}$ and $t\geq 2^{-k}$, it holds that
	\begin{equation}
		\left\|A_{0}^1f\right\|_{L^p_x} \leq C \left\| \Delta_{0}^v f \right\|_{L^1_xL^p_v}
		\leq C \left\| f \right\|_{\widetilde L^1\left( \mathbb{R}^D_x; B^\alpha_{p,q}\left(\mathbb{R}^D_v\right)\right)}
	\end{equation}
	and
	\begin{equation}
		\begin{aligned}
			\left\|A_{2^k}^t f\right\|_{L^p_x}
			& \leq {C\over t^{D\left(1-\frac 1p\right)}} \left\| \Delta_{t2^k}^v f \right\|_{L^1_xL^p_v}\\
			& \leq {C\over t^{\alpha+D\left(1-\frac 1p\right)} 2^{k\alpha}} \left\| f \right\|_{\widetilde L^1\left( \mathbb{R}^D_x; B^\alpha_{p,q}\left(\mathbb{R}^D_v\right)\right)},
		\end{aligned}
	\end{equation}
	where $C>0$ is independent of $t$ and $2^k$.
	
	Furthermore, if
	\begin{equation}
		D\left(1-\frac 1p\right) < 1,
	\end{equation}
	then, for any $\beta\in\mathbb{R}$, it holds that
	\begin{equation}
		\left\|B_{0}^1g\right\|_{L^p_x}
		\leq C \left\| S_{2}^v g \right\|_{L^1_xL^p_v}
		\leq C \left\| g \right\|_{\widetilde L^1\left( \mathbb{R}^D_x; B^\beta_{p,q}\left(\mathbb{R}^D_v\right)\right)}
	\end{equation}
	and, if further $\beta + D\left(1-\frac 1p\right) \neq 1$,
	\begin{equation}
		\begin{aligned}
			\left\|B_{2^k}^tg\right\|_{L^p_x}
			& \leq
			{C\over t^{\beta+D\left(1-\frac 1p\right)} 2^{k\beta}}
			\Bigg(
			2^{ -j_k \left( 1 - \beta - D \left( 1 - \frac{1}{p} \right) \right) }
			\left\| \Delta_0^v g \right\|_{L^1_x L^p_v}\\
			&\hspace{20mm} +
			\sum_{j=0}^{j_k+1}
			2^{ -\left( j_k + 1 - j \right) \left( 1 - \beta - D \left( 1 - \frac{1}{p} \right) \right) }
			2^{j\beta} \left\| \Delta_{2^j}^v g \right\|_{L^1_x L^p_v}\Bigg)\\
			& \leq {C\over t^{\gamma+D\left(1-\frac 1p\right)} 2^{k\gamma}}
			\left\| g \right\|_{\widetilde L^1\left( \mathbb{R}^D_x; B^\beta_{p,q}\left(\mathbb{R}^D_v\right)\right)},
		\end{aligned}
	\end{equation}
	where $\gamma=\min \left\{\beta,1-D\left(1-\frac 1p\right)\right\}$, $j_k\geq 1$ is the largest integer such that $2^{j_k-1} \leq t2^k $ and $C>0$ is independent of $t$ and $2^k$.
	
	Finally, if $\beta + D\left(1-\frac 1p\right) = 1$, then
	\begin{equation}
		\begin{aligned}
			\left\|B_{2^k}^tg\right\|_{L^p_x}
			& \leq
			{C\over t 2^{k\beta}}
			\left(
			\left\| \Delta_0^v g \right\|_{L^1_x L^p_v}+
			\sum_{j=0}^{j_k+1}
			2^{j\beta} \left\| \Delta_{2^j}^v g \right\|_{L^1_x L^p_v}\right)\\
			& \leq {C\over t 2^{k\beta}}\log\left(1+t2^k\right)^\frac{1}{q'}
			\left\| g \right\|_{\widetilde L^1\left( \mathbb{R}^D_x; B^\beta_{p,q}\left(\mathbb{R}^D_v\right)\right)},
		\end{aligned}
	\end{equation}
	where $C>0$ is independent of $t$ and $2^k$.
\end{lem}

\begin{proof}
	These bounds will follow from the dispersive properties of the transport operator, which were first established by Castella and Perthame in \cite{CP96}. More precisely, we will utilize the following dispersive estimate, valid for any $1\leq p\leq\infty$, which is a simple consequence of the change of variables $v \mapsto y = x-tv$,
	\begin{equation}
		\begin{aligned}
			\left\|h(x-tv,v)\right\|_{L^p_xL^1_v} & =
			t^{-D}\left\|h\left(y,{x-y \over t}\right)\right\|_{L^p_xL^1_y} \\
			& \leq t^{-D}\left\|h\left(y,{x-y \over t}\right)\right\|_{L^1_yL^p_x} \\
			& = t^{-D\left(1-\frac 1p\right)}\left\|h\left(x,v\right)\right\|_{L^1_xL^p_v} .
		\end{aligned}
	\end{equation}
	
	Thus, we easily deduce, utilizing identities \eqref{x-to-v}, that
	\begin{equation}
		\left\|A_{0}^1f(x)\right\|_{L^p_x} \leq C \left\| \left(\Delta_{0}^v f\right)(x,v) \right\|_{L^1_xL^p_v}
		\leq C \left\| f(x,v) \right\|_{\widetilde L^1\left( \mathbb{R}^D_x; B^\alpha_{p,q}\left(\mathbb{R}^D_v\right)\right)},
	\end{equation}
	and
	\begin{equation}
		\begin{aligned}
			\left\|A_{2^k}^tf(x)\right\|_{L^p_x} &\leq Ct^{-D\left(1-\frac 1p\right)} \left\| \left(\Delta_{t2^k}^v f\right)(x,v) \right\|_{L^1_xL^p_v} \\
			& \leq Ct^{-\alpha-D\left(1-\frac 1p\right)} 2^{-k\alpha} \left\| f(x,v) \right\|_{\widetilde L^1\left( \mathbb{R}^D_x; B^\alpha_{p,q}\left(\mathbb{R}^D_v\right)\right)},
		\end{aligned}
	\end{equation}
	which concludes the estimate for $A_{0}^1$ and $A_{2^k}^t$.
	
	As for $B_{0}^1$ and $B_{2^k}^t$, we first obtain
	\begin{equation}
			\left\|B_{0}^1g(x)\right\|_{L^p_x}
			\leq C\int_0^1 s^{-D\left(1-\frac 1p\right)} \left\| \left(S_{s}^v g\right)(x,v) \right\|_{L^1_xL^p_v} ds
	\end{equation}
	and
	\begin{equation}\label{high freq B}
		\left\|B_{2^k}^tg(x)\right\|_{L^p_x} \leq C\int_0^1 (st)^{-D\left(1-\frac 1p\right)} \left\| \left(\Delta_{st2^k}^v g\right)(x,v) \right\|_{L^1_xL^p_v} ds.
	\end{equation}
	Then, since $D\left(1-\frac 1p\right)<1$ and $S_{s}^v=S_{s}^vS_{2}^v$, we deduce that
	\begin{equation}
			\left\|B_{0}^1g(x)\right\|_{L^p_x}
			\leq C \left\| \left(S_{2}^v g\right)(x,v) \right\|_{L^1_xL^p_v}
			\leq C \left\| g(x,v) \right\|_{\widetilde L^1\left( \mathbb{R}^D_x; B^\beta_{p,q}\left(\mathbb{R}^D_v\right)\right)},
	\end{equation}
	which concludes the estimate on the low frequency component.
	
	Regarding the high frequencies, we further split the integral in \eqref{high freq B} into small dyadic intervals
	\begin{equation}
		\begin{gathered}
			I_0 = \left[ 0, {1\over t2^k} \right],\ 
			I_1 = \left[ {1\over t2^k}, {2\over t2^k} \right],\ 
			\ldots ,\ 
			I_j = \left[ {2^{j-1}\over t2^k},  {2^j\over t2^k} \right] \\
			\text{and finally } I_{j_k} = \left[ {2^{j_k-1}\over t2^k}, 1 \right],
		\end{gathered}
	\end{equation}
	where $j_k\geq 1$ is the largest integer such that $2^{j_k-1} \leq t2^k $. Thus, on each dyadic interval $I_j$, where $ 1 \leq j \leq j_k$, the frequency $st2^k$ is between $2^{j-1}$ and $2^{j}$. Therefore, we deduce that, for any $1 \leq j \leq j_k$,
	\begin{equation}
		\begin{aligned}
			\int_{I_j} (st)^{-D \left( 1 - \frac 1 p \right) }
			\left\| \Delta_{st2^k}^v g \right\|_{L^1_x L^p_v} ds
			& =
			\int_{I_j} (st)^{-D \left( 1 - \frac 1 p \right) }
			\left\| \Delta_{\left[2^{j-2},2^{j+1}\right]}^v
			\Delta_{st2^k}^v g \right\|_{L^1_x L^p_v} ds \\
			& \leq C \left|I_j\right| 2^{ -(j-k)D \left( 1 - \frac 1 p \right) }
			\left\| \Delta_{\left[2^{j-2},2^{j+1}\right]}^v g \right\|_{L^1_x L^p_v} \\
			& =\frac Ct
			2^{ \left( j - k \right) \left( 1 - D \left( 1 - \frac{1}{p} \right) \right) }
			\left\| \Delta_{\left[2^{j-2},2^{j+1}\right]}^v g \right\|_{L^1_x L^p_v}.
		\end{aligned}
	\end{equation}
	Furthermore, when $j=0$, it is readily seen, since $D\left(1-\frac{1}{p}\right) < 1$, that
	\begin{equation}
		\begin{aligned}
			& \int_{I_0} (st)^{-D \left( 1 - \frac 1 p \right) }
			\left\| \Delta_{st2^k}^v g \right\|_{L^1_x L^p_v} ds\\
			& \hspace{30mm} =
			\int_{I_0} (st)^{-D \left( 1 - \frac 1 p \right) }
			\left\| S_4^v
			\Delta_{st2^k}^v g \right\|_{L^1_x L^p_v} ds \\
			& \hspace{30mm} \leq C \int_{I_0} (st)^{ -D \left( 1 - \frac 1 p \right) } ds
			\left\| S_4^v g \right\|_{L^1_x L^p_v} \\
			& \hspace{30mm} =C \frac{1}{t\left(1-D\left(1-\frac{1}{p}\right)\right)}
			2^{ - k \left( 1 - D \left( 1 - \frac{1}{p} \right) \right) }
			\left\| S_4^v g \right\|_{L^1_x L^p_v}.
		\end{aligned}
	\end{equation}
	Thus, on the whole, we have obtained that
	\begin{equation}
		\begin{aligned}
			&\left\|B_{2^k}^tg\right\|_{L^p_x}\\
			& \leq C\int_0^1 (st)^{-D\left(1-\frac 1p\right)} \left\| \Delta_{st2^k}^v g \right\|_{L^1_xL^p_v} ds\\
			& = C\sum_{j=0}^{j_k} \int_{I_j} (st)^{-D\left(1-\frac 1p\right)} \left\| \Delta_{st2^k}^v g \right\|_{L^1_xL^p_v} ds \\
			& \leq {C\over t}
			\left(
			2^{ - k \left( 1 - D \left( 1 - \frac{1}{p} \right) \right) }
			\left\| \Delta_0^v g \right\|_{L^1_x L^p_v}
			+
			\sum_{j=0}^{j_k+1}
			2^{ \left( j - k \right) \left( 1 - D \left( 1 - \frac{1}{p} \right) \right) }
			\left\| \Delta_{2^j}^v g \right\|_{L^1_x L^p_v}
			\right) \\
			& \leq C
			t^{ - \beta - D \left( 1 - \frac{1}{p} \right) }
			2^{-k\beta}
			\Bigg(
			2^{ - j_k \left( 1 - \beta - D \left( 1 - \frac{1}{p} \right) \right) }
			\left\| \Delta_0^v g \right\|_{L^1_x L^p_v}\\
			&\hspace{20mm}+
			\sum_{j=0}^{j_k+1}
			2^{ -\left( j_k + 1 - j \right) \left( 1 - \beta - D \left( 1 - \frac{1}{p} \right) \right) }
			2^{j\beta} \left\| \Delta_{2^j}^v g \right\|_{L^1_x L^p_v}
			\Bigg).
		\end{aligned}
	\end{equation}
	
	It follows that, if $1 > \beta + D\left( 1 - \frac 1 p \right)$, a further application of H\"older's inequality to the preceding estimate yields
	\begin{equation}
		\begin{aligned}
			&\left\|B_{2^k}^tg(x)\right\|_{L^p_x}\\
			& \leq C
			t^{ - \beta - D \left( 1 - \frac{1}{p} \right) }
			2^{-k\beta}
			\left\| \left\{ 2^{- j \left( 1 - \beta - D\left( 1 - \frac 1 p \right) \right) } \right\}_{j=0}^\infty \right\|_{\ell^{q'}}
			\left\| g(x,v) \right\|_{\widetilde L^1\left( \mathbb{R}^D_x; B^\beta_{p,q}\left(\mathbb{R}^D_v\right)\right)}.
		\end{aligned}
	\end{equation}
	And in case $1 < \beta + D\left( 1 - \frac 1 p \right)$, we obtain
	\begin{equation}
		\begin{aligned}
			&\left\|B_{2^k}^tg(x)\right\|_{L^p_x}\\
			& \leq \frac Ct
			2^{ k \left( D \left( 1 - \frac{1}{p} \right) - 1 \right) }
			\left\| \left\{ 2^{ j \left( 1 - \beta - D\left( 1 - \frac 1 p \right) \right) } \right\}_{j=0}^\infty \right\|_{\ell^{q'}}
			\left\| g(x,v) \right\|_{\widetilde L^1\left( \mathbb{R}^D_x; B^\beta_{p,q}\left(\mathbb{R}^D_v\right)\right)}.
		\end{aligned}
	\end{equation}
	Finally, if $1 = \beta + D\left( 1 - \frac 1 p \right)$, since $j_k+3\leq\frac{\log\left(t2^{k+4}\right)}{\log 2}$, then
	\begin{equation}
		\begin{aligned}
			\left\|B_{2^k}^tg\right\|_{L^p_x}
			& \leq
			{C\over t 2^{k\beta}}
			\left(
			\left\| \Delta_0^v g \right\|_{L^1_x L^p_v}+
			\sum_{j=0}^{j_k+1}
			2^{j\beta} \left\| \Delta_{2^j}^v g \right\|_{L^1_x L^p_v}\right)\\
			& \leq {C\over t 2^{k\beta}}\left(j_k+3\right)^\frac{1}{q'}
			\left\| g \right\|_{\widetilde L^1\left( \mathbb{R}^D_x; B^\beta_{p,q}\left(\mathbb{R}^D_v\right)\right)}\\
			& \leq {C\over t 2^{k\beta}}\log\left(1+t2^k\right)^\frac{1}{q'}
			\left\| g \right\|_{\widetilde L^1\left( \mathbb{R}^D_x; B^\beta_{p,q}\left(\mathbb{R}^D_v\right)\right)},
		\end{aligned}
	\end{equation}
	which concludes the proof of the lemma.
\end{proof}

We proceed now to the proof of Theorem \ref{pre averaging lemma p}, which is the simplest of our main results.

\begin{proof}[Proof of Theorem \ref{pre averaging lemma p}]
	According to Proposition \ref{crucial}, we begin with the following dyadic interpolation formulas, for each $k\in\mathbb{N}$,
	\begin{equation}
		\begin{aligned}
			\Delta_{0}^x\int_{\mathbb{R}^D} f(x,v) dv
			& = A_{0}^1 f (x) + B_{0}^1 g (x), \\
			\Delta_{2^k}^x\int_{\mathbb{R}^D} f(x,v) dv
			& = A_{2^k}^t f (x) + t B_{2^k}^t g (x).
		\end{aligned}
	\end{equation}
	
	Then, a direct application of Lemma \ref{operators norms dispersive} yields the estimates, for every $t\geq 2^{-k}$,
	\begin{equation}\label{high low frequencies}
		\begin{aligned}
			\left\|\Delta_0^x\int_{\mathbb{R}^D} f(x,v) dv\right\|_{L^p_x}
			& \leq
			C\left\| f \right\|_{\widetilde L^1\left( \mathbb{R}^D_x; B^\alpha_{p,q}\left(\mathbb{R}^D_v\right)\right)}
			+
			C\left\| g \right\|_{\widetilde L^1\left( \mathbb{R}^D_x; B^\beta_{p,q}\left(\mathbb{R}^D_v\right)\right)},
			\\
			\left\|\Delta_{2^k}^x\int_{\mathbb{R}^D} f(x,v) dv\right\|_{L^p_x}
			& \leq
			{C\over t^{\alpha+D\left(1-\frac 1p\right)} 2^{k\alpha}} \left(t2^k\right)^\alpha\left\| \Delta_{t2^k}^v f \right\|_{L^1_xL^p_v} \\
			&+
			{C\over t^{\beta-1+D\left(1-\frac 1p\right)} 2^{k\beta}}
			\Bigg(
			2^{ - j_k \left( 1 - \beta - D \left( 1 - \frac{1}{p} \right) \right) }
			\left\| \Delta_0^v g \right\|_{L^1_x L^p_v}\\
			&+
			\sum_{j=0}^{j_k+1}
			2^{ -\left( j_k + 1 - j \right) \left( 1 - \beta - D \left( 1 - \frac{1}{p} \right) \right) }
			2^{j\beta} \left\| \Delta_{2^j}^v g \right\|_{L^1_x L^p_v}\Bigg),
		\end{aligned}
	\end{equation}
	where $j_k\geq 1$ is the largest integer such that $2^{j_k-1} \leq t2^k $ and $C>0$ is independent of $t$ and $2^k$, which concludes the estimate on the low frequencies of the velocity averages $\Delta_{0}^x \int_{\mathbb{R}^D} f(x,v) dv$. Note that the above estimate remains valid for the value $\beta + D\left(1-\frac 1p\right) = 1$.

	Next, in order to treat the high frequencies, for each value of $k$, we have to select the value of the interpolation parameter $t$ which will optimize \eqref{high low frequencies}, i.e. minimize its right-hand side. The heuristic argument yielding the optimal value for $t$ goes as follows. In virtue of Lemma \ref{operators norms dispersive}, the estimate on the high frequencies in \eqref{high low frequencies} essentially amounts to
	\begin{equation}
		\left\|\Delta_{2^k}^x\int_{\mathbb{R}^D} f(x,v) dv\right\|_{L^p_x}\\
		\leq C\left(
		{1\over t^{\alpha+D\left(1-\frac 1p\right)} 2^{k\alpha}}
		+
		{1\over t^{\gamma-1+D\left(1-\frac 1p\right)} 2^{k\gamma}}
		\right),
	\end{equation}
	where $\gamma=\min \left\{\beta,1-D\left(1-\frac 1p\right)\right\}$. Therefore, up to multiplicative constants, the right-hand side above will be minimized when both terms are equal, which leads to an optimal value of the interpolation parameter $t$ of
	\begin{equation}
			t_k =2^{-k\frac{\alpha-\gamma}{1+\alpha-\gamma}}\geq 2^{-k},
	\end{equation}
	where we have used the hypothesis $\alpha>-D\left(1-\frac 1p\right)$ to deduce that $1+\alpha-\gamma>0$.
	
	Note that, in the case $\beta+D\left(1-\frac 1p\right)\geq 1$, we can choose $t=\infty$, which is, in fact, more optimal than $t2^k=2^{k\frac{1}{1+\alpha-\gamma}}$, since it eliminates the first term in the right-hand side of the above estimates. This case is discussed later on.
	
	
	Thus, in the case $ \beta + D \left( 1 - \frac 1 p \right)<1$ so that $\gamma=\beta$, setting $t=t_k$ in \eqref{high low frequencies}, denoting $ s = \frac{\alpha+D \left( 1 - \frac 1 p \right)}{1+\alpha-\gamma}  - D \left( 1 - \frac 1 p \right)  $, noticing that $j_k=\left[1+{k\over 1+\alpha-\gamma}\right]$ and summing over $k$, yields that
	\begin{equation}\label{besov averaging p}
		\begin{aligned}
			& \left\| \left\{ 2^{ k s } \left\| \Delta_{2^k}^x \int_{\mathbb{R}^D} f(x,v) dv \right\|_{L^p_x} \right\}_{k=0}^\infty \right\|_{\ell^q} \\
			&\hspace{20mm} \leq
			C \left\| \left\{ \left(t_k2^k\right)^\alpha \left\| \Delta_{t_k2^k}^v f \right\|_{L^1_x L^p_v} \right\}_{k=0}^\infty \right\|_{\ell^q}\\
			&\hspace{20mm} + C \left\| \left\{ 2^{ - k \frac{1}{ 1 + \alpha - \beta } \left( 1 - \beta - D \left( 1 - \frac 1 p \right) \right) } \right\}_{k=0}^\infty \right\|_{\ell^q} \left\| \Delta_0^v g \right\|_{L^1_x L^p_v} \\
			&\hspace{20mm} + C \left\| \left\{ \sum_{j=0}^{ j_k + 1 } 2^{-\left(j_k + 1 - j\right) \left( 1 - \beta - D\left( 1 - \frac 1 p \right) \right) } 2^{j\beta}
			\left\| \Delta_{2^j }^v g \right\|_{L^1_x L^p_v} \right\}_{k=0}^\infty \right\|_{\ell^q}.
		\end{aligned}
	\end{equation}
	Then, it is readily seen that the first term in the right-hand side above is controlled by $\left\| f \right\|_{\widetilde L^1\left(\mathbb{R}^D_x;B^\alpha_{p,q}\left(\mathbb{R}^D_v\right)\right)}$, for $t_k2^k=2^{k{1\over 1+\alpha-\beta}}\rightarrow\infty$ as $k\rightarrow\infty$. Furthermore, writing $ a_j =  2^{ j \beta } \left\| \Delta_{2^j }^v g \right\|_{L^1_x L^p_v} \mathbb{1}_{\left\{j\geq 0\right\}} $ and $ b_j = 2^{ -j \left( 1 -\beta - D \left( 1 - \frac 1 p \right) \right)}  \mathbb{1}_{ \left\{ j \geq 0 \right\} } $, for all $j\in\mathbb{Z}$, we see that $a=\left\{a_j\right\}_{j\in\mathbb{Z}} \in \ell^q$, $b=\left\{b_j\right\}_{j\in\mathbb{Z}} \in \ell^1$ and that
	\begin{equation}
		\begin{aligned}
		\left\| \left\{ \sum_{j=0}^{ j_k + 1 } 2^{-\left(j_k + 1 - j\right) \left( 1 - \beta - D\left( 1 - \frac 1 p \right) \right) } 2^{j\beta}
		\left\| \Delta_{2^j }^v g \right\|_{L^1_x L^p_v}\right\}_{k=0}^\infty \right\|_{\ell^q}
		& = \left\| \left\{\left( a*b \right)_{j_k+1}\right\}_{k=0}^\infty \right\|_{\ell^q}\\
		& \leq \left\| a \right\|_{\ell^q} \left\| b \right\|_{\ell^1} \\
		& \leq C
		\left\| g \right\|_{\widetilde L^1\left(\mathbb{R}^D_x;B^\beta_{p,q}\left(\mathbb{R}^D_v\right)\right)}.
		\end{aligned}
	\end{equation}
	Thus, on the whole, combining the preceding estimates with \eqref{besov averaging p}, we deduce that
	\begin{equation}
		\begin{aligned}
			& \left\| \left\{ 2^{ k s } \left\| \Delta_{2^k}^x \int_{\mathbb{R}^D} f(x,v) dv \right\|_{L^p_x} \right\}_{k=0}^\infty \right\|_{\ell^q} \\
			& \hspace{40mm} \leq 
			C \left\| f \right\|_{\widetilde L^1\left(\mathbb{R}^D_x;B^\alpha_{p,q}\left(\mathbb{R}^D_v\right)\right)}
			+
			C \left\| g \right\|_{\widetilde L^1\left(\mathbb{R}^D_x;B^\beta_{p,q}\left(\mathbb{R}^D_v\right)\right)},
		\end{aligned}
	\end{equation}
	which concludes the proof in the case $ \beta + D \left( 1 - \frac 1 p \right)<1$.

	Regarding the case $ \beta + D \left( 1 - \frac 1 p \right)\geq1$ so that $\gamma=1-D \left( 1 - \frac 1 p \right)$, we note that the estimate \eqref{high low frequencies} on the high frequencies implies that, recalling $j_k$ satisfies $t2^{k}\leq 2^{j_k}\leq t2^{k+1}$,
	\begin{equation}
		\begin{aligned}
			\left\|\Delta_{2^k}^x\int_{\mathbb{R}^D} f(x,v) dv\right\|_{L^p_x}
			& \leq
			{C\over t^{\alpha+D\left(1-\frac 1p\right)} 2^{k\alpha}}
			\left\| f \right\|_{\widetilde L^1\left(\mathbb{R}^D_x;B^\alpha_{p,q}\left(\mathbb{R}^D_v\right)\right)} \\
			&+
			{C\over 2^{k\left(1-D\left(1-\frac 1p\right)\right)}}
			\Bigg(
			\left\| \Delta_0^v g \right\|_{L^1_x L^p_v}\\
			&+
			\sum_{j=0}^{j_k+1}
			2^{ j \left( 1 - \beta - D \left( 1 - \frac{1}{p} \right) \right) }
			2^{j\beta} \left\| \Delta_{2^j}^v g \right\|_{L^1_x L^p_v}\Bigg).
		\end{aligned}
	\end{equation}
	Therefore, letting $t$ tend to infinity, denoting $ s = \frac{\alpha+D \left( 1 - \frac 1 p \right)}{1+\alpha-\gamma}  - D \left( 1 - \frac 1 p \right) = 1  - D \left( 1 - \frac 1 p \right)  $ and noticing that $\lim_{t\rightarrow\infty}j_k=\infty$, we find
	\begin{equation}
		\begin{aligned}
			2^{ks} & \left\|\Delta_{2^k}^x\int_{\mathbb{R}^D} f(x,v) dv\right\|_{L^p_x}\\
			& \leq
			C
			\left(
			\left\| \Delta_0^v g \right\|_{L^1_x L^p_v}
			+
			\sum_{j=0}^{\infty}
			2^{ j \left( 1 - \beta - D \left( 1 - \frac{1}{p} \right) \right) }
			2^{j\beta} \left\| \Delta_{2^j}^v g \right\|_{L^1_x L^p_v}\right).
		\end{aligned}
	\end{equation}
	
	It then follows that
	\begin{equation}
		\left\| \left\{ 2^{ k s } \left\| \Delta_{2^k}^x \int_{\mathbb{R}^D} f(x,v) dv \right\|_{L^p_x} \right\}_{k=0}^\infty \right\|_{\ell^\infty}
		\leq 
		C \left\| g \right\|_{\widetilde L^1\left(\mathbb{R}^D_x;B^\beta_{p,1}\left(\mathbb{R}^D_v\right)\right)},
	\end{equation}
	and, in case $ \beta + D \left( 1 - \frac 1 p \right)> 1$,
	\begin{equation}
		\left\| \left\{ 2^{ k s } \left\| \Delta_{2^k}^x \int_{\mathbb{R}^D} f(x,v) dv \right\|_{L^p_x} \right\}_{k=0}^\infty \right\|_{\ell^\infty}
		\leq 
		C \left\| g \right\|_{\widetilde L^1\left(\mathbb{R}^D_x;B^\beta_{p,q}\left(\mathbb{R}^D_v\right)\right)},
	\end{equation}
	which concludes the proof of the theorem.
\end{proof}

Next is the demonstration of Theorem \ref{pre averaging lemma p 2}, which builds upon the previous proof of Theorem \ref{pre averaging lemma p}.

\begin{proof}[Proof of Theorem \ref{pre averaging lemma p 2}]
	As in the proof of Theorem \ref{pre averaging lemma p}, we begin with the dyadic decomposition provided by Proposition \ref{crucial}, for each $k\in\mathbb{N}$,
	\begin{equation}
		\begin{aligned}
			\Delta_{0}^x\int_{\mathbb{R}^D} f(x,v) dv
			& = A_{0}^1 f (x) + B_{0}^1 g (x), \\
			\Delta_{2^k}^x\int_{\mathbb{R}^D} f(x,v) dv
			& = A_{2^k}^t f (x) + t B_{2^k}^t g (x),
		\end{aligned}
	\end{equation}
	and we utilize Lemma \ref{operators norms dispersive} to obtain, in virtue of property \eqref{crucial 3} from Proposition \ref{crucial}, that
	\begin{equation}\label{pre averaging lemma p 2.1}
		\begin{aligned}
			\left\|
			\Delta_{0}^x\int_{\mathbb{R}^D} f(x,v) dv
			\right\|_{L^p_x}
			& \leq
			C\left\| S_2^x f \right\|_{\widetilde L^1\left( \mathbb{R}^D_x; B^\alpha_{p,q}\left(\mathbb{R}^D_v\right)\right)}
			+ 
			C \left\| S_2^x g \right\|_{\widetilde L^1\left( \mathbb{R}^D_x; B^\beta_{p,q}\left(\mathbb{R}^D_v\right)\right)},\\
			\left\|
			\Delta_{2^k}^x\int_{\mathbb{R}^D} f(x,v) dv
			\right\|_{L^p_x}
			& \leq C
			{1\over t^{\alpha+D\left(1-\frac 1p\right)} 2^{k\alpha}} \left\| \Delta_{\left[2^{k-1},2^{k+1}\right]}^x f \right\|_{\widetilde L^1\left( \mathbb{R}^D_x; B^\alpha_{p,q}\left(\mathbb{R}^D_v\right)\right)}\\
			& + 
			C
			\left[
			\begin{tabular}{cc}
			${t^{1-\gamma-D\left(1-\frac 1p\right)}\over 2^{k\gamma}}$ &
			\footnotesize\text{if }$\beta + D\left(1-\frac 1p\right) < 1$
			\\
			${\log\left(1+t2^k\right)^\frac{1}{q'}\over 2^{k\gamma}}$ &
			\footnotesize\text{if }$\beta + D\left(1-\frac 1p\right) = 1$
			\\
			${1\over 2^{k\gamma}}$ &
			\footnotesize\text{if }$\beta + D\left(1-\frac 1p\right) > 1$
			\end{tabular}
			\right]\\
			&\times
			\left\| \Delta_{\left[2^{k-1},2^{k+1}\right]}^x g \right\|_{\widetilde L^1\left( \mathbb{R}^D_x; B^\beta_{p,q}\left(\mathbb{R}^D_v\right)\right)},
		\end{aligned}
	\end{equation}
	where $\gamma=\min \left\{\beta,1-D\left(1-\frac 1p\right)\right\}$. This concludes the estimate on the low frequencies of the velocity averages $\Delta_{0}^x \int_{\mathbb{R}^D} f(x,v) dv$.
	
	Next, in order to control the high frequencies, optimizing in $t$ for each value of $k$, we fix the interpolation parameter $t$ as $t_k=2^{-k\frac{\left(\alpha-\gamma\right) + \left(a-b\right)}{1+\alpha-\gamma}}$ (cf. proof of Theorem \ref{pre averaging lemma p} for a heuristic explanation on how to choose this optimal parameter). Note that $t_k\geq 2^{-k}$, for $b\geq a-1$ and $1+\alpha-\gamma>0$, and that this choice is independent of $1\leq p, q \leq \infty$.
	
	Furthermore, in the cases $\beta+D\left(1-\frac 1p\right)>1$ or $\beta+D\left(1-\frac 1p\right)=1$ and $q=1$, we can choose $t=\infty$, which is, in fact, more optimal than $t2^k=2^{k\frac{1+b-a}{1+\alpha-\gamma}}$, since it eliminates the first term in the right-hand side of the above estimate. This case is discussed later on.
	

	Therefore, denoting $ s = \left(1+b-a\right)\frac{\alpha+ D \left( 1 - \frac 1 p \right)}{1+\alpha-\gamma} + a - D \left( 1 - \frac 1 p \right) $, we find that
	\begin{equation}
		\begin{aligned}
			& 2^{ks} \left\|
			\Delta_{2^k}^x\int_{\mathbb{R}^D} f(x,v) dv
			\right\|_{L^p_x}\\
			& \leq C
			2^{ka} \left\| \Delta_{\left[2^{k-1},2^{k+1}\right]}^x f \right\|_{\widetilde L^1\left( \mathbb{R}^D_x; B^\alpha_{p,q}\left(\mathbb{R}^D_v\right)\right)}\\
			& + C
			\left[
			\begin{tabular}{cc}
			$2^{kb}$ &
			\footnotesize\text{if }$\beta + D\left(1-\frac 1p\right) \neq 1$
			\\
			$2^{kb}\left(1+\frac{1+b-a}{1+\alpha-\gamma}k\right)^\frac{1}{q'}$ &
			\footnotesize\text{if }$\beta + D\left(1-\frac 1p\right) = 1$
			\end{tabular}
			\right]
			\left\| \Delta_{\left[2^{k-1},2^{k+1}\right]}^x g \right\|_{\widetilde L^1\left( \mathbb{R}^D_x; B^\beta_{p,q}\left(\mathbb{R}^D_v\right)\right)}.
		\end{aligned}
	\end{equation}
	Consequently, summing over $k$, we obtain, in case $\beta + D\left(1-\frac 1p\right) \neq 1$ or $q=1$,
	\begin{equation}
		\begin{aligned}
			& \left\| \left\{
			2^{ks} \left\|
			\Delta_{2^k}^x\int_{\mathbb{R}^D} f(x,v) dv
			\right\|_{L^p_x} \right\}_{k=0}^\infty
			\right\|_{\ell^q} \\
			&\hspace{30mm} \leq C
			\left\| \left\{
			2^{ka} \left\| \Delta_{2^k}^x f \right\|_{\widetilde L^1\left( \mathbb{R}^D_x; B^\alpha_{p,q}\left(\mathbb{R}^D_v\right)\right)} \right\}_{k=-1}^\infty
			\right\|_{\ell^q} \\
			&\hspace{30mm} +
			C
			\left\|\left\{
			2^{kb} \left\| \Delta_{2^k}^x g \right\|_{\widetilde L^1\left( \mathbb{R}^D_x; B^\beta_{p,q}\left(\mathbb{R}^D_v\right)\right)}
			\right\}_{k=-1}^\infty\right\|_{\ell^q}\\
			&\hspace{30mm}=
			C \left\| f \right\|_{B^{a,\alpha}_{1,p,q}\left(\mathbb{R}^D_x\times\mathbb{R}^D_v\right)}
			+
			C \left\| g \right\|_{B^{b,\beta}_{1,p,q}\left(\mathbb{R}^D_x\times\mathbb{R}^D_v\right)},
		\end{aligned}
	\end{equation}
	while, if $\beta + D\left(1-\frac 1p\right) = 1$ and $q\neq 1$, we find, for any $\epsilon>0$,
	\begin{equation}
		\begin{aligned}
			& \left\| \left\{
			2^{k(s-\epsilon)} \left\|
			\Delta_{2^k}^x\int_{\mathbb{R}^D} f(x,v) dv
			\right\|_{L^p_x} \right\}_{k=0}^\infty
			\right\|_{\ell^q} \\
			&\hspace{30mm} \leq C
			\left\| \left\{
			2^{k(a-\epsilon)} \left\| \Delta_{2^k}^x f \right\|_{\widetilde L^1\left( \mathbb{R}^D_x; B^\alpha_{p,q}\left(\mathbb{R}^D_v\right)\right)} \right\}_{k=-1}^\infty
			\right\|_{\ell^q} \\
			&\hspace{30mm} +
			C
			\left\|\left\{
			2^{kb} \left\| \Delta_{2^k}^x g \right\|_{\widetilde L^1\left( \mathbb{R}^D_x; B^\beta_{p,q}\left(\mathbb{R}^D_v\right)\right)}
			\right\}_{k=-1}^\infty\right\|_{\ell^q}\\
			&\hspace{30mm}=
			C \left\| f \right\|_{B^{a-\epsilon,\alpha}_{1,p,q}\left(\mathbb{R}^D_x\times\mathbb{R}^D_v\right)}
			+
			C \left\| g \right\|_{B^{b,\beta}_{1,p,q}\left(\mathbb{R}^D_x\times\mathbb{R}^D_v\right)}.
		\end{aligned}
	\end{equation}
	
	We handle now the cases $\beta+D\left(1-\frac 1p\right)>1$ or $\beta+D\left(1-\frac 1p\right)=1$ and $q=1$, by letting $t$ tend to infinity in \eqref{pre averaging lemma p 2.1}, as mentioned previously. This leads to
	\begin{equation}
		\left\|
		\Delta_{2^k}^x\int_{\mathbb{R}^D} f(x,v) dv
		\right\|_{L^p_x}
		\leq
		\frac{C}{2^{k\gamma}}
		\left\| \Delta_{\left[2^{k-1},2^{k+1}\right]}^x g \right\|_{\widetilde L^1\left( \mathbb{R}^D_x; B^\beta_{p,q}\left(\mathbb{R}^D_v\right)\right)}.
	\end{equation}
	Hence, recalling $s=1+b-D\left(1-\frac 1p\right)$ and summing over $k$ yields
	\begin{equation}
		\begin{aligned}
			& \left\| \left\{
			2^{ks} \left\|
			\Delta_{2^k}^x\int_{\mathbb{R}^D} f(x,v) dv
			\right\|_{L^p_x} \right\}_{k=0}^\infty
			\right\|_{\ell^q} \\
			&\hspace{30mm} \leq
			C
			\left\|\left\{
			2^{kb} \left\| \Delta_{2^k}^x g \right\|_{\widetilde L^1\left( \mathbb{R}^D_x; B^\beta_{p,q}\left(\mathbb{R}^D_v\right)\right)}
			\right\}_{k=-1}^\infty\right\|_{\ell^q}\\
			&\hspace{30mm}=
			C \left\| g \right\|_{B^{b,\beta}_{1,p,q}\left(\mathbb{R}^D_x\times\mathbb{R}^D_v\right)},
		\end{aligned}
	\end{equation}
	which concludes the proof of the theorem.
\end{proof}

	\subsection{Velocity averaging in $L^1_xL^p_v$, homogeneous case}

	Here, we explain how the proofs of the previous section in the inhomogeneous case are adapted to the homogeneous case and we give justifications of Theorems \ref{pre averaging lemma p h} and \ref{pre averaging lemma p 2 h}.
	
	Again, for the sake of simplicity as in the previous section and without any loss of generality, we will only consider here the operators $A_\delta^t$ and $B_\delta^t$ as defined for the specific cutoff $\tilde\rho(s)=\delta_{\left\{s=1\right\}}$

	\begin{lem}\label{operators norms dispersive h}
		For every $1\leq p, q\leq \infty$, $\alpha,\beta\in\mathbb{R}$, $k\in\mathbb{N}$ and $t>0$ such that
		\begin{equation}
			\beta + D\left(1-\frac 1p\right) < 1,
		\end{equation}
		it holds that
		\begin{equation}
			\begin{aligned}
				\left\|A_{2^k}^t f\right\|_{L^p_x}
				& \leq {C\over t^{D\left(1-\frac 1p\right)}} \left\| \Delta_{t2^k}^v f \right\|_{L^1_xL^p_v}\\
				& \leq {C\over t^{\alpha+D\left(1-\frac 1p\right)} 2^{k\alpha}} \left\| f \right\|_{\widetilde L^1\left( \mathbb{R}^D_x; \dot B^\alpha_{p,q}\left(\mathbb{R}^D_v\right)\right)},
			\end{aligned}
		\end{equation}
		and
		\begin{equation}
			\begin{aligned}
				\left\|B_{2^k}^tg\right\|_{L^p_x}
				& \leq
				{C\over t^{\beta+D\left(1-\frac 1p\right)} 2^{k\beta}}
				\sum_{j=-\infty}^{j_k+1}
				2^{ -\left( j_k + 1 - j \right) \left( 1 - \beta - D \left( 1 - \frac{1}{p} \right) \right) }
				2^{j\beta} \left\| \Delta_{2^j}^v g \right\|_{L^1_x L^p_v}\\
				& \leq {C\over t^{\beta+D\left(1-\frac 1p\right)} 2^{k\beta}}
				\left\| g \right\|_{\widetilde L^1\left( \mathbb{R}^D_x; \dot B^\beta_{p,q}\left(\mathbb{R}^D_v\right)\right)},
			\end{aligned}
		\end{equation}
		where $j_k\geq 1$ is the largest integer such that $2^{j_k-1} \leq t2^k $ and $C>0$ is independent of $t$ and $2^k$.

		Furthermore, if $\beta + D\left(1-\frac 1p\right) = 1$ and $q=1$, then it still holds that
		\begin{equation}
			\begin{aligned}
				\left\|B_{2^k}^tg\right\|_{L^p_x}
				& \leq
				{C\over t 2^{k\beta}}
				\sum_{j=-\infty}^{j_k+1}
				2^{j\beta} \left\| \Delta_{2^j}^v g \right\|_{L^1_x L^p_v}\\
				& \leq {C\over t 2^{k\beta}}
				\left\| g \right\|_{\widetilde L^1\left( \mathbb{R}^D_x; \dot B^\beta_{p,1}\left(\mathbb{R}^D_v\right)\right)}.
			\end{aligned}
		\end{equation}
	\end{lem}

	\begin{proof}
		We follow precisely the steps of the proof of Lemma \ref{operators norms dispersive}. Thus, we first have, as a consequence of the dispersive properties of the transport operator, that
		\begin{equation}
			\begin{aligned}
				\left\|A_{2^k}^tf(x)\right\|_{L^p_x} &\leq Ct^{-D\left(1-\frac 1p\right)} \left\| \left(\Delta_{t2^k}^v f\right)(x,v) \right\|_{L^1_xL^p_v} \\
				& \leq Ct^{-\alpha-D\left(1-\frac 1p\right)} 2^{-k\alpha} \left\| f(x,v) \right\|_{\widetilde L^1\left( \mathbb{R}^D_x; \dot B^\alpha_{p,q}\left(\mathbb{R}^D_v\right)\right)},
			\end{aligned}
		\end{equation}
		which concludes the estimate for $A_{2^k}^t$.

		As for $B_{2^k}^t$, we first obtain
		\begin{equation}
			\left\|B_{2^k}^tg(x)\right\|_{L^p_x} \leq C\int_0^1 (st)^{-D\left(1-\frac 1p\right)} \left\| \left(\Delta_{st2^k}^v g\right)(x,v) \right\|_{L^1_xL^p_v} ds.
		\end{equation}
		Then, we further split the above integral into small dyadic intervals
		\begin{equation}
				I_j = \left[ {2^{j-1}\over t2^k},  {2^j\over t2^k} \right]
				\quad\text{and}\quad I_{j_k} = \left[ {2^{j_k-1}\over t2^k}, 1 \right],
		\end{equation}
		where $j,j_k\in\mathbb{Z}$, $j\leq j_k$ and $j_k$ is the largest integer such that $2^{j_k-1} \leq t2^k $. Thus, on each dyadic interval $I_j$, the frequency $st2^k$ is between $2^{j-1}$ and $2^{j}$. Therefore, we deduce that
		\begin{equation}
			\begin{aligned}
				\int_{I_j} (st)^{-D \left( 1 - \frac 1 p \right) }
				\left\| \Delta_{st2^k}^v g \right\|_{L^1_x L^p_v} ds
				& =
				\int_{I_j} (st)^{-D \left( 1 - \frac 1 p \right) }
				\left\| \Delta_{\left[2^{j-2},2^{j+1}\right]}^v
				\Delta_{st2^k}^v g \right\|_{L^1_x L^p_v} ds \\
				& \leq C \left|I_j\right| 2^{ -(j-k)D \left( 1 - \frac 1 p \right) }
				\left\| \Delta_{\left[2^{j-2},2^{j+1}\right]}^v g \right\|_{L^1_x L^p_v} \\
				& =\frac Ct
				2^{ \left( j - k \right) \left( 1 - D \left( 1 - \frac{1}{p} \right) \right) }
				\left\| \Delta_{\left[2^{j-2},2^{j+1}\right]}^v g \right\|_{L^1_x L^p_v}.
			\end{aligned}
		\end{equation}
		Thus, on the whole, we have obtained that
		\begin{equation}
			\begin{aligned}
				\left\|B_{2^k}^tg\right\|_{L^p_x}
				& \leq C\int_0^1 (st)^{-D\left(1-\frac 1p\right)} \left\| \Delta_{st2^k}^v g \right\|_{L^1_xL^p_v} ds\\
				& = C\sum_{j=-\infty}^{j_k} \int_{I_j} (st)^{-D\left(1-\frac 1p\right)} \left\| \Delta_{st2^k}^v g \right\|_{L^1_xL^p_v} ds \\
				& \leq {C\over t}
				\sum_{j=-\infty}^{j_k+1}
				2^{ \left( j - k \right) \left( 1 - D \left( 1 - \frac{1}{p} \right) \right) }
				\left\| \Delta_{2^j}^v g \right\|_{L^1_x L^p_v}
				\\
				& \leq C
				t^{ - \beta - D \left( 1 - \frac{1}{p} \right) }2^{-k\beta}
				\sum_{j=-\infty}^{j_k+1}
				2^{ - \left( j_k+1-j \right) \left( 1 - \beta - D \left( 1 - \frac{1}{p} \right) \right) }2^{j\beta}
				\left\| \Delta_{2^j}^v g \right\|_{L^1_x L^p_v}.
			\end{aligned}
		\end{equation}

		It follows that, if $1 - \beta - D\left( 1 - \frac 1 p \right)>0$, a further application of H\"older's inequality to the preceding estimate yields
		\begin{equation}
			\begin{aligned}
				&\left\|B_{2^k}^tg(x)\right\|_{L^p_x}\\
				& \leq C
				t^{ - \beta - D \left( 1 - \frac{1}{p} \right) }
				2^{-k\beta}
				\left\| \left\{ 2^{- j \left( 1 - \beta - D\left( 1 - \frac 1 p \right) \right) } \right\}_{j=0}^\infty \right\|_{\ell^{q'}}
				\left\| g(x,v) \right\|_{\widetilde L^1\left( \mathbb{R}^D_x; \dot B^\beta_{p,q}\left(\mathbb{R}^D_v\right)\right)},
			\end{aligned}
		\end{equation}
		which concludes the proof of the lemma.
	\end{proof}

	\begin{proof}[Proof of Theorem \ref{pre averaging lemma p h}]
		According to Proposition \ref{crucial}, we begin with the following dyadic interpolation formula, for each $k\in\mathbb{N}$,
		\begin{equation}
			\Delta_{2^k}^x\int_{\mathbb{R}^D} f(x,v) dv
			= A_{2^k}^t f (x) + t B_{2^k}^t g (x).
		\end{equation}

		Then, a direct application of Lemma \ref{operators norms dispersive h} yields the estimate, for every $t>0$,
		\begin{equation}\label{high low frequencies h}
			\begin{aligned}
				&\left\|\Delta_{2^k}^x\int_{\mathbb{R}^D} f(x,v) dv\right\|_{L^p_x}\\
				& \leq
				{C\over t^{\alpha+D\left(1-\frac 1p\right)} 2^{k\alpha}} \left(t2^k\right)^\alpha\left\| \Delta_{t2^k}^v f \right\|_{L^1_xL^p_v} \\
				&+
				{C\over t^{\beta-1+D\left(1-\frac 1p\right)} 2^{k\beta}}
				\sum_{j=-\infty}^{j_k+1}
				2^{ -\left( j_k + 1 - j \right) \left( 1 - \beta - D \left( 1 - \frac{1}{p} \right) \right) }
				2^{j\beta} \left\| \Delta_{2^j}^v g \right\|_{L^1_x L^p_v},
			\end{aligned}
		\end{equation}
		where $j_k\in\mathbb{Z}$ is the largest integer such that $2^{j_k-1} \leq t2^k $ and $C>0$ is independent of $t$ and $2^k$.

		Next, optimizing in $t$ for each value of $k$, we fix the interpolation parameter $t$ as
		\begin{equation}
			t_k= \lambda^\frac{1}{1+\alpha-\beta}
			2^{-k\frac{\alpha-\beta}{1+\alpha-\beta}},
		\end{equation}
		where $\lambda=\frac{\left\| f \right\|_{\widetilde L^1\left(\mathbb{R}^D_x;\dot B^\alpha_{p,q}\left(\mathbb{R}^D_v\right)\right)}}{\left\| g \right\|_{\widetilde L^1\left(\mathbb{R}^D_x;\dot B^\beta_{p,q}\left(\mathbb{R}^D_v\right)\right)}}$ (cf. proof of Theorem \ref{pre averaging lemma p} for a heuristic explanation on how to choose this optimal parameter). Note that $1+\alpha-\beta>0$ since $\alpha>-D\left(1-\frac 1p\right)$. Thus, setting $t=t_k$ in \eqref{high low frequencies h}, denoting $ s = \frac{\alpha+D \left( 1 - \frac 1 p \right)}{1+\alpha-\beta}  - D \left( 1 - \frac 1 p \right)  $, noticing that $j_k=\left[1+{\log \lambda\over ( 1+\alpha-\beta )\log 2}+{k\over 1+\alpha-\beta}\right]$ and summing over $k$, yields that
		\begin{equation}\label{besov averaging p h}
			\begin{aligned}
				& \left\| \left\{ 2^{ k s } \left\| \Delta_{2^k}^x \int_{\mathbb{R}^D} f(x,v) dv \right\|_{L^p_x} \right\}_{k=-\infty}^\infty \right\|_{\ell^q} \\
				& \leq
				\frac{C}{\lambda^{\frac{\alpha+D\left(1-\frac 1p\right)}{1+\alpha-\beta}}} \left\| \left\{ \left(t_k2^k\right)^\alpha \left\| \Delta_{t_k2^k}^v f \right\|_{L^1_x L^p_v} \right\}_{k=-\infty}^\infty \right\|_{\ell^q}\\
				& + \frac{C}{\lambda^{\frac{\beta-1+D\left(1-\frac 1p\right)}{1+\alpha-\beta}}} \left\| \left\{ \sum_{j=-\infty}^{ j_k + 1 } 2^{-\left(j_k + 1 - j\right) \left( 1 - \beta - D\left( 1 - \frac 1 p \right) \right) } 2^{j\beta}
				\left\| \Delta_{2^j }^v g \right\|_{L^1_x L^p_v} \right\}_{k=-\infty}^\infty \right\|_{\ell^q}.
			\end{aligned}
		\end{equation}
		Then, it is readily seen that the first term in the right-hand side above is controlled by $\left\| f \right\|_{\widetilde L^1\left(\mathbb{R}^D_x;\dot B^\alpha_{p,q}\left(\mathbb{R}^D_v\right)\right)}$, for $t_k2^k=\left(\lambda 2^{k}\right)^{1\over 1+\alpha-\beta}\rightarrow\pm\infty$ as $k\rightarrow\pm\infty$. Furthermore, writing $ a_j =  2^{ j \beta } \left\| \Delta_{2^j }^v g \right\|_{L^1_x L^p_v} $ and $ b_j = 2^{ -j \left( 1 -\beta - D \left( 1 - \frac 1 p \right) \right)}  \mathbb{1}_{ \left\{ j \geq 0 \right\} } $, for all $j\in\mathbb{Z}$, we see that $a=\left\{a_j\right\}_{j\in\mathbb{Z}} \in \ell^q$, $b=\left\{b_j\right\}_{j\in\mathbb{Z}} \in \ell^1$ and that
		\begin{equation}
			\begin{aligned}
			& \left\| \left\{ \sum_{j=-\infty}^{ j_k + 1 } 2^{-\left(j_k + 1 - j\right) \left( 1 - \beta - D\left( 1 - \frac 1 p \right) \right) } 2^{j\beta}
			\left\| \Delta_{2^j }^v g \right\|_{L^1_x L^p_v}\right\}_{k=-\infty}^\infty \right\|_{\ell^q} \\
			& =  \left\| \left\{\left( a*b \right)_{j_k+1}\right\}_{k=-\infty}^\infty \right\|_{\ell^q}
			\leq  \left\| a \right\|_{\ell^q} \left\| b \right\|_{\ell^1} 
			\leq  C
			\left\| g \right\|_{\widetilde L^1\left(\mathbb{R}^D_x;\dot B^\beta_{p,q}\left(\mathbb{R}^D_v\right)\right)}.
			\end{aligned}
		\end{equation}
		Thus, on the whole, combining the preceding estimates with \eqref{besov averaging p h}, we deduce that
		\begin{equation}
			\begin{aligned}
				& \left\| \left\{ 2^{ k s } \left\| \Delta_{2^k}^x \int_{\mathbb{R}^D} f(x,v) dv \right\|_{L^p_x} \right\}_{k=-\infty}^\infty \right\|_{\ell^q} \\
				& \hspace{40mm} \leq 
				C \left\| f \right\|_{\widetilde L^1\left(\mathbb{R}^D_x;\dot B^\alpha_{p,q}\left(\mathbb{R}^D_v\right)\right)}^{\frac{1-\beta-D\left(1-\frac 1p\right)}{1+\alpha-\beta}}
				\times
				\left\| g \right\|_{\widetilde L^1\left(\mathbb{R}^D_x;\dot B^\beta_{p,q}\left(\mathbb{R}^D_v\right)\right)}^{\frac{\alpha+D\left(1-\frac 1p\right)}{1+\alpha-\beta}},
			\end{aligned}
		\end{equation}
		which concludes the proof of the theorem.
	\end{proof}

	\begin{proof}[Proof of Theorem \ref{pre averaging lemma p 2 h}]
		As in the proof of Theorem \ref{pre averaging lemma p h}, we begin with the dyadic decomposition provided by Proposition \ref{crucial}, for each $k\in\mathbb{N}$,
		\begin{equation}
				\Delta_{2^k}^x\int_{\mathbb{R}^D} f(x,v) dv
				= A_{2^k}^t f (x) + t B_{2^k}^t g (x),
		\end{equation}
		and we utilize Lemma \ref{operators norms dispersive h} to obtain, in virtue of property \eqref{crucial 3} from Proposition \ref{crucial}, that
		\begin{equation}
			\begin{aligned}
				\left\|
				\Delta_{2^k}^x\int_{\mathbb{R}^D} f(x,v) dv
				\right\|_{L^p_x}
				& \leq C
				{1\over t^{\alpha+D\left(1-\frac 1p\right)} 2^{k\alpha}} \left\| \Delta_{\left[2^{k-1},2^{k+1}\right]}^x f \right\|_{\widetilde L^1\left( \mathbb{R}^D_x; \dot B^\alpha_{p,q}\left(\mathbb{R}^D_v\right)\right)}\\
				& + 
				C
				{t^{1-\beta-D\left(1-\frac 1p\right)}\over 2^{k\beta}}
				\left\| \Delta_{\left[2^{k-1},2^{k+1}\right]}^x g \right\|_{\widetilde L^1\left( \mathbb{R}^D_x; \dot B^\beta_{p,q}\left(\mathbb{R}^D_v\right)\right)}.
			\end{aligned}
		\end{equation}

		Next, optimizing in $t$ for each value of $k$, we fix the interpolation parameter $t$ as
		\begin{equation}
			t_k= \lambda^\frac{1}{1+\alpha-\beta}
			2^{-k\frac{(\alpha-\beta)+(a-b)}{1+\alpha-\beta}},
		\end{equation}
		where $\lambda=\frac{\left\| f \right\|_{\dot B^{a,\alpha}_{1,p,q}\left(\mathbb{R}^D_x\times\mathbb{R}^D_v\right)}}{\left\| g \right\|_{\dot B^{b,\beta}_{1,p,q}\left(\mathbb{R}^D_x\times\mathbb{R}^D_v\right)}}$ (cf. proof of Theorem \ref{pre averaging lemma p} for a heuristic explanation on how to choose this optimal parameter). Note that $1+\alpha-\beta>0$ since $\alpha>-D\left(1-\frac 1p\right)$ and that this choice is independent of $1\leq p, q \leq \infty$. Therefore, denoting $ s = \left(1+b-a\right)\frac{\alpha+ D \left( 1 - \frac 1 p \right)}{1+\alpha-\beta} + a - D \left( 1 - \frac 1 p \right) $, we find that
		\begin{equation}
			\begin{aligned}
				2^{ks} \left\|
				\Delta_{2^k}^x\int_{\mathbb{R}^D} f(x,v) dv
				\right\|_{L^p_x}
				& \leq \frac{C}{\lambda^{\frac{\alpha+D\left(1-\frac 1p\right)}{1+\alpha-\beta}}}
				2^{ka} \left\| \Delta_{\left[2^{k-1},2^{k+1}\right]}^x f \right\|_{\widetilde L^1\left( \mathbb{R}^D_x; \dot B^\alpha_{p,q}\left(\mathbb{R}^D_v\right)\right)}\\
				& + \frac{C}{\lambda^{\frac{\beta-1+D\left(1-\frac 1p\right)}{1+\alpha-\beta}}}
				2^{kb}
				\left\| \Delta_{\left[2^{k-1},2^{k+1}\right]}^x g \right\|_{\widetilde L^1\left( \mathbb{R}^D_x; \dot B^\beta_{p,q}\left(\mathbb{R}^D_v\right)\right)}.
			\end{aligned}
		\end{equation}

		Finally, summing over $k$, we obtain
		\begin{equation}
			\begin{aligned}
				& \left\| \left\{
				2^{ks} \left\|
				\Delta_{2^k}^x\int_{\mathbb{R}^D} f(x,v) dv
				\right\|_{L^p_x} \right\}_{k=-\infty}^\infty
				\right\|_{\ell^q} \\
				&\hspace{30mm} \leq \frac{C}{\lambda^{\frac{\alpha+D\left(1-\frac 1p\right)}{1+\alpha-\beta}}}
				\left\| \left\{
				2^{ka} \left\| \Delta_{2^k}^x f \right\|_{\widetilde L^1\left( \mathbb{R}^D_x; \dot B^\alpha_{p,q}\left(\mathbb{R}^D_v\right)\right)} \right\}_{k=-\infty}^\infty
				\right\|_{\ell^q} \\
				&\hspace{30mm} +
				\frac{C}{\lambda^{\frac{\beta-1+D\left(1-\frac 1p\right)}{1+\alpha-\beta}}}
				\left\|\left\{
				2^{kb} \left\| \Delta_{2^k}^x g \right\|_{\widetilde L^1\left( \mathbb{R}^D_x; \dot B^\beta_{p,q}\left(\mathbb{R}^D_v\right)\right)}
				\right\}_{k=-\infty}^\infty\right\|_{\ell^q}\\
				&\hspace{30mm}=C
				\left\| f \right\|_{\dot B^{a,\alpha}_{1,p,q}\left(\mathbb{R}^D_x\times\mathbb{R}^D_v\right)}^{\frac{1-\beta-D\left(1-\frac 1p\right)}{1+\alpha-\beta}}
				\times
				\left\| g \right\|_{\dot B^{b,\beta}_{1,p,q}\left(\mathbb{R}^D_x\times\mathbb{R}^D_v\right)}^{\frac{\alpha+D\left(1-\frac 1p\right)}{1+\alpha-\beta}},
			\end{aligned}
		\end{equation}
		which concludes the proof of the theorem.
	\end{proof}

\subsection{The classical $L^2_xL^2_v$ case revisited}

We provide now the proofs for the classical Hilbertian case of velocity averaging.

\begin{lem}\label{operators norms}
	Let $\phi(v)\in C_0^\infty\left(\mathbb{R}^D\right)$. For every $1\leq q\leq \infty$, $\alpha,\beta\in\mathbb{R}$, $\beta\neq \frac 12$, $k\in\mathbb{N}$ and $t\geq 2^{-k}$, it holds that
	\begin{equation}
		\begin{aligned}
			\left\|A_0^1(f\phi)\right\|_{L^2_x}
			&\leq
			C
			\left\| \Delta_0^x f \right\|_{\widetilde L^2 \left( \mathbb{R}^D_x ; B^{\alpha}_{2,q} \left(\mathbb{R}^D_v\right) \right)}
			\leq
			C
			\left\| f \right\|_{\widetilde L^2 \left( \mathbb{R}^D_x ; B^{\alpha}_{2,q} \left(\mathbb{R}^D_v\right) \right)},\\
			\left\|A_{2^k}^t(f\phi)\right\|_{L^2_x}
			&\leq
			\frac{C}{\left(t2^k\right)^{\alpha+\frac 12}}
			\left\| \Delta_{2^k}^x f \right\|_{\widetilde L^2 \left( \mathbb{R}^D_x ; B^{\alpha}_{2,q} \left(\mathbb{R}^D_v\right) \right)}
			\leq
			\frac{C}{\left(t2^k\right)^{\alpha+\frac 12}}
			\left\| f \right\|_{\widetilde L^2 \left( \mathbb{R}^D_x ; B^{\alpha}_{2,q} \left(\mathbb{R}^D_v\right) \right)},
		\end{aligned}
	\end{equation}
	and
	\begin{equation}
		\begin{aligned}
			\left\|B_0^1(g\phi)\right\|_{L^2_x}
			&\leq
			C
			\left\| \Delta_0^x g \right\|_{\widetilde L^2 \left( \mathbb{R}^D_x ; B^{\beta}_{2,q} \left(\mathbb{R}^D_v\right) \right)}
			\leq
			C
			\left\| g \right\|_{\widetilde L^2 \left( \mathbb{R}^D_x ; B^{\beta}_{2,q} \left(\mathbb{R}^D_v\right) \right)},\\
			\left\|B_{2^k}^t(g\phi)\right\|_{L^2_x}
			&\leq
			\frac{C}{\left(t2^k\right)^{\gamma+\frac 12}}
			\left\| \Delta_{2^k}^x g \right\|_{\widetilde L^2 \left( \mathbb{R}^D_x ; B^{\beta}_{2,q} \left(\mathbb{R}^D_v\right) \right)}
			\leq
			\frac{C}{\left(t2^k\right)^{\gamma+\frac 12}}
			\left\| g \right\|_{\widetilde L^2 \left( \mathbb{R}^D_x ; B^{\beta}_{2,q} \left(\mathbb{R}^D_v\right) \right)},
		\end{aligned}
	\end{equation}
	where $\gamma=\min\left\{\beta,\frac 12\right\}$ and $C>0$ is independent of $t$ and $2^k$.
	
	Furthermore, if $\beta=\frac 12$, then, it holds that
	\begin{equation}
		\begin{aligned}
			\left\|B_0^1(g\phi)\right\|_{L^2_x}
			&\leq
			C
			\left\| \Delta_0^x g \right\|_{\widetilde L^2 \left( \mathbb{R}^D_x ; B^{\beta}_{2,q} \left(\mathbb{R}^D_v\right) \right)}
			\leq
			C
			\left\| g \right\|_{\widetilde L^2 \left( \mathbb{R}^D_x ; B^{\beta}_{2,q} \left(\mathbb{R}^D_v\right) \right)},\\
			\left\|B_{2^k}^t(g\phi)\right\|_{L^2_x}
			&\leq
			\frac{C}{t2^k}\log\left(1+t2^k\right)^\frac{1}{q'}
			\left\| \Delta_{2^k}^x g \right\|_{\widetilde L^2 \left( \mathbb{R}^D_x ; B^{\beta}_{2,q} \left(\mathbb{R}^D_v\right) \right)}\\
			&\leq
			\frac{C}{t2^k}\log\left(1+t2^k\right)^\frac{1}{q'}
			\left\| g \right\|_{\widetilde L^2 \left( \mathbb{R}^D_x ; B^{\beta}_{2,q} \left(\mathbb{R}^D_v\right) \right)}.
		\end{aligned}
	\end{equation}
\end{lem}

The above lemma will be obtained as a consequence of the following Lemmas \ref{lem lambda strip} and \ref{lem lambda strip low}.

\begin{lem}\label{lem lambda strip}
	For any $\chi\in\mathcal{S}\left(\mathbb{R}\right)$ and $\rho\in C^\infty\left(\mathbb{R}\right)$ such that $\chi\equiv 1$ in a neighborhood of the origin and $\hat\rho$ is compactly supported and bounded pointwise, and for each $\lambda\geq 1$, $R>0$ and $t\geq\frac 1R$, let $h_1(\eta,v)$ and $h_2(\eta,v)$ be defined by
	\begin{equation}
		\begin{aligned}
			h_1(\eta,v) & =
			\chi\left( \left|v-\frac{v\cdot\eta}{|\eta|}\frac{\eta}{|\eta|}\right| \right)
			\rho\left(\lambda\frac{v\cdot\eta}{|\eta|}\right),\\
			h_2(\eta,v) & =
			\mathbb{1}_{\left\{\frac R2 \leq |\eta| \leq 2R\right\}} \chi\left( \left|v-\frac{v\cdot\eta}{|\eta|}\frac{\eta}{|\eta|}\right| \right)
			\rho\left(t v\cdot\eta \right),
		\end{aligned}
	\end{equation}
	for all $\eta,v\in\mathbb{R}^D$.
	
	Then, for each $k\in\mathbb{N}$, $\left\|\Delta_0^v h_1\right\|_{L^2_v}$ and $\left\|\Delta_{2^k}^v h_1\right\|_{L^2_v}$ are independent of $\eta$ and, for every $1\leq q\leq\infty$ and $\alpha\neq \frac 12$, it holds that
	\begin{equation}\label{lambda strip}
		\begin{gathered}
			\left\|h_1\right\|_{B^{-\alpha}_{2,q}\left(\mathbb{R}^D_v\right)}
			\leq \frac{C}{\lambda^{\gamma+\frac 12}},\\
			\left\|h_2\right\|_{\widetilde L^\infty \left( \mathbb{R}^D_\eta ; B^{-\alpha}_{2,q} \left(\mathbb{R}^D_v\right) \right)}
			\leq \frac{C}{\left(tR\right)^{\gamma+\frac 12}},
		\end{gathered}
	\end{equation}
	where $\gamma=\min\left\{\alpha,\frac 12\right\}$ and the constant $C>0$ is independent of $\lambda$, $t$ and $R$.
	
	Furthermore, if $\alpha=\frac 12$, then it holds that
	\begin{equation}\label{lambda strip log}
		\begin{gathered}
			\left\|h_1\right\|_{B^{-\alpha}_{2,q}\left(\mathbb{R}^D_v\right)}
			\leq \frac{C}{\lambda}\log\left(1+\lambda\right)^\frac{1}{q'},\\
			\left\|h_2\right\|_{\widetilde L^\infty \left( \mathbb{R}^D_\eta ; B^{-\alpha}_{2,q} \left(\mathbb{R}^D_v\right) \right)}
			\leq \frac{C}{tR}\log\left(1+tR\right)^\frac{1}{q'}.
		\end{gathered}
	\end{equation}
	
	Finally, if the support of $\hat\rho$ doesn't contain the origin, then, for every $\alpha\in\mathbb{R}$, it holds that
	\begin{equation}\label{lambda strip support}
		\begin{gathered}
			\left\|h_1\right\|_{B^{-\alpha}_{2,q}\left(\mathbb{R}^D_v\right)}
			\leq \frac{C}{\lambda^{\alpha+\frac 12}},\\
			\left\|h_2\right\|_{\widetilde L^\infty \left( \mathbb{R}^D_\eta ; B^{-\alpha}_{2,q} \left(\mathbb{R}^D_v\right) \right)}
			\leq \frac{C}{\left(tR\right)^{\alpha+\frac 12}}.
		\end{gathered}
	\end{equation}
\end{lem}

\begin{lem}\label{lem lambda strip low}
	For any $\chi\in\mathcal{S}\left(\mathbb{R}\right)$ and $\rho\in C^\infty\left(\mathbb{R}\right)$ such that $\chi\equiv 1$ in a neighborhood of the origin and $\hat\rho$ is compactly supported and bounded pointwise, let $h(\eta,v)$ be defined by
	\begin{equation}
		h(\eta,v) =
		\mathbb{1}_{\left\{ |\eta| \leq 1\right\}} \chi\left( \left|v\right| \right)
		\rho\left(v\cdot\eta \right),
	\end{equation}
	for all $\eta,v\in\mathbb{R}^D$.
	
	Then, for every $1\leq q\leq\infty$ and $\alpha\in\mathbb{R}$, it holds that
	\begin{equation}
		h(\eta,v) \in \widetilde L^\infty \left( \mathbb{R}^D_\eta ; B^{-\alpha}_{2,q} \left(\mathbb{R}^D_v\right) \right).
	\end{equation}
\end{lem}

We defer the proofs of Lemmas \ref{lem lambda strip} and \ref{lem lambda strip low} and proceed now to the proof of Lemma \ref{operators norms}.

\begin{proof}[Proof of Lemma \ref{operators norms}]
	First, notice that, for any $\rho(s)\in\mathcal{S}\left(\mathbb{R}\right)$, it holds
	\begin{equation}
		\mathcal{F}\left({\rho(s)-\rho(0)\over is}\right)(r)=
		2\pi \rho(0)\mathbb{1}_{\left\{r\geq 0\right\}}-\int_{-\infty}^r\hat\rho(\sigma)d\sigma,
	\end{equation}
	in the sense of tempered distributions. Indeed, one easily obtains by duality, for any test function $\varphi(r)\in\mathcal{S}\left(\mathbb{R}\right)$, that
	\begin{equation}
		\begin{aligned}
			\int_\mathbb{R}\mathcal{F}\left({\rho(s)-\rho(0)\over is}\right)(r)\varphi(r)dr&=
			\int_{\mathbb{R}}\frac{1}{2\pi}\int_{\mathbb{R}}{e^{is\sigma}-1\over is}\hat\rho(\sigma)d\sigma \hat\varphi(s) ds
			\\
			&=
			\int_{\mathbb{R}}\frac{1}{2\pi}\int_{\mathbb{R}}\left[\int_0^\sigma e^{ ist}dt\right]\hat\rho(\sigma)d\sigma \hat\varphi(s) ds
			\\
			&=
			\int_{\mathbb{R}}\left[\int_0^\sigma \varphi(t)dt\right]\hat\rho(\sigma)d\sigma
			\\
			&=
			\int_0^\infty\int_t^\infty \varphi(t) \hat\rho(\sigma)d\sigma dt
			-
			\int_{-\infty}^0\int_{-\infty}^t \varphi(t) \hat\rho(\sigma)d\sigma dt
			\\
			&=
			\int_0^\infty\int_\mathbb{R} \hat\rho(\sigma)d\sigma \varphi(t) dt
			-
			\int_{\mathbb{R}}\int_{-\infty}^t \hat\rho(\sigma)d\sigma \varphi(t) dt
			\\
			&=
			\int_{\mathbb{R}}\left[
			2\pi\rho(0)\mathbb{1}_{\left\{t\geq 0\right\}}-\int_{-\infty}^t\hat\rho(\sigma)d\sigma
			\right]\varphi(t)dt.
		\end{aligned}
	\end{equation}

	In the particular setting of the present lemma, we impose that $\hat\rho$ be compactly supported, supported away from the origin and that $\rho(0)=1$. It then follows that $\tau(s)=\frac{1-\rho(s)}{i s}$ is smooth near the origin, that
	\begin{equation}
		\hat\tau(r)=\int_\mathbb{R}e^{-i sr} \frac{1-\rho(s)}{i s} ds
		=
		\int_{-\infty}^r \hat\rho(\sigma) d\sigma - 2\pi \mathbb{1}_{\left\{r\geq 0\right\}}
	\end{equation}
	is compactly supported as well and that $|\hat\tau(r)|$ is bounded pointwise. Notice, however, that the origin is always contained in the support of $\hat\tau$.
	
	Consider now $R>0$ so that $\mathrm{supp}\phi(v)\subset\left\{|v|\leq R\right\}$ and let $\chi(r)\in C_0^\infty\left(\mathbb{R}\right)$ be a cutoff function satisfying $\mathbb{1}_{\left\{|r|\leq R\right\}}\leq \chi(r)\leq\mathbb{1}_{\left\{|r|\leq 2R\right\}}$. Then, according to the identities \eqref{crucial 1} from Proposition \ref{crucial} and Lemmas \ref{lem lambda strip} and \ref{lem lambda strip low}, for any $t\geq 2^{-k}$, we obtain
	\begin{equation}
		\begin{aligned}
			\left\|A_{0}^1(f\phi)\right\|_{L^2_x} &=\frac{1}{(2\pi)^\frac{D}{2}}
			\left\| \int_{\mathbb{R}^D} \psi\left(\eta\right) \hat f (\eta,v) \phi(v)
			\chi\left( \left|v\right| \right)
			\rho\left(v\cdot\eta\right) dv \right\|_{L^2_\eta}\\
				&\leq
				\left\| \Delta_{0}^xf \phi \right\|_{\widetilde L^2 \left( \mathbb{R}^D_x ; B^{\alpha}_{2,q} \left(\mathbb{R}^D_v\right) \right)}
				\left\|
				\mathbb{1}_{\left\{|\eta|\leq 1\right\}}
				\chi\left( \left|v\right| \right)
				\rho\left(v\cdot\eta\right) \right\|_{\widetilde L^\infty \left( \mathbb{R}^D_\eta ; B^{-\alpha}_{2,q'} \left(\mathbb{R}^D_v\right) \right)}\\
			& \leq
			C
			\left\| \Delta_{0}^x f \phi \right\|_{\widetilde L^2 \left( \mathbb{R}^D_x ; B^{\alpha}_{2,q} \left(\mathbb{R}^D_v\right) \right)}
			\leq
			C
			\left\| f \phi \right\|_{\widetilde L^2 \left( \mathbb{R}^D_x ; B^{\alpha}_{2,q} \left(\mathbb{R}^D_v\right) \right)},\\
			\left\|A_{2^k}^t(f\phi)\right\|_{L^2_x} &=\frac{1}{(2\pi)^\frac{D}{2}}
			\left\| \int_{\mathbb{R}^D} \varphi\left(\frac{\eta}{2^k}\right) \hat f (\eta,v) \phi(v)
			\chi\left( \left|v-\frac{v\cdot\eta}{|\eta|}\frac{\eta}{|\eta|}\right| \right)
			\rho\left(tv\cdot\eta\right) dv \right\|_{L^2_\eta}\\
				&\leq
				\left\| \Delta_{2^k}^xf \phi \right\|_{\widetilde L^2 \left( \mathbb{R}^D_x ; B^{\alpha}_{2,q} \left(\mathbb{R}^D_v\right) \right)}\\
				&\times\left\|
				\mathbb{1}_{\left\{2^{k-1}\leq |\eta|\leq 2^{k+1}\right\}}
				\chi\left( \left|v-\frac{v\cdot\eta}{|\eta|}\frac{\eta}{|\eta|}\right| \right)
				\rho\left(tv\cdot\eta\right) \right\|_{\widetilde L^\infty \left( \mathbb{R}^D_\eta ; B^{-\alpha}_{2,q'} \left(\mathbb{R}^D_v\right) \right)}\\
			& \leq
			\frac{C}{\left(t2^k\right)^{\alpha+\frac 12}}
			\left\| \Delta_{2^k}^x f \phi \right\|_{\widetilde L^2 \left( \mathbb{R}^D_x ; B^{\alpha}_{2,q} \left(\mathbb{R}^D_v\right) \right)}\\
			& \leq
			\frac{C}{\left(t2^k\right)^{\alpha+\frac 12}}
			\left\| f \phi \right\|_{\widetilde L^2 \left( \mathbb{R}^D_x ; B^{\alpha}_{2,q} \left(\mathbb{R}^D_v\right) \right)}
		\end{aligned}
	\end{equation}
	and, similarly,
	\begin{equation}
		\begin{aligned}
			\left\|B_{0}^1(g\phi)\right\|_{L^2_x}
			&\leq
			C
			\left\| \Delta_{0}^x g \phi \right\|_{\widetilde L^2 \left( \mathbb{R}^D_x ; B^{\beta}_{2,q} \left(\mathbb{R}^D_v\right) \right)}
			\leq
			C
			\left\| g \phi \right\|_{\widetilde L^2 \left( \mathbb{R}^D_x ; B^{\beta}_{2,q} \left(\mathbb{R}^D_v\right) \right)},\\
			\left\|B_{2^k}^t(g\phi)\right\|_{L^2_x}
			&\leq
			\frac{C}{t2^k}\log\left(1+tR\right)^\frac{1}{q'}
			\left\| \Delta_{2^k}^x g \phi \right\|_{\widetilde L^2 \left( \mathbb{R}^D_x ; B^{\frac 12}_{2,q} \left(\mathbb{R}^D_v\right) \right)}\\
			& \leq
			\frac{C}{t2^k}\log\left(1+tR\right)^\frac{1}{q'}
			\left\| g \phi \right\|_{\widetilde L^2 \left( \mathbb{R}^D_x ; B^{\frac 12}_{2,q} \left(\mathbb{R}^D_v\right) \right)},
		\end{aligned}
	\end{equation}
	and, if $\beta\neq \frac 12$,
	\begin{equation}
		\begin{aligned}
			\left\|B_{2^k}^t(g\phi)\right\|_{L^2_x}
			&\leq
			\frac{C}{\left(t2^k\right)^{\gamma+\frac 12}}
			\left\| \Delta_{2^k}^x g \phi \right\|_{\widetilde L^2 \left( \mathbb{R}^D_x ; B^{\beta}_{2,q} \left(\mathbb{R}^D_v\right) \right)}\\
			& \leq
			\frac{C}{\left(t2^k\right)^{\gamma+\frac 12}}
			\left\| g \phi \right\|_{\widetilde L^2 \left( \mathbb{R}^D_x ; B^{\beta}_{2,q} \left(\mathbb{R}^D_v\right) \right)},
		\end{aligned}
	\end{equation}
	where $\gamma=\min\left\{\beta,\frac 12\right\}$ and $C>0$ is independent of $t$ and $2^k$.
	
	Then, the conclusion of the lemma easily follows from a straightforward application of the methods of paradifferential calculus to the products $f(x,v)\phi(v)$ and $g(x,v)\phi(v)$. For the sake of completeness, we have reproduced in Lemma \ref{paradifferential} the precise estimates that are required here in order to conclude the proof.
\end{proof}

\begin{proof}[Proof of Lemma \ref{lem lambda strip}]
	For each $\eta\in\mathbb{R}^D$, considering an orthogonal transformation $R_\eta:\mathbb{R}^D\to\mathbb{R}^D$ such that the unit vector $\frac{\eta}{|\eta|}$ is mapped onto $\left(0,\dots,0,1\right)$ and writing $v'=\left(v_1,\dots,v_{D-1}\right)$ so that $v=\left(v',v_D\right)$, it holds that
	\begin{equation}
		\begin{aligned}
			\mathcal{F}_v h_1(\eta,\xi)
			&=
			\int_{\mathbb{R}^D} e^{-i v\cdot\left(R_\eta \xi\right)}
			\chi\left( \left|v'\right| \right) \rho\left(\lambda v_D\right)
			dv\\
			&=
			\pi\left(\left(R_\eta\xi\right)_1,\ldots,\left(R_\eta\xi\right)_{D-1}\right)
			\frac{1}{\lambda}\hat\rho\left(\frac{1}{\lambda}\left(R_\eta\xi\right)_D\right)
		\end{aligned}
	\end{equation}
	and
	\begin{equation}
		\begin{aligned}
			\mathcal{F}_v h_2(\eta,\xi)
			&= \mathbb{1}_{\left\{\frac R2 \leq |\eta| \leq 2R\right\}}
			\int_{\mathbb{R}^D} e^{-i v\cdot\left(R_\eta \xi\right)}
			\chi\left( \left|v'\right| \right) \rho\left(t|\eta| v_D\right)
			dv\\
			&= \mathbb{1}_{\left\{\frac R2 \leq |\eta| \leq 2R\right\}}
			\pi\left(\left(R_\eta\xi\right)_1,\ldots,\left(R_\eta\xi\right)_{D-1}\right)
			\frac{1}{t|\eta|}\hat\rho\left(\frac{1}{t|\eta|}\left(R_\eta\xi\right)_D\right),
		\end{aligned}
	\end{equation}
	where $\pi\left(\xi'\right)=\int_{\mathbb{R}^{D-1}}e^{- i \xi'\cdot v'}\chi\left(\left|v'\right|\right)dv'\in\mathcal{S}\left(\mathbb{R}^{D-1}\right)$.
	
	Consequently, by Plancherel's theorem,
	\begin{equation}
		\begin{aligned}
			\left\| \Delta_{0}^v h_1 \right\|_{L^2_v}&=
			\frac{1}{\lambda(2\pi)^\frac{D}{2}}
			\left\|
			\psi\left(\xi\right)
			\pi\left(\xi_1,\ldots,\xi_{D-1}\right)
			\hat\rho\left(\frac{1}{\lambda} \xi_D \right)
			\right\|_{L^2_\xi},\\
			\left\| \Delta_{2^k}^v h_1 \right\|_{L^2_v}&=
			\frac{1}{\lambda(2\pi)^\frac{D}{2}}
			\left\|
			\varphi\left(\frac{\xi}{2^k}\right)
			\pi\left(\xi_1,\ldots,\xi_{D-1}\right)
			\hat\rho\left(\frac{1}{\lambda} \xi_D \right)
			\right\|_{L^2_\xi}
		\end{aligned}
	\end{equation}
	and
	\begin{equation}
		\begin{aligned}
			\left\| \Delta_{0}^v h_2 \right\|_{L^2_v}&=
			\mathbb{1}_{\left\{\frac R2 \leq |\eta| \leq 2R\right\}} \frac{1}{t|\eta|(2\pi)^\frac{D}{2}}
			\left\|
			\psi\left(\xi\right)
			\pi\left(\xi_1,\ldots,\xi_{D-1}\right)
			\hat\rho\left(\frac{1}{t|\eta|} \xi_D \right)
			\right\|_{L^2_\xi},\\
			\left\| \Delta_{2^k}^v h_2 \right\|_{L^2_v}&=
			\mathbb{1}_{\left\{\frac R2 \leq |\eta| \leq 2R\right\}} \frac{1}{t|\eta|(2\pi)^\frac{D}{2}}
			\left\|
			\varphi\left(\frac{\xi}{2^k}\right)
			\pi\left(\xi_1,\ldots,\xi_{D-1}\right)
			\hat\rho\left(\frac{1}{t|\eta|} \xi_D \right)
			\right\|_{L^2_\xi}.
		\end{aligned}
	\end{equation}
	
	Recall now that $\mathrm{supp}\psi\left(\xi\right) \subset \left\{ \left|\xi\right|\leq 1 \right\}$ and $\mathrm{supp}\varphi\left(\frac{\xi}{2^k}\right) \subset \left\{2^{k-1}\leq\left|\xi\right|\leq 2^{k+1} \right\}$. In particular, for every $k\in\mathbb{N}$, we have that
	\begin{equation}
		\begin{aligned}
			\varphi\left(\frac{\xi}{2^k}\right) & = \varphi\left(\frac{\xi}{2^k}\right) \left( \mathbb{1}_{\left\{\left|\xi'\right|\geq 2^{k-2}\right\}} + \mathbb{1}_{\left\{\left|\xi'\right|<2^{k-2}\right\}} \right)\\
			& \leq
			\mathbb{1}_{\left\{ 2^{k-2} \leq \left|\xi'\right| \leq 2^{k+1}\right\}}
			\mathbb{1}_{\left\{\left|\xi_D\right| \leq 2^{k+1}\right\}}
			+
			\mathbb{1}_{\left\{ \left|\xi'\right| < 2^{k-2}\right\}}
			\mathbb{1}_{\left\{ 2^{k-2} \leq \left|\xi_D\right| \leq 2^{k+1}\right\}}.
		\end{aligned}
	\end{equation}
	Hence, we deduce that
	\begin{equation}\label{lambda strip 1}
			\left\| \Delta_{0}^v h_1 \right\|_{L^2_v}\leq \frac C\lambda,\qquad
			\left\| \Delta_{0}^v h_2 \right\|_{L^2_v}\leq \frac C{tR},
	\end{equation}
	and, utilizing that $\pi\left(\xi'\right)$ decays faster than any polynomial and assuming without any loss of generality that $\hat\rho(r)$ is supported inside $\left\{|r|\leq 2\right\}$, we further obtain that
	\begin{equation}\label{support 1}
		\begin{aligned}
			\left\| \Delta_{2^k}^v h_1 \right\|_{L^2_v}
			&\leq \frac{1}{\lambda}
			\left\|
			\mathbb{1}_{\left\{ 2^{k-2} \leq \left|\xi'\right| \leq 2^{k+1}\right\}}
			\pi\left(\xi'\right)
			\right\|_{L^2_{\xi'}}
			\left\|
			\mathbb{1}_{\left\{\left|\xi_D\right| \leq 2^{k+1}\right\}}
			\hat\rho\left(\frac{1}{\lambda} \xi_D\right)
			\right\|_{L^2_{\xi_D}}\\
			&+
			\frac{1}{\lambda}
			\left\|
			\mathbb{1}_{\left\{ \left|\xi'\right| < 2^{k-2}\right\}}
			\pi\left(\xi'\right)
			\right\|_{L^2_{\xi'}}
			\left\|
			\mathbb{1}_{\left\{ 2^{k-2} \leq \left|\xi_D\right| \leq 2^{k+1}\right\}}
			\hat\rho\left(\frac{1}{\lambda} \xi_D\right)
			\right\|_{L^2_{\xi_D}}\\
			& \leq
			\frac{2^{\frac{k}{2}}}{\lambda}
			\left\|
			\mathbb{1}_{\left\{ 2^{k-2} \leq \left|\xi'\right| \leq 2^{k+1}\right\}}
			\pi\left(\xi'\right)
			\right\|_{L^2_{\xi'}}
			\left\|
			\mathbb{1}_{\left\{\left|\xi_D\right| \leq 2\right\}}
			\hat\rho\left(\frac{2^k}{\lambda} \xi_D\right)
			\right\|_{L^2_{\xi_D}}\\
			&+
			C \frac{2^\frac{k}{2}}{\lambda}
			\left\|
			\mathbb{1}_{\left\{ \frac{1}{4} \leq \left|\xi_D\right| \leq 2 \right\}}
			\hat\rho\left(\frac{2^k}{\lambda} \xi_D\right)
			\right\|_{L^2_{\xi_D}} \\
			&\leq
			C\left(
			\frac{2^{\frac k2}}{\lambda\left(1+2^{k\left( 1 - \alpha\right)}\right)}
			+
			\frac{2^\frac{k}{2}}{\lambda}
			\mathbb{1}_{\left\{ \frac{2^{k}}{\lambda} \leq 8 \right\}}
			\right)
		\end{aligned}
	\end{equation}
	and, similarly, that
	\begin{equation}\label{support 2}
		\begin{aligned}
			\left\| \Delta_{2^k}^v h_2 \right\|_{L^2_v}
			&\leq C \mathbb{1}_{\left\{\frac R2 \leq |\eta| \leq 2R\right\}}
			\left(
			\frac{2^{\frac k2}}{t|\eta|\left(1+2^{k\left( 1 - \alpha\right)}\right)}
			+
			\frac{2^\frac{k}{2}}{t|\eta|}
			\mathbb{1}_{\left\{ \frac{2^{k}}{t|\eta|} \leq 8 \right\}}
			\right)\\
			&\leq C
			\left(
			\frac{2^{\frac k2}}{tR\left(1+2^{k\left( 1 - \alpha\right)}\right)}
			+
			\frac{2^\frac{k}{2}}{tR}
			\mathbb{1}_{\left\{ \frac{2^{k}}{tR} \leq 16 \right\}}
			\right).
		\end{aligned}
	\end{equation}
	
	Then, recalling the geometric sum formula $\sum_{k=0}^n x^k=\frac{x^{n+1}-1}{x-1}$, valid for any $x\neq 1$, it follows that, in the case $1\leq q<\infty$ and $\alpha\neq\frac 12$,
	\begin{equation}\label{lambda strip 2}
		\begin{aligned}
			& \left\|\left\{
			2^{-k\alpha}
			\left\| \Delta_{2^k}^v h_1 \right\|_{L^2_v}
			\right\}_{k=0}^\infty\right\|_{\ell^q}\\
			&\hspace{14mm} \leq
			\frac C\lambda \left(
			\left\|\left\{
			\frac{2^{k\left(\frac 12-\alpha\right)}}{1+2^{k\left( 1 - \alpha\right)}}
			\right\}_{k=0}^\infty\right\|_{\ell^q}
			+
			\left\|\left\{
			2^{k\left(\frac 12-\alpha\right)}
			\mathbb{1}_{\left\{ k \leq \frac{\log\left(8\lambda\right)}{\log 2} \right\}}
			\right\}_{k=0}^\infty\right\|_{\ell^q}
			\right)
			\\
			&\hspace{14mm} \leq
			\frac C\lambda
			\left(
			1+\left(\frac{\left(16\lambda\right)^{\left(\frac 12 -\alpha\right)q}-1}{2^{\left(\frac 12 - \alpha\right)q}-1}\right)^\frac{1}{q}
			\right)
			\leq
			\frac{C}{\lambda^{\gamma+\frac 12}},
		\end{aligned}
	\end{equation}
	where $\gamma=\min\left\{\alpha, \frac 12\right\}$, and, similarly, that
	\begin{equation}\label{lambda strip 3}
		\left\|\left\{
		2^{-k\alpha}
		\left\| \Delta_{2^k}^v h_2 \right\|_{L^2_v}
		\right\}_{k=0}^\infty\right\|_{\ell^q}
		\leq
		\frac{C}{\left(tR\right)^{\gamma+\frac 12}}.
	\end{equation}
	Furthermore, both cases $q=\infty$ and $\alpha=\frac 12$ follow from obvious modifications of the preceding estimates. Thus, combining \eqref{lambda strip 1}, \eqref{lambda strip 2} and \eqref{lambda strip 3} readily shows that the estimates \eqref{lambda strip} and \eqref{lambda strip log} hold.
	
	It only remains to handle the case $\alpha\geq \frac 12$ and $0\notin \mathrm{supp}\hat\rho$. Without any loss of generality, we may assume that $\hat\rho(r)$ is supported inside $\left\{1\leq |r|\leq 2\right\}$. In this setting, it is possible to refine estimates \eqref{support 1} and \eqref{support 2} as to obtain
	\begin{equation}
		\begin{aligned}
			\left\| \Delta_{2^k}^v h_1 \right\|_{L^2_v}
			& \leq
			\frac{2^{\frac{k}{2}}}{\lambda}
			\left\|
			\mathbb{1}_{\left\{ 2^{k-2} \leq \left|\xi'\right| \leq 2^{k+1}\right\}}
			\pi\left(\xi'\right)
			\right\|_{L^2_{\xi'}}
			\left\|
			\mathbb{1}_{\left\{\left|\xi_D\right| \leq 2\right\}}
			\hat\rho\left(\frac{2^k}{\lambda} \xi_D\right)
			\right\|_{L^2_{\xi_D}}\\
			&+
			C \frac{2^\frac{k}{2}}{\lambda}
			\left\|
			\mathbb{1}_{\left\{ \frac{1}{4} \leq \left|\xi_D\right| \leq 2 \right\}}
			\hat\rho\left(\frac{2^k}{\lambda} \xi_D\right)
			\right\|_{L^2_{\xi_D}} \\
			&\leq
			C\left(
			\frac{2^{\frac k2}}{\lambda\left(1+2^{k}\right)}
			\mathbb{1}_{\left\{ \frac 12\leq \frac{2^{k}}{\lambda} \right\}}
			+
			\frac{2^\frac{k}{2}}{\lambda}
			\mathbb{1}_{\left\{\frac 12 \leq \frac{2^{k}}{\lambda} \leq 8 \right\}}
			\right) \\
			&\leq
			C\left(
			\frac{2^{k\alpha}}{\lambda^{\alpha+\frac 12}\left(1+2^{k}\right)}
			\mathbb{1}_{\left\{ \frac 12\leq \frac{2^{k}}{\lambda} \right\}}
			+
			\frac{2^{k\alpha}}{\lambda^{\alpha+\frac 12}}
			\mathbb{1}_{\left\{\frac 12 \leq \frac{2^{k}}{\lambda} \leq 8 \right\}}
			\right)
		\end{aligned}
	\end{equation}
	and, similarly, that
	\begin{equation}
		\begin{aligned}
			\left\| \Delta_{2^k}^v h_2 \right\|_{L^2_v}
			&\leq C \mathbb{1}_{\left\{\frac R2 \leq |\eta| \leq 2R\right\}}
			\left(
			\frac{2^{\frac k2}}{t|\eta|\left(1+2^{k\left( 1 + \alpha\right)}\right)}
			\mathbb{1}_{\left\{ \frac 12\leq \frac{2^{k}}{t|\eta|} \right\}}
			+
			\frac{2^\frac{k}{2}}{t|\eta|}
			\mathbb{1}_{\left\{ \frac 12 \leq \frac{2^{k}}{t|\eta|} \leq 8 \right\}}
			\right)\\
			&\leq C
			\left(
			\frac{2^{\frac k2}}{tR\left(1+2^{k}\right)}
			\mathbb{1}_{\left\{ \frac 14\leq \frac{2^{k}}{tR} \right\}}
			+
			\frac{2^\frac{k}{2}}{tR}
			\mathbb{1}_{\left\{ \frac 14 \leq \frac{2^{k}}{tR} \leq 16 \right\}}
			\right)\\
			&\leq
			C\left(
			\frac{2^{k\alpha}}{(tR)^{\alpha+\frac 12}\left(1+2^{k}\right)}
			\mathbb{1}_{\left\{ \frac 14\leq \frac{2^{k}}{tR} \right\}}
			+
			\frac{2^{k\alpha}}{(tR)^{\alpha+\frac 12}}
			\mathbb{1}_{\left\{\frac 14 \leq \frac{2^{k}}{tR} \leq 16 \right\}}
			\right).
		\end{aligned}
	\end{equation}
	
	It is then readily seen that, for any $1\leq q\leq\infty$,
	\begin{equation}
		\begin{aligned}
			& \left\|\left\{
			2^{-k\alpha}
			\left\| \Delta_{2^k}^v h_1 \right\|_{L^2_v}
			\right\}_{k=0}^\infty\right\|_{\ell^q}\\
			&\hspace{14mm} \leq
			\frac C{\lambda^{\alpha+\frac 12}} \left(
			\left\|\left\{
			\frac{1}{1+2^k}
			\right\}_{k=0}^\infty\right\|_{\ell^q}
			+
			\left\|\left\{
			\mathbb{1}_{\left\{ \frac{\log \lambda}{\log 2} - 1 \leq k \leq \frac{\log\lambda}{\log 2} + 3 \right\}}
			\right\}_{k=0}^\infty\right\|_{\ell^q}
			\right)
			\\
			&\hspace{14mm}
			\leq
			\frac{C}{\lambda^{\alpha+\frac 12}},
		\end{aligned}
	\end{equation}
	and, similarly, that
	\begin{equation}
		\left\|\left\{
		2^{-k\alpha}
		\left\| \Delta_{2^k}^v h_2 \right\|_{L^2_v}
		\right\}_{k=0}^\infty\right\|_{\ell^q}
		\leq
		\frac{C}{\left(tR\right)^{\alpha+\frac 12}}.
	\end{equation}
	Finally, notice that, since $\mathrm{supp}\psi(\xi)\subset\left\{|\xi_D|\leq 1\right\}$ and $\lambda\geq 1$, $tR\geq 1$,
	\begin{equation}
		\begin{aligned}
			\left\| \Delta_{0}^v h_1 \right\|_{L^2_v}&=
			\frac{1}{\lambda}
			\left\|
			\psi\left(\xi\right)
			\pi\left(\xi_1,\ldots,\xi_{D-1}\right)
			\hat\rho\left(\frac{1}{\lambda} \xi_D \right)
			\right\|_{L^2_\xi}\mathbb{1}_{\left\{\lambda<1\right\}}=0,\\
			\left\| \Delta_{0}^v h_2 \right\|_{L^2_v}&=
			\mathbb{1}_{\left\{\frac R2 \leq |\eta| \leq 2R\right\}} \frac{1}{t|\eta|}
			\left\|
			\psi\left(\xi\right)
			\pi\left(\xi_1,\ldots,\xi_{D-1}\right)
			\hat\rho\left(\frac{1}{t|\eta|} \xi_D \right)
			\right\|_{L^2_\xi}\mathbb{1}_{\left\{tR<1\right\}}=0,
		\end{aligned}
	\end{equation}
	which readily establishes \eqref{lambda strip support} and thus, concludes the proof of the lemma.
\end{proof}

\begin{proof}[Proof of Lemma \ref{lem lambda strip low}]
	For each $\eta\in\mathbb{R}^D$, considering an orthogonal transformation $R_\eta:\mathbb{R}^D\to\mathbb{R}^D$ such that $\frac{\eta}{|\eta|}$ is mapped onto $\left(0,\ldots,0,1\right)$ and writing $v'=\left(v_1,\ldots,v_{D-1}\right)$ so that $v=\left(v',v_D\right)$, it holds that
	\begin{equation}
		\begin{aligned}
			\mathcal{F}_v h(\eta,\xi)
			&= \mathbb{1}_{\left\{|\eta| \leq 1\right\}}
			\int_{\mathbb{R}^D} e^{- i v\cdot\left(R_\eta \xi\right)}
			\chi\left( \left|v\right| \right) \rho\left(|\eta| v_D\right)
			dv\\
			&= \mathbb{1}_{\left\{|\eta| \leq 1\right\}}
			\int_{\mathbb{R}}
			\pi\left(\left(R_\eta\xi\right)_1,\ldots,\left(R_\eta\xi\right)_{D-1},\left(R_\eta\xi\right)_D-r\right)
			\frac{1}{|\eta|}\hat\rho\left(\frac{1}{|\eta|}r\right)dr,
		\end{aligned}
	\end{equation}
	where $\pi\left(\xi\right)=\int_{\mathbb{R}^{D-1}}e^{-i \xi\cdot v}\chi\left(\left|v\right|\right)dv\in\mathcal{S}\left(\mathbb{R}^{D}\right)$. Consequently,
	\begin{equation}
		\begin{aligned}
			\left\| \Delta_{0}^v h \right\|_{L^2_v}&=
			\mathbb{1}_{\left\{|\eta| \leq 1\right\}}\frac{1}{\left(2\pi\right)^\frac{D}{2}}
			\left\|
			\psi\left(\xi\right)
			\int_{\mathbb{R}}
			\pi\left(\xi',\xi_D-r\right)
			\frac{1}{|\eta|}\hat\rho\left(\frac{1}{|\eta|}r\right)dr
			\right\|_{L^2_\xi}\\
			&\leq\left\|\pi(\xi)\right\|_{L^2_\xi}\left\|\hat\rho(\xi_D)\right\|_{L^1_{\xi_D}}
			,\\
			\left\| \Delta_{2^k}^v h \right\|_{L^2_v}&=
			\mathbb{1}_{\left\{ |\eta| \leq 1\right\}}\frac{1}{\left(2\pi\right)^\frac{D}{2}}
			\left\|
			\varphi\left(\frac{\xi}{2^k}\right)
			\int_{\mathbb{R}}
			\pi\left(\xi',\xi_D-r\right)
			\frac{1}{|\eta|}\hat\rho\left(\frac{1}{|\eta|}r\right)dr
			\right\|_{L^2_\xi}.
		\end{aligned}
	\end{equation}
	
	Recall now that $\mathrm{supp}\varphi\left(\frac{\xi}{2^k}\right) \subset \left\{2^{k-1}\leq\left|\xi\right|\leq 2^{k+1} \right\}$. In particular, for every $k\in\mathbb{N}$, we have that
	\begin{equation}
		\begin{aligned}
			\varphi\left(\frac{\xi}{2^k}\right) & = \varphi\left(\frac{\xi}{2^k}\right) \left( \mathbb{1}_{\left\{\left|\xi'\right|\geq 2^{k-2}\right\}} + \mathbb{1}_{\left\{\left|\xi'\right|<2^{k-2}\right\}} \right)\\
			& \leq
			\mathbb{1}_{\left\{ 2^{k-2} \leq \left|\xi'\right| \leq 2^{k+1}\right\}}
			\mathbb{1}_{\left\{\left|\xi_D\right| \leq 2^{k+1}\right\}}
			+
			\mathbb{1}_{\left\{ \left|\xi'\right| < 2^{k-2}\right\}}
			\mathbb{1}_{\left\{ 2^{k-2} \leq \left|\xi_D\right| \leq 2^{k+1}\right\}}.
		\end{aligned}
	\end{equation}
	Hence, we deduce, utilizing that $\pi\left(\xi\right)$ decays faster than any polynomial and assuming without any loss of generality that $\hat\rho(r)$ is supported inside $\left\{|r|\leq 2\right\}$, that
	\begin{equation}
		\begin{aligned}
			\left\| \Delta_{2^k}^v h \right\|_{L^2_v}
			&\leq
			\left\|
			\mathbb{1}_{\left\{ 2^{k-2} \leq \left|\xi'\right| \leq 2^{k+1}\right\}}
			\pi\left(\xi\right)
			\right\|_{L^2_{\xi}}
			\left\|
			\hat\rho\left( \xi_D\right)
			\right\|_{L^1_{\xi_D}}\\
			&+
			\left\|
			\mathbb{1}_{\left\{ 2^{k-2}-2 \leq \left|\xi_D\right| \leq 2^{k+1}+2\right\}}
			\pi\left(\xi\right)
			\right\|_{L^2_{\xi}}
			\left\|
			\hat\rho\left(\xi_D\right)
			\right\|_{L^1_{\xi_D}}\\
			&\leq
			\frac{C}{2^{k\left( 1 - \alpha\right)}}.
		\end{aligned}
	\end{equation}
	It clearly follows that
	\begin{equation}
		\left\|\left\{
		2^{-k\alpha}
		\left\| \Delta_{2^k}^v h \right\|_{L^\infty_\eta L^2_v}
		\right\}_{k=0}^\infty\right\|_{\ell^q}
		\leq
		2C,
	\end{equation}
	which concludes the proof of the lemma.
\end{proof}

Now that the technical Lemmas \ref{operators norms}, \ref{lem lambda strip} and \ref{lem lambda strip low} are established, we may proceed to the proof of the main result Theorem \ref{classical}.

\begin{proof}[Proof of Theorem \ref{classical}]
	
	First of all, we easily obtain that
	\begin{equation}
		\begin{aligned}
			\left\|\Delta_{0}^x\int_{\mathbb{R}^D} f(x,v)\phi(v)dv\right\|_{L^2_x}
			&\leq C
			\left\|\Delta_{0}^xf(x,v)\right\|_{\widetilde L^2 \left( \mathbb{R}^D_x ; B^{\alpha}_{2,q} \left(\mathbb{R}^D_v\right) \right)}
			\left\|\phi(v)\right\|_{B^{-\alpha}_{2,q'} \left(\mathbb{R}^D_v\right)}\\
			&\leq C
			\left\|f(x,v)\right\|_{B^{a,\alpha}_{2,2,q} \left(\mathbb{R}^D_x\times\mathbb{R}^D_v\right)}
			\left\|\phi(v)\right\|_{B^{-\alpha}_{2,q'} \left(\mathbb{R}^D_v\right)},
		\end{aligned}
	\end{equation}
	which concludes the control of the low frequencies. Notice, however, that it will be convenient, in order to carry out interpolation arguments later on, to also estimate the low frequencies with the same decomposition that was used in the proof of Theorem \ref{pre averaging lemma p 2} and provided by Proposition \ref{crucial}, i.e.
	\begin{equation}
		\Delta_{0}^x\int_{\mathbb{R}^D} f(x,v)\phi(v) dv
		= A_{0}^1 (f\phi) (x) + B_{0}^1 (g\phi) (x).
	\end{equation}
	This is easily done with an application of Lemma \ref{operators norms}, thus yielding
	\begin{equation}
		\begin{aligned}
			\left\|A_0^1(f\phi)\right\|_{L^2_x}
			&\leq
			C
			\left\| \Delta_0^x f \right\|_{\widetilde L^2 \left( \mathbb{R}^D_x ; B^{\alpha}_{2,q} \left(\mathbb{R}^D_v\right) \right)}
			\leq
			C
			\left\| f \right\|_{B^{a,\alpha}_{2,2,q} \left(\mathbb{R}^D_x\times\mathbb{R}^D_v\right)},
		\end{aligned}
	\end{equation}
	and
	\begin{equation}
		\begin{aligned}
			\left\|B_0^1(g\phi)\right\|_{L^2_x}
			&\leq
			C
			\left\| \Delta_0^x g \right\|_{\widetilde L^2 \left( \mathbb{R}^D_x ; B^{\beta}_{2,q} \left(\mathbb{R}^D_v\right) \right)}
			\leq
			C
			\left\| g \right\|_{B^{b,\beta}_{2,2,q} \left(\mathbb{R}^D_x\times\mathbb{R}^D_v\right)}.
		\end{aligned}
	\end{equation}
	
	In order to estimate the high frequencies, according to proposition \ref{crucial}, we consider $\rho\in \mathcal{S}\left(\mathbb{R}\right)$ a cutoff function such that $\tilde\rho(r)=\frac{1}{2\pi}\hat\rho(-r)$ is compactly supported $\mathrm{supp}\tilde\rho(r)\subset\left\{1\leq |r|\leq 2\right\}$ and $\rho(0)=\frac{1}{2\pi}\int_{\mathbb{R}}\hat\rho(r)dr=1$, so that, for any $t>0$, we have the dyadic frequency decompositions, for each $k\in\mathbb{N}$,
	\begin{equation}	
			\Delta_{2^k}^x\int_{\mathbb{R}^D} f(x,v) \phi(v) dv
			= A_{2^k}^t (f\phi) (x) + t B_{2^k}^t (g\phi) (x).
	\end{equation}
	
	Consequently, in virtue of lemma \ref{operators norms}, we infer that, if $\beta\neq \frac 12$ or $q=1$, for any $t\geq 2^{-k}$,
	\begin{equation}\label{classical 2}
		\begin{aligned}
			\left\|\Delta_{2^k}^x\int_{\mathbb{R}^D} f(x,v)\phi(v) dv\right\|_{L^2_x}
			&\leq
			\left\| A_{2^k}^t (f\phi) \right\|_{L^2_x} + t \left\| B_{2^k}^t (g\phi) \right\|_{L^2_x} \\
			& \leq
			\frac{C}{\left(t2^k\right)^{\alpha+\frac{1}{2}}} \left\|\Delta_{2^k}^x f \right\|_{\widetilde L^2 \left( \mathbb{R}^D_x ; B^{\alpha}_{2,q} \left(\mathbb{R}^D_v\right) \right)} \\
			& +
			2^{-k}\frac{C}{\left(t2^k\right)^{\gamma-\frac{1}{2}}} \left\| \Delta_{2^k}^x g \right\|_{\widetilde L^2 \left( \mathbb{R}^D_x ; B^{\beta}_{2,q} \left(\mathbb{R}^D_v\right) \right)},
		\end{aligned}
	\end{equation}
	where $\gamma=\min\left\{\beta,\frac 12\right\}$, and similarly, if $\beta=\frac 12$ and $q\neq 1$,
	\begin{equation}
		\begin{aligned}
			\left\|\Delta_{2^k}^x\int_{\mathbb{R}^D} f(x,v)\phi(v) dv\right\|_{L^2_x}
			&\leq
			\left\| A_{2^k}^t (f\phi) \right\|_{L^2_x} + t \left\| B_{2^k}^t (g\phi) \right\|_{L^2_x} \\
			& \leq
			\frac{C}{\left(t2^k\right)^{\alpha+\frac{1}{2}}} \left\|\Delta_{2^k}^x f \right\|_{\widetilde L^2 \left( \mathbb{R}^D_x ; B^{\alpha}_{2,q} \left(\mathbb{R}^D_v\right) \right)} \\
			& +
			\frac{C}{2^k}\log\left(1+t2^k\right)^\frac{1}{q'} \left\| \Delta_{2^k}^x g \right\|_{\widetilde L^2 \left( \mathbb{R}^D_x ; B^{\beta}_{2,q} \left(\mathbb{R}^D_v\right) \right)},
		\end{aligned}
	\end{equation}
	where $C>0$ is independent of $t$ and $2^k$.

	Then, optimizing the interpolation parameter $t$, we choose $t2^k=2^{k\frac{1+b-a}{1+\alpha-\gamma}}$, which is admissible since $\frac{1+b-a}{1+\alpha-\gamma}\geq 0$ and thus $t\geq 2^{-k}$.
	
	Furthermore, in the cases $\beta>\frac 12$ or $\beta=\frac 12$ and $q=1$, we can choose $t=\infty$, which is, in fact, more optimal than $t2^k=2^{k\frac{1+b-a}{1+\alpha-\gamma}}$, since it eliminates the first term in the right-hand side of the above estimates. This case is discussed later on.
	

	It follows that, recalling $ s = (1+b-a) \frac{\alpha+\frac 12}{1+\alpha-\gamma} +a = (1+b-a) \frac{\gamma-\frac 12}{1+\alpha-\gamma}+1+b$, in the case $\beta\neq \frac 12$ or $q=1$,
	\begin{equation}
		\begin{aligned}
			& 2^{ks} \left\|\Delta_{2^k}^x\int_{\mathbb{R}^D} f(x,v)\phi(v) dv\right\|_{L^2_x} \\
			&\hspace{20mm} \leq 
			C 2^{ka} \left\|\Delta_{2^k}^x f \right\|_{\widetilde L^2 \left( \mathbb{R}^D_x ; B^{\alpha}_{2,q} \left(\mathbb{R}^D_v\right) \right)}+
			C 2^{kb} \left\| \Delta_{2^k}^x g \right\|_{\widetilde L^2 \left( \mathbb{R}^D_x ; B^{\beta}_{2,q} \left(\mathbb{R}^D_v\right) \right)},
		\end{aligned}
	\end{equation}
	and similarly, if $\beta=\frac 12$ and $q\neq 1$, for every $\epsilon>0$,
	\begin{equation}
		\begin{aligned}
			& 2^{k(s-\epsilon)} \left\|\Delta_{2^k}^x\int_{\mathbb{R}^D} f(x,v)\phi(v) dv\right\|_{L^2_x} \\
			&\hspace{20mm} \leq 
			C 2^{ka} \left\|\Delta_{2^k}^x f \right\|_{\widetilde L^2 \left( \mathbb{R}^D_x ; B^{\alpha}_{2,q} \left(\mathbb{R}^D_v\right) \right)}+
			C 2^{kb} \left\| \Delta_{2^k}^x g \right\|_{\widetilde L^2 \left( \mathbb{R}^D_x ; B^{\beta}_{2,q} \left(\mathbb{R}^D_v\right) \right)}.
		\end{aligned}
	\end{equation}
	Finally, summing over $k\in\mathbb{N}$ yields that, in the case $\beta\neq \frac 12$ or $q=1$,
	\begin{equation}
		\begin{aligned}
			&\left\|\left\{
			2^{ks}\left\|\Delta_{2^k}^x\int_{\mathbb{R}^D} f(x,v) \phi(v) dv\right\|_{L^2_x}
			\right\}_{k=0}^\infty\right\|_{\ell^q} \\
			&\hspace{40mm} \leq
			C
			\left\|\left\{ 2^{ka}
			\left\|\Delta_{2^k}^x f \right\|_{\widetilde L^2 \left( \mathbb{R}^D_x ; B^{\alpha}_{2,q} \left(\mathbb{R}^D_v\right) \right)}
			\right\}_{k=0}^\infty \right\|_{\ell^q} \\
			&\hspace{40mm} +
			C
			\left\|\left\{ 2^{kb}
			\left\| \Delta_{2^k}^x g \right\|_{\widetilde L^2 \left( \mathbb{R}^D_x ; B^{\beta}_{2,q} \left(\mathbb{R}^D_v\right) \right)}
			\right\}_{k=0}^\infty \right\|_{\ell^q},
		\end{aligned}
	\end{equation}
	and similarly, when $\beta=\frac 12$ and $q\neq 1$,
	\begin{equation}
		\begin{aligned}
			&\left\|\left\{
			2^{k(s-\epsilon)}\left\|\Delta_{2^k}^x\int_{\mathbb{R}^D} f(x,v) \phi(v) dv\right\|_{L^2_x}
			\right\}_{k=0}^\infty\right\|_{\ell^q} \\
			&\hspace{40mm} \leq
			C
			\left\|\left\{ 2^{ka}
			\left\|\Delta_{2^k}^x f \right\|_{\widetilde L^2 \left( \mathbb{R}^D_x ; B^{\alpha}_{2,q} \left(\mathbb{R}^D_v\right) \right)}
			\right\}_{k=0}^\infty \right\|_{\ell^q} \\
			&\hspace{40mm} +
			C
			\left\|\left\{ 2^{kb}
			\left\| \Delta_{2^k}^x g \right\|_{\widetilde L^2 \left( \mathbb{R}^D_x ; B^{\beta}_{2,q} \left(\mathbb{R}^D_v\right) \right)}
			\right\}_{k=0}^\infty \right\|_{\ell^q}.
		\end{aligned}
	\end{equation}

	We handle now the cases $\beta>\frac 12$ or $\beta=\frac 12$ and $q=1$, by letting $t$ tend to infinity in \eqref{classical 2}, as mentioned previously. This leads to
	\begin{equation}
		\left\|\Delta_{2^k}^x\int_{\mathbb{R}^D} f(x,v)\phi(v) dv\right\|_{L^2_x}
		\leq
		\frac{C}{2^k} \left\| \Delta_{2^k}^x g \right\|_{\widetilde L^2 \left( \mathbb{R}^D_x ; B^{\beta}_{2,q} \left(\mathbb{R}^D_v\right) \right)}.
	\end{equation}
	Hence, recalling $s=1+b$ and summing over $k$ yields
	\begin{equation}
		\begin{aligned}
			&\left\|\left\{
			2^{ks}\left\|\Delta_{2^k}^x\int_{\mathbb{R}^D} f(x,v) \phi(v) dv\right\|_{L^2_x}
			\right\}_{k=0}^\infty\right\|_{\ell^q} \\
			&\hspace{40mm} \leq
			C
			\left\|\left\{ 2^{kb}
			\left\| \Delta_{2^k}^x g \right\|_{\widetilde L^2 \left( \mathbb{R}^D_x ; B^{\beta}_{2,q} \left(\mathbb{R}^D_v\right) \right)}
			\right\}_{k=0}^\infty \right\|_{\ell^q},
		\end{aligned}
	\end{equation}
	which concludes the proof of the theorem.
\end{proof}

\subsection{The $L_x^1L_v^p$ and $L_x^2L_v^2$ cases reconciled}

We give now the proof of the general Theorem \ref{main averaging lemma}, which will follow from a simple interpolation procedure. In particular, the following lemma results from the interpolation between Lemma \ref{operators norms dispersive} and Lemma \ref{operators norms}.

	\begin{lem}\label{operators norms interpolation}
		Let $\phi(v)\in C_0^\infty\left(\mathbb{R}^D\right)$. For every $1\leq p, r\leq \infty$, $1\leq q < \infty$, $\alpha\in\mathbb{R}$, $k\in\mathbb{N}$ and $t\geq 2^{-k}$ such that
		\begin{equation}
			 r\leq p\leq r',
		\end{equation}
		 it holds that
		\begin{equation}
			\left\|A_{0}^1\left(f\phi\right)\right\|_{L^p_x}
			\leq C \left\| f \right\|_{\widetilde L^r\left( \mathbb{R}^D_x; B^\alpha_{p,q}\left(\mathbb{R}^D_v\right)\right)}
		\end{equation}
		and
		\begin{equation}
			\left\|A_{2^k}^t\left(f\phi\right)\right\|_{L^p_x}
			\leq {C\over t^{\alpha+1-\frac 1r + D\left(\frac 1r - \frac 1p\right)} 2^{k\left(\alpha+1-\frac 1r\right)}} \left\| f \right\|_{\widetilde L^r\left( \mathbb{R}^D_x; B^\alpha_{p,q}\left(\mathbb{R}^D_v\right)\right)},
		\end{equation}
		where $C>0$ is independent of $t$ and $2^k$.
		
		Furthermore, if
		\begin{equation}
			D\left(\frac 1r-\frac 1p\right)<\frac 2r -1\qquad \text{or}\qquad p=r=2,
		\end{equation}
		then, for any $\beta\in\mathbb{R}$, it holds that
		\begin{equation}
			\left\|B_{0}^1\left(g\phi\right)\right\|_{L^p_x}
			\leq C \left\| g \right\|_{\widetilde L^r\left( \mathbb{R}^D_x; B^\beta_{p,q}\left(\mathbb{R}^D_v\right)\right)},
		\end{equation}
		and, if further $\beta\neq\frac 1r-D\left(\frac 1r-\frac{1}{p}\right)$,
		\begin{equation}
			\left\|B_{2^k}^t\left(g\phi\right)\right\|_{L^p_x}
			\leq {C\over t^{\gamma+1-\frac 1r + D\left(\frac 1r - \frac 1p\right)} 2^{k\left(\gamma+1-\frac 1r\right)}} \left\| g \right\|_{\widetilde L^r\left( \mathbb{R}^D_x; B^\beta_{p,q}\left(\mathbb{R}^D_v\right)\right)},
		\end{equation}
		where $\gamma=\min\left\{\beta,\frac 1r-D\left(\frac 1r-\frac{1}{p}\right)\right\}$ and $C>0$ is independent of $t$ and $2^k$.
		
		Finally, if $\beta=\frac 1r-D\left(\frac 1r-\frac{1}{p}\right)$, then
		\begin{equation}
			\begin{aligned}
				\left\|B_{2^k}^t\left(g\phi\right)\right\|_{L^p_x}
				& \leq {C\over t 2^{k\left(\beta+1-\frac 1r\right)}}
				\log\left(1+t2^k\right)^\frac{1}{q'}
				\left\| g \right\|_{\widetilde L^r\left( \mathbb{R}^D_x; B^\beta_{p,q}\left(\mathbb{R}^D_v\right)\right)},
			\end{aligned}
		\end{equation}
		where $C>0$ is independent of $t$ and $2^k$.
	\end{lem}
	
	\begin{proof}
		First, if $r=1$, then the statement of the present lemma is a straightforward consequence of Lemma \ref{operators norms dispersive} and Lemma \ref{paradifferential}, while the case $r=2$ is contained in Lemma \ref{operators norms}. If $1<r<2$, then the result will follow from the interpolation of Lemma \ref{operators norms dispersive} with Lemma \ref{operators norms}. To this end, we define the interpolation parameter $0<\lambda<1$ by $\lambda=2\left(1-\frac 1r\right)$, so that
		\begin{equation}
			\frac 1r=\frac{1-\lambda}{1}+\frac{\lambda}{2},
		\end{equation}
		and the parameter $1\leq p_0<\infty$ by
		\begin{equation}
			\frac 1p = \frac{1-\lambda}{p_0} + \frac{\lambda}{2}.
		\end{equation}
		It is then easy to check that if $D\left(\frac 1r-\frac 1p\right)<\frac 2r -1$, then
		\begin{equation}
			D\left(1-\frac 1{p_0}\right) <1.
		\end{equation}
		
		Further notice that
		\begin{equation}
			\frac 1r - D\left( \frac 1r - \frac 1p \right) =
			(1-\lambda)\left[ 1 - D\left( 1 - \frac 1{p_0} \right) \right] + {\lambda \over 2}.
		\end{equation}
		Therefore, if $\beta\left\{\substack{<\\=\\>}\right\}\frac 1r - D\left( \frac 1r - \frac 1p \right)$, it is always possible to respectively find $\beta_0\left\{\substack{<\\=\\>}\right\} 1 - D\left( 1 - \frac 1{p_0} \right)$ and $\beta_1\left\{\substack{<\\=\\>}\right\}\frac 12$ such that
		\begin{equation}
			\beta=(1-\lambda)\beta_0 + \lambda \beta_1.
		\end{equation}
		Consequently, defining $\gamma_0=\min\left\{\beta_0,1-D\left(1-\frac{1}{p_0}\right)\right\}$ and $\gamma_1=\min\left\{\beta_1,\frac 12\right\}$, it also holds that
		\begin{equation}
			\gamma=(1-\lambda)\gamma_0 + \lambda \gamma_1.
		\end{equation}

		Then, in the case $\beta\neq\frac 1r-D\left(\frac 1r-\frac{1}{p}\right)$, we deduce according to Lemma \ref{operators norms dispersive} that
		\begin{equation}\label{lambda 0}
			\begin{aligned}
				\left\|A_{0}^1\left(f\phi\right)\right\|_{L^{p_0}_x}
				& \leq C \left\| f \right\|_{\widetilde L^1\left( \mathbb{R}^D_x; B^{\alpha}_{p_0,q}\left(\mathbb{R}^D_v\right)\right)}, \\
				\left\|A_{2^k}^t\left(f\phi\right)\right\|_{L^{p_0}_x}
				& \leq {C\over t^{\alpha+D\left(1-\frac 1{p_0}\right)} 2^{k\alpha}} \left\| f \right\|_{\widetilde L^1\left( \mathbb{R}^D_x; B^{\alpha}_{p_0,q}\left(\mathbb{R}^D_v\right)\right)}, \\
				\left\|B_{0}^1\left(g\phi\right)\right\|_{L^{p_0}_x}
				& \leq C \left\| g \right\|_{\widetilde L^1\left( \mathbb{R}^D_x; B^{\beta_0}_{p_0,q}\left(\mathbb{R}^D_v\right)\right)},\\
				\left\|B_{2^k}^t\left(g\phi\right)\right\|_{L^{p_0}_x}
				& \leq {C\over t^{\gamma_0+D\left(1-\frac 1{p_0}\right)} 2^{k\gamma_0}} \left\| g \right\|_{\widetilde L^1\left( \mathbb{R}^D_x; B^{\beta_0}_{p_0,q}\left(\mathbb{R}^D_v\right)\right)},
			\end{aligned}
		\end{equation}
		where $C>0$ is independent of $t$ and $2^k$, and employing Lemma \ref{operators norms}, we infer that
		\begin{equation}\label{lambda 1}
			\begin{aligned}
				\left\|A_{0}^1\left(f\phi\right)\right\|_{L^2_x}
				&\leq
				C
				\left\| f \right\|_{\widetilde L^2 \left( \mathbb{R}^D_x ; B^{\alpha}_{2,q} \left(\mathbb{R}^D_v\right) \right)},\\
				\left\|A_{2^k}^t\left(f\phi\right)\right\|_{L^2_x}
				&\leq
				\frac{C}{\left(t2^k\right)^{\alpha+\frac 12}}
				\left\| f \right\|_{\widetilde L^2 \left( \mathbb{R}^D_x ; B^{\alpha}_{2,q} \left(\mathbb{R}^D_v\right) \right)},\\
				\left\|B_{0}^1\left(g\phi\right)\right\|_{L^2_x}
				&\leq
				C
				\left\| g \right\|_{\widetilde L^2 \left( \mathbb{R}^D_x ; B^{\beta_1}_{2,q} \left(\mathbb{R}^D_v\right) \right)},\\
				\left\|B_{2^k}^t\left(g\phi\right)\right\|_{L^2_x}
				&\leq
				\frac{C}{\left(t2^k\right)^{\gamma_1+\frac 12}}
				\left\| g \right\|_{\widetilde L^2 \left( \mathbb{R}^D_x ; B^{\beta_1}_{2,q} \left(\mathbb{R}^D_v\right) \right)},
			\end{aligned}
		\end{equation}
		where $C>0$ is independent of $t$ and $2^k$.
		
		In the case $\beta=\frac 1r-D\left(\frac 1r-\frac{1}{p}\right)$, according to the same Lemmas \ref{operators norms dispersive} and \ref{operators norms}, solely the estimates on $B_{2^k}^t$ should be replaced by
		\begin{equation}
			\begin{aligned}
				\left\|B_{2^k}^t\left(g\phi\right)\right\|_{L^{p_0}_x}
				& \leq {C\over t 2^{k\beta_0}}
				\log\left(1+t2^k\right)^\frac{1}{q'}
				\left\| g \right\|_{\widetilde L^1\left( \mathbb{R}^D_x; B^{\beta_0}_{p_0,q}\left(\mathbb{R}^D_v\right)\right)},\\
				\left\|B_{2^k}^t\left(g\phi\right)\right\|_{L^2_x}
				&\leq
				\frac{C}{t2^k}
				\log\left(1+t2^k\right)^\frac{1}{q'}
				\left\| g \right\|_{\widetilde L^2 \left( \mathbb{R}^D_x ; B^{\beta_1}_{2,q} \left(\mathbb{R}^D_v\right) \right)}.
			\end{aligned}
		\end{equation}
		
		We are now going to utilize standard results from complex interpolation theory (cf. \cite{bergh}) in order to obtain new estimates from the interpolation of estimates \eqref{lambda 0} and \eqref{lambda 1}.
		
		To this end, first recall that the real interpolation of Lebesgue spaces (cf. \cite{bergh}) yields that $\left(L^{p_0},L^{p_1}\right)_{[\theta]}=L^{p}$, for any $1\leq p,p_0,p_1< \infty$ and $0<\theta<1$ such that $\frac 1p=\frac {1-\theta}{p_0}+\frac \theta{p_1}$.

		Furthermore, the complex interpolation of standard Besov spaces (cf. \cite{bergh}) yields, in particular, that $\left(B^{\alpha_0}_{p_0,q},B^{\alpha_1}_{p_1,q}\right)_{[\theta]}=B^{\alpha}_{p,q}$, for any $\alpha,\alpha_0,\alpha_1\in\mathbb{R}$ and $1\leq p,p_0,p_1,q<\infty$ such that $\alpha=(1-\theta)\alpha_0+\theta\alpha_1$ and $\frac 1p=\frac{1-\theta}{p_0}+\frac{\theta}{p_1}$. It is then possible to smoothly adapt the proof of this standard property to obtain a corresponding result for the Besov spaces introduced in Section \ref{LP decomposition}. Namely, one can show that $\left(\widetilde L^{r_0} B^{\alpha_0}_{p_0,q},\widetilde L^{r_1} B^{\alpha_1}_{p_1,q}\right)_{[\theta]}=\widetilde L^r B^{\alpha}_{p,q}$, for any $\alpha,\alpha_0,\alpha_1\in\mathbb{R}$ and $1\leq p,p_0,p_1,q,r,r_0,r_1<\infty$ such that $\alpha=(1-\theta)\alpha_0+\theta\alpha_1$, $\frac 1p=\frac{1-\theta}{p_0}+\frac{\theta}{p_1}$ and $\frac 1r=\frac{1-\theta}{r_0}+\frac{\theta}{r_1}$.

		Therefore, we deduce from the complex interpolation of \eqref{lambda 0} and \eqref{lambda 1} that, in the case $\beta\neq\frac 1r-D\left(\frac 1r-\frac{1}{p}\right)$,
		\begin{equation}
			\begin{aligned}
				\left\|A_{0}^1\left(f\phi\right)\right\|_{L^{p}_x}
				& \leq C \left\| f \right\|_{\widetilde L^r\left( \mathbb{R}^D_x; B^{\alpha}_{p,q}\left(\mathbb{R}^D_v\right)\right)},\\
				\left\|A_{2^k}^t\left(f\phi\right)\right\|_{L^{p}_x}
				& \leq {C\over \left(t^{\alpha+D\left(1-\frac 1{p_0}\right)} 2^{k\alpha}\right)^{1-\lambda} \left(\left(t2^k\right)^{\alpha + \frac 1 2}\right)^\lambda} \left\| f \right\|_{\widetilde L^r\left( \mathbb{R}^D_x; B^{\alpha}_{p,q}\left(\mathbb{R}^D_v\right)\right)} \\
				& = {C\over t^{\alpha+1-\frac 1r +D\left(\frac 1r-\frac 1p\right)} 2^{k\left(\alpha+1-\frac 1r\right)}} \left\| f \right\|_{\widetilde L^r\left( \mathbb{R}^D_x; B^{\alpha}_{p,q}\left(\mathbb{R}^D_v\right)\right)},
			\end{aligned}
		\end{equation}
		and
		\begin{equation}
			\begin{aligned}
				\left\|B_{0}^1\left(g\phi\right)\right\|_{L^{p}_x}
				& \leq C \left\| g \right\|_{\widetilde L^r\left( \mathbb{R}^D_x; B^{\beta}_{p,q}\left(\mathbb{R}^D_v\right)\right)},\\
				\left\|B_{2^k}^t\left(g\phi\right)\right\|_{L^{p}_x}
				& \leq {C\over \left(t^{\gamma_0+D\left(1-\frac 1{p_0}\right)} 2^{k\gamma_0}\right)^{1-\lambda} \left(\left(t2^k\right)^{\gamma_1+\frac 1 2}\right)^\lambda} \left\| g \right\|_{\widetilde L^r\left( \mathbb{R}^D_x; B^{\beta}_{p,q}\left(\mathbb{R}^D_v\right)\right)} \\
				& = {C\over t^{\gamma+1-\frac 1r +D\left(\frac 1r-\frac 1p\right)} 2^{k\left(\gamma+1-\frac 1r\right)}} \left\| g \right\|_{\widetilde L^r\left( \mathbb{R}^D_x; B^{\beta}_{p,q}\left(\mathbb{R}^D_v\right)\right)},
			\end{aligned}
		\end{equation}
		where $C>0$ is independent of $t$ and $2^k$.
		
		Accordingly, in the case $\beta=\frac 1r-D\left(\frac 1r-\frac{1}{p}\right)$, only the estimate on $B_{2^k}^t$ should be replaced by
		\begin{equation}
			\begin{aligned}
				\left\|B_{2^k}^t\left(g\phi\right)\right\|_{L^{p}_x}
				& \leq {C\over \left(t2^{k\beta_0}\right)^{1-\lambda} \left(t2^k\right)^\lambda}
				\log\left(1+t2^k\right)^\frac{1}{q'}
				\left\| g \right\|_{\widetilde L^r\left( \mathbb{R}^D_x; B^{\beta}_{p,q}\left(\mathbb{R}^D_v\right)\right)} \\
				& = {C\over t 2^{k\left(\beta+1-\frac 1r\right)}}
				\log\left(1+t2^k\right)^\frac{1}{q'}
				\left\| g \right\|_{\widetilde L^r\left( \mathbb{R}^D_x; B^{\beta}_{p,q}\left(\mathbb{R}^D_v\right)\right)},
			\end{aligned}
		\end{equation}
		where $C>0$ is independent of $t$ and $2^k$, which concludes the proof of the lemma.
	\end{proof}

	\begin{proof}[Proof of Theorem \ref{main averaging lemma}]
		As usual, we begin with the decompositions
		\begin{equation}
			\Delta_{0}^x\int_{\mathbb{R}^D} f(x,v)\phi(v) dv
			= A_{0}^1 (f\phi) (x) + B_{0}^1 (g\phi) (x)
		\end{equation}
		and
		\begin{equation}
				\Delta_{2^k}^x\int_{\mathbb{R}^D} f(x,v) \phi(v) dv
				= A_{2^k}^t (f\phi) (x) + t B_{2^k}^t (g\phi) (x).
		\end{equation}
		provided by Proposition \ref{crucial}.
		
		We show in Lemma \ref{operators norms interpolation} that, in the case $\beta\neq\frac 1r-D\left(\frac 1r-\frac{1}{p}\right)$ or $\beta=\frac 1r-D\left(\frac 1r-\frac{1}{p}\right)$ and $q=1$, the operators above satisfy the bounds
		\begin{equation}
			\begin{aligned}
				\left\|A_{0}^1\left(f\phi\right)\right\|_{L^p_x}
				& \leq C \left\| f \right\|_{\widetilde L^r\left( \mathbb{R}^D_x; B^\alpha_{p,q}\left(\mathbb{R}^D_v\right)\right)}, \\
				\left\|A_{2^k}^t\left(f\phi\right)\right\|_{L^p_x}
				& \leq {C\over t^{\alpha+1-\frac 1r + D\left(\frac 1r - \frac 1p\right)} 2^{k\left(\alpha+1-\frac 1r\right)}} \left\| f \right\|_{\widetilde L^r\left( \mathbb{R}^D_x; B^\alpha_{p,q}\left(\mathbb{R}^D_v\right)\right)}, \\
				\left\|B_{0}^1\left(g\phi\right)\right\|_{L^p_x}
				& \leq C \left\| g \right\|_{\widetilde L^r\left( \mathbb{R}^D_x; B^\beta_{p,q}\left(\mathbb{R}^D_v\right)\right)},\\
				\left\|B_{2^k}^t\left(g\phi\right)\right\|_{L^p_x}
				& \leq {C\over t^{\gamma+1-\frac 1r + D\left(\frac 1r - \frac 1p\right)} 2^{k\left(\gamma+1-\frac 1r\right)}} \left\| g \right\|_{\widetilde L^r\left( \mathbb{R}^D_x; B^\beta_{p,q}\left(\mathbb{R}^D_v\right)\right)},
			\end{aligned}
		\end{equation}
		where $\gamma=\min\left\{\beta,\frac 1r-D\left(\frac 1r-\frac{1}{p}\right)\right\}$ and $C>0$ is independent of $t$ and $2^k$.
		
		It then follows that, in virtue of the identities \eqref{crucial 3} from Proposition \ref{crucial},
		\begin{equation}
			\begin{aligned}
				\left\|
				\Delta_{0}^x\int_{\mathbb{R}^D} f(x,v)\phi(v) dv
				\right\|_{L^p_x} 
				&\leq
				\left\| A_{0}^1\left( S_2^x f\phi \right) \right\|_{L^p_x}
				+
				\left\| B_{0}^1\left(  S_2^x g\phi \right) \right\|_{L^p_x} \\
				&\leq C
				\left\| S_2^x f \right\|_{\widetilde L^r\left( \mathbb{R}^D_x; B^\alpha_{p,q}\left(\mathbb{R}^D_v\right)\right)}
				+ C
				\left\| S_2^x g \right\|_{\widetilde L^r\left( \mathbb{R}^D_x; B^\beta_{p,q}\left(\mathbb{R}^D_v\right)\right)}\\
				&\leq C
				\left\| f \right\|_{B^{a,\alpha}_{r,p,q}\left(\mathbb{R}^D_x\times\mathbb{R}^D_v\right)}
				+ C
				\left\| g \right\|_{B^{b,\beta}_{r,p,q}\left(\mathbb{R}^D_x\times\mathbb{R}^D_v\right)},
			\end{aligned}
		\end{equation}
		which concludes the estimate on the low frequencies.
		
		Regarding the high frequencies, we obtain
		\begin{equation}\label{high freq interpol}
			\begin{aligned}
				&\left\|
				\Delta_{2^k}^x\int_{\mathbb{R}^D} f(x,v)\phi(v) dv
				\right\|_{L^p_x} \\
				&\hspace{30mm} \leq
				\left\| A_{2^k}^t\left( \Delta_{\left[2^{k-1},2^{k+1}\right]}^x f\phi \right) \right\|_{L^p_x}
				+ t
				\left\| B_{2^k}^t\left(  \Delta_{\left[2^{k-1},2^{k+1}\right]}^x g\phi \right) \right\|_{L^p_x} \\
				&\hspace{30mm} \leq C
				{1\over t^{\alpha+1-\frac 1r + D\left(\frac 1r - \frac 1p\right)} 2^{k\left(\alpha+1-\frac 1r\right)}}
				\left\| \Delta_{\left[2^{k-1},2^{k+1}\right]}^x f \right\|_{\widetilde L^r\left( \mathbb{R}^D_x; B^\alpha_{p,q}\left(\mathbb{R}^D_v\right)\right)}\\
				&\hspace{30mm} + 
				C{ t^{ \frac 1r -\gamma - D\left(\frac 1r - \frac 1p\right)} \over 2^{k\left(\gamma+1-\frac 1r\right)}}
				\left\| \Delta_{\left[2^{k-1},2^{k+1}\right]}^x g \right\|_{\widetilde L^r\left( \mathbb{R}^D_x; B^\beta_{p,q}\left(\mathbb{R}^D_v\right)\right)}.
			\end{aligned}
		\end{equation}
		
		Next, optimizing in $t$ for each value of $k$, we fix the interpolation parameter $t$ as $t_k=2^{-k\frac{\left(\alpha-\gamma\right) + \left(a-b\right)}{1+\alpha-\gamma}}$. Note that $t_k\geq 2^{-k}$, for $b\geq a-1$, and that this choice is independent of $1\leq p, q, r \leq \infty$.

		Furthermore, in the cases $\beta>\frac 1r-D\left(\frac 1r-\frac{1}{p}\right)$ or $\beta=\frac 1r-D\left(\frac 1r-\frac{1}{p}\right)$ and $q=1$, we can choose $t=\infty$, which is, in fact, more optimal than $t2^k=2^{k\frac{1+b-a}{1+\alpha-\gamma}}$, since it eliminates the first term in the right-hand side of the above estimates. This case is discussed later on.

		Therefore, denoting $ s = \left(1+b-a\right) \frac{\alpha + 1 -\frac 1r + D \left( \frac 1r - \frac 1 p \right)}{1+\alpha-\gamma} +a - D \left( \frac 1r - \frac 1 p \right) $ and setting $t=t_k$, we find that
		\begin{equation}
			\begin{aligned}
				& 2^{ks} \left\|
				\Delta_{2^k}^x\int_{\mathbb{R}^D} f(x,v)\phi(v) dv
				\right\|_{L^p_x} \\
				& \leq C
				2^{ka} \left\| \Delta_{\left[2^{k-1},2^{k+1}\right]}^x f \right\|_{\widetilde L^r\left( \mathbb{R}^D_x; B^\alpha_{p,q}\left(\mathbb{R}^D_v\right)\right)}
				+ 
				C 2^{kb} \left\| \Delta_{\left[2^{k-1},2^{k+1}\right]}^x g \right\|_{\widetilde L^r\left( \mathbb{R}^D_x; B^\beta_{p,q}\left(\mathbb{R}^D_v\right)\right)}.
			\end{aligned}
		\end{equation}
		Hence, summing over $k$, we obtain
		\begin{equation}
			\begin{aligned}
				& \left\| \left\{
				2^{ks} \left\|
				\Delta_{2^k}^x\int_{\mathbb{R}^D} f(x,v)\phi(v) dv
				\right\|_{L^p_x} \right\}_{k=0}^\infty
				\right\|_{\ell^q} \\
				&\hspace{30mm} \leq C
				\left\| \left\{
				2^{ka} \left\| \Delta_{2^k}^x f \right\|_{\widetilde L^r\left( \mathbb{R}^D_x; B^\alpha_{p,q}\left(\mathbb{R}^D_v\right)\right)} \right\}_{k=-1}^\infty
				\right\|_{\ell^q} \\
				&\hspace{30mm} +
				C
				\left\|\left\{
				2^{kb} \left\| \Delta_{2^k}^x g \right\|_{\widetilde L^r\left( \mathbb{R}^D_x; B^\beta_{p,q}\left(\mathbb{R}^D_v\right)\right)}
				\right\}_{k=-1}^\infty\right\|_{\ell^q}\\
				&\hspace{30mm}=
				C \left\| f \right\|_{B^{a,\alpha}_{r,p,q}\left(\mathbb{R}^D_x\times\mathbb{R}^D_v\right)}
				+
				C \left\| g \right\|_{B^{b,\beta}_{r,p,q}\left(\mathbb{R}^D_x\times\mathbb{R}^D_v\right)},
			\end{aligned}
		\end{equation}
		which concludes the proof of the theorem in the case $\beta<\frac 1r-D\left(\frac 1r-\frac{1}{p}\right)$.

		We handle now the cases $\beta>\frac 1r-D\left(\frac 1r-\frac{1}{p}\right)$ or $\beta=\frac 1r-D\left(\frac 1r-\frac{1}{p}\right)$ and $q=1$, by letting $t$ tend to infinity in \eqref{high freq interpol}, as mentioned previously. This leads to
		\begin{equation}
			\left\|
			\Delta_{2^k}^x\int_{\mathbb{R}^D} f(x,v)\phi(v) dv
			\right\|_{L^p_x} \\
			\leq
			{C \over 2^{k\left(1-D\left(\frac 1r-\frac{1}{p}\right)\right)}}
			\left\| \Delta_{\left[2^{k-1},2^{k+1}\right]}^x g \right\|_{\widetilde L^r\left( \mathbb{R}^D_x; B^\beta_{p,q}\left(\mathbb{R}^D_v\right)\right)}.
		\end{equation}
		Hence, recalling $ s = 1+b - D \left( \frac 1r - \frac 1 p \right) $ and summing over $k$ yields
		\begin{equation}
			\begin{aligned}
				&\left\|\left\{
				2^{ks}\left\|\Delta_{2^k}^x\int_{\mathbb{R}^D} f(x,v) \phi(v) dv\right\|_{L^p_x}
				\right\}_{k=0}^\infty\right\|_{\ell^q} \\
				&\hspace{40mm} \leq
				C
				\left\|\left\{ 2^{kb}
				\left\| \Delta_{2^k}^x g \right\|_{\widetilde L^r \left( \mathbb{R}^D_x ; B^{\beta}_{p,q} \left(\mathbb{R}^D_v\right) \right)}
				\right\}_{k=-1}^\infty \right\|_{\ell^q},
			\end{aligned}
		\end{equation}
		which concludes the proof of the theorem in the cases $\beta>\frac 1r-D\left(\frac 1r-\frac{1}{p}\right)$ or $\beta=\frac 1r-D\left(\frac 1r-\frac{1}{p}\right)$ and $q=1$.

		As for the case $\beta=\frac 1r-D\left(\frac 1r-\frac{1}{p}\right)$ and $q\neq 1$, employing the corresponding estimate from Lemma \ref{operators norms interpolation}, we find that, as in the proofs of Theorems \ref{pre averaging lemma p 2} and \ref{classical}, for every $\epsilon>0$,
		\begin{equation}
			\begin{aligned}
				& \left\| \left\{
				2^{k(s-\epsilon)} \left\|
				\Delta_{2^k}^x\int_{\mathbb{R}^D} f(x,v)\phi(v) dv
				\right\|_{L^p_x} \right\}_{k=0}^\infty
				\right\|_{\ell^q} \\
				&\hspace{30mm}\leq
				C \left\| f \right\|_{B^{a,\alpha}_{r,p,q}\left(\mathbb{R}^D_x\times\mathbb{R}^D_v\right)}
				+
				C \left\| g \right\|_{B^{b,\beta}_{r,p,q}\left(\mathbb{R}^D_x\times\mathbb{R}^D_v\right)},
			\end{aligned}
		\end{equation}

		which concludes the proof of the theorem.
	\end{proof}

We proceed now to the proof of the most general Theorem \ref{main averaging lemma 2}, which will follow from a quite involved interpolation procedure. In particular, the following lemma is a generalization of Lemma \ref{operators norms interpolation}.

\begin{lem}\label{operators norms interpolation lorentz}
	Let $\phi(v)\in C_0^\infty\left(\mathbb{R}^D\right)$. For every $1\leq p, q, r\leq \infty$, $0<m< \infty$, $\alpha,\lambda\in\mathbb{R}$ and $k\in\mathbb{N}$ such that
	\begin{equation}
		r\leq p\leq r'
	\end{equation}
	and
	\begin{equation}
		\frac{1}{m} =\alpha+1-\frac 1r + D\left(\frac 1r - \frac 1p\right)-\lambda,
	\end{equation}
	it holds that
	\begin{equation}
		\left\| \mathbb{1}_{\left\{t\geq 2^{-k}\right\}} t^\lambda \left\|A_{2^k}^t\left(f\phi\right)\right\|_{L^p_x} \right\|_{L_t^{m,q}}
		\leq {C\over 2^{k\left(\alpha+1-\frac 1r\right)}} \left\| f \right\|_{\widetilde L^r\left( \mathbb{R}^D_x; B^{\alpha}_{p,q}\left(\mathbb{R}^D_v\right)\right)},
	\end{equation}
	where $C>0$ is independent of $2^k$.
	
	Furthermore, if
	\begin{equation}
		D\left(\frac 1r-\frac 1p\right)<\frac 2r -1\qquad \text{or}\qquad p=r=2,
	\end{equation}
	then, for any $0<n < \infty$ and $\beta,\tau\in\mathbb{R}$ such that
	\begin{equation}
		\beta<\frac 1r-D\left(\frac 1r-\frac{1}{p}\right)
	\end{equation}
	and
	\begin{equation}
			\frac{1}{n} =\beta+1-\frac 1r + D\left(\frac 1r - \frac 1p\right)-\tau,
	\end{equation}
	it holds that
	\begin{equation}
		\left\| \mathbb{1}_{\left\{t\geq 2^{-k}\right\}} t^\tau \left\|B_{2^k}^t\left(g\phi\right)\right\|_{L^p_x} \right\|_{L_t^{n,q}}
		\leq {C\over 2^{k\left(\beta+1-\frac 1r\right)}} \left\| g \right\|_{\widetilde L^r\left( \mathbb{R}^D_x; B^{\beta}_{p,q}\left(\mathbb{R}^D_v\right)\right)},
	\end{equation}
	where $C>0$ is independent of $2^k$.
\end{lem}

\begin{proof}
	We proceed by interpolation of the estimates from Lemma \ref{operators norms interpolation} on the velocity regularity index, which in particular illustrates the importance of systematically dealing with the most general cases of function spaces as possible.
	
	To this end, consider $\alpha_0<\alpha<\alpha_1$, $\beta_0<\beta<\beta_1$, $0<m_1<m<m_0<\infty$ and $0<n_1<n<n_0<\infty$, such that
	\begin{equation}
		\alpha={\alpha_0+\alpha_1\over 2}, \qquad \beta={\beta_0+\beta_1\over 2}, \qquad
		\beta_1<\frac 1r-D\left(\frac 1r-\frac{1}{p}\right),
	\end{equation}
	and, for each $i=0,1$,
	\begin{equation}
		\begin{aligned}
			\frac{1}{m_i} & =\alpha_i+1-\frac 1r + D\left(\frac 1r - \frac 1p\right)-\lambda, \\
			\frac{1}{n_i} & =\beta_i+1-\frac 1r + D\left(\frac 1r - \frac 1p\right)-\tau.
		\end{aligned}
	\end{equation}
	
	Then, using Lemma \ref{operators norms interpolation}, it holds that, for each $i=0,1$ and every $t\geq 2^{-k}$,
	\begin{equation}
			\left\|A_{2^k}^t\left(f\phi\right)\right\|_{L^p_x}
			\leq {C\over t^{\alpha_i+1-\frac 1r + D\left(\frac 1r - \frac 1p\right)} 2^{k\left(\alpha_i+1-\frac 1r\right)}} \left\| f \right\|_{\widetilde L^r\left( \mathbb{R}^D_x; B^{\alpha_i}_{p,q}\left(\mathbb{R}^D_v\right)\right)}
	\end{equation}
	and, if further $D\left(\frac 1r-\frac 1p\right)<\frac 2r -1$ or $p=r=2$,
	\begin{equation}
			\left\|B_{2^k}^t\left(g\phi\right)\right\|_{L^p_x}
			\leq {C\over t^{\beta_i+1-\frac 1r + D\left(\frac 1r - \frac 1p\right)} 2^{k\left(\beta_i+1-\frac 1r\right)}} \left\| g \right\|_{\widetilde L^r\left( \mathbb{R}^D_x; B^{\beta_i}_{p,q}\left(\mathbb{R}^D_v\right)\right)}.
	\end{equation}
	It follows that
	\begin{equation}\label{to interpolate}
		\begin{aligned}
			\left\| \mathbb{1}_{\left\{t\geq 2^{-k}\right\}}  t^\lambda \left\|A_{2^k}^t\left(f\phi\right)\right\|_{L^p_x} \right\|_{L_t^{m_i,\infty}}
			& \leq {C\over 2^{k\left(\alpha_i+1-\frac 1r\right)}} \left\| f \right\|_{\widetilde L^r\left( \mathbb{R}^D_x; B^{\alpha_i}_{p,q}\left(\mathbb{R}^D_v\right)\right)}, \\
			\left\| \mathbb{1}_{\left\{t\geq 2^{-k}\right\}} t^\tau \left\|B_{2^k}^t\left(g\phi\right)\right\|_{L^p_x} \right\|_{L_t^{n_i,\infty}}
			& \leq {C\over 2^{k\left(\beta_i+1-\frac 1r\right)}} \left\| g \right\|_{\widetilde L^r\left( \mathbb{R}^D_x; B^{\beta_i}_{p,q}\left(\mathbb{R}^D_v\right)\right)},
		\end{aligned}
	\end{equation}
	where $L^{p,q}$ denotes the standard Lorentz spaces.
	
	Recall now that the real interpolation of Lorentz spaces (cf. \cite{bergh}) yields, in particular, that $\left(L^{m_0,\infty}A,L^{m_1,\infty}A\right)_{\frac 12,q}=L^{m,q}A$, where $A$ is any fixed Banach space, for any $0<m,m_0,m_1,q\leq \infty$ such that $\frac 1m=\frac 12\left(\frac 1{m_0}+\frac 1{m_1}\right)$ and $m_0\neq m_1$.
	
	Furthermore, the real interpolation of standard Besov spaces (cf. \cite{bergh}) yields, in particular, that $\left(B^{\alpha_0}_{p,q},B^{\alpha_1}_{p,q}\right)_{\frac 12,c}=B^{\alpha}_{p,c}$, for any $\alpha,\alpha_0,\alpha_1\in\mathbb{R}$ and $1\leq p,q,c\leq\infty$ such that $\alpha=\frac{\alpha_0+\alpha_1}{2}$ and $\alpha_0\neq\alpha_1$. It is then possible to smoothly adapt the proof of this standard property to obtain a corresponding result for the Besov spaces introduced in Section \ref{LP decomposition}. Namely, one can show that $\left(\widetilde L^r B^{\alpha_0}_{p,q},\widetilde L^r B^{\alpha_1}_{p,q}\right)_{\frac 12,c}=\widetilde L^r B^{\alpha}_{p,c}$, for any $\alpha,\alpha_0,\alpha_1\in\mathbb{R}$ and $1\leq p,q,c\leq\infty$ such that $\alpha=\frac{\alpha_0+\alpha_1}{2}$ and $\alpha_0\neq\alpha_1$.
	
	Therefore, we deduce from \eqref{to interpolate} that
	\begin{equation}
		\begin{aligned}
			\left\| \mathbb{1}_{\left\{t\geq 2^{-k}\right\}} t^\lambda \left\|A_{2^k}^t\left(f\phi\right)\right\|_{L^p_x} \right\|_{L_t^{m,q}}
			& \leq {C\over 2^{k\left(\alpha+1-\frac 1r\right)}} \left\| f \right\|_{\widetilde L^r\left( \mathbb{R}^D_x; B^{\alpha}_{p,q}\left(\mathbb{R}^D_v\right)\right)}, \\
			\left\| \mathbb{1}_{\left\{t\geq 2^{-k}\right\}} t^\tau \left\|B_{2^k}^t\left(g\phi\right)\right\|_{L^p_x} \right\|_{L_t^{n,q}}
			& \leq {C\over 2^{k\left(\beta+1-\frac 1r\right)}} \left\| g \right\|_{\widetilde L^r\left( \mathbb{R}^D_x; B^{\beta}_{p,q}\left(\mathbb{R}^D_v\right)\right)},
		\end{aligned}
	\end{equation}
	which concludes the proof of the lemma.
	
	Note that, unfortunately, in the case $\beta\geq\frac 1r-D\left(\frac 1r-\frac{1}{p}\right)$, we cannot use interpolation methods to improve the results of Lemma \ref{operators norms interpolation} because $\gamma-\frac 1r+D\left(\frac 1r-\frac{1}{p}\right)=0$, where $\gamma=\min\left\{\beta,\frac 1r-D\left(\frac 1r-\frac{1}{p}\right)\right\}$, remains constant over that range of parameters.
\end{proof}

\begin{proof}[Proof of Theorem \ref{main averaging lemma 2}]
	As usual, we begin with the decompositions
	\begin{equation}
		\Delta_{0}^x\int_{\mathbb{R}^D} f(x,v)\phi(v) dv
		= A_{0}^1 (f\phi) (x) + B_{0}^1 (g\phi) (x)
	\end{equation}
	and
	\begin{equation}
			\Delta_{2^k}^x\int_{\mathbb{R}^D} f(x,v) \phi(v) dv
			= A_{2^k}^t (f\phi) (x) + t B_{2^k}^t (g\phi) (x).
	\end{equation}
	provided by Proposition \ref{crucial}.
	
	We have shown in Lemma \ref{operators norms interpolation lorentz} above that the operators $A_{2^k}^t$ and $B_{2^k}^t$ satisfy the bounds
	\begin{equation}\label{A B norm}
		\begin{aligned}
			\left\| \mathbb{1}_{\left\{t\geq 2^{-k}\right\}} t^{\lambda} \left\|A_{2^k}^t\left(f\phi\right)\right\|_{L^{p_0}_x} \right\|_{L_t^{q_0}}
			& \leq {C\over 2^{k\left(\alpha+1-\frac 1{r_0}\right)}} \left\| f \right\|_{\widetilde L^{r_0}\left( \mathbb{R}^D_x; B^{\alpha}_{p_0,q_0}\left(\mathbb{R}^D_v\right)\right)}, \\
			\left\| \mathbb{1}_{\left\{t\geq 2^{-k}\right\}} t^{\tau} \left\|B_{2^k}^t\left(g\phi\right)\right\|_{L^{p_1}_x} \right\|_{L_t^{q_1}}
			& \leq {C\over 2^{k\left(\beta+1-\frac 1{r_1}\right)}} \left\| g \right\|_{\widetilde L^{r_1}\left( \mathbb{R}^D_x; B^{\beta}_{p_1,q_1}\left(\mathbb{R}^D_v\right)\right)},
		\end{aligned}
	\end{equation}
	where $C>0$ is independent of $2^k$ and
	\begin{equation}
		\begin{aligned}
			\lambda & =\alpha+1-\frac 1{r_0} + D\left(\frac 1{r_0} - \frac 1{p_0}\right) - \frac 1{q_0}, \\
			\tau & =\beta+1-\frac 1{r_1} + D\left(\frac 1{r_1} - \frac 1{p_1}\right) - \frac 1{q_1}.
		\end{aligned}
	\end{equation}
	
	We will now make use of the construction of interpolation spaces known as \emph{espaces de moyennes} presented in section \ref{interpolation}. Specifically, in virtue of the property $\left(L^{p_0},L^{p_1}\right)_{\theta,p}=L^p$ valid for all $1\leq p,p_0,p_1\leq\infty$ and $0<\theta<1$ such that $\frac 1p=\frac{1-\theta}{p_0}+\frac{\theta}{p_1}$, we wish to express the Lebesgue space $L^p$ using the norm \eqref{espace moyenne}. That is to say, we will employ the norm
	\begin{equation}
		\inf_{a=a_0+a_1} \left( \left\|s^{-\theta} a_0(s)\right\|^{q_0}_{L^{q_0}\left((0,\infty),\frac{ds}{s};L^{p_0}(\mathbb{R}^D,dx)\right)} + \left\|s^{1-\theta} a_1(s)\right\|^{q_1}_{L^{q_1}\left((0,\infty),\frac{ds}{s};L^{p_1}(\mathbb{R}^D,dx)\right)} \right)^\frac{1}{p},
	\end{equation}
	which is equivalent to the usual norm on $L^p\left(\mathbb{R}^D,dx\right)$, where $\frac 1p=\frac{1-\theta}{p_0}+\frac{\theta}{p_1}$, provided $\frac 1p=\frac{1-\theta}{q_0}+\frac{\theta}{q_1}$ and $q_0,q_1<\infty$.
	
	To this end, for some suitable bijective function $t(s):(0,\infty)\rightarrow(0,\infty)$ to be determined later on, we will decompose
	\begin{equation}
			\Delta_{2^k}^x\int_{\mathbb{R}^D} f(x,v) \phi(v) dv = a = a_0(s) + a_1(s)
			= \Delta_{2^k}^x\int_{\mathbb{R}^D} f(x,v) \phi(v) dv + 0,
	\end{equation}
	when $0<t(s)< 2^{-k}$, and
	\begin{equation}
			\Delta_{2^k}^x\int_{\mathbb{R}^D} f(x,v) \phi(v) dv = a = a_0(s) + a_1(s)
			= A_{2^k}^t (f\phi) (x) + t B_{2^k}^t (g\phi) (x),
	\end{equation}
	when $t(s)\geq 2^{-k}$. It follows that, using Bernstein's inequalities,
	\begin{equation}
		\begin{aligned}
			& \left\|\Delta_{2^k}^x\int_{\mathbb{R}^D} f(x,v) \phi(v)dv\right\|_{L^p_x}
			\\
			&  \leq 
			C \left\|\mathbb{1}_{\left\{0<t< 2^{-k}\right\}}s^{-\theta-\frac{1}{q_0}} \left\|\Delta_{2^k}^x\int_{\mathbb{R}^D} f(x,v) \phi(v)dv\right\|_{L^{p_0}_x}\right\|^\frac{q_0}{p}_{L^{q_0}_s}
			\\
			&  + C \left\|\mathbb{1}_{\left\{t\geq 2^{-k}\right\}}s^{-\theta-\frac{1}{q_0}} \left\|A_{2^k}^t\left(f\phi\right)\right\|_{L^{p_0}_x}\right\|^\frac{q_0}{p}_{L^{q_0}_s} + C \left\|\mathbb{1}_{\left\{t\geq 2^{-k}\right\}}s^{1-\theta-\frac{1}{q_1}} t\left\|B_{2^k}^t\left(g\phi\right)\right\|_{L^{p_1}_x}\right\|^\frac{q_1}{p}_{L^{q_1}_s}\\
			\\
			&  \leq 
			C 2^{kD\left(\frac{1}{r_0}-\frac{1}{p_0}\right)\frac{q_0}{p}} \left\|\mathbb{1}_{\left\{0<t< 2^{-k}\right\}}s(t)^{-\theta}\left({s'(t)\over s(t)}\right)^\frac{1}{q_0} \right\|_{L^{q_0}_t}^\frac{q_0}{p}
			\left\|\Delta_{2^k}^x\int_{\mathbb{R}^D} f(x,v) \phi(v)dv\right\|_{L^{r_0}_x}^\frac{q_0}{p}
			\\
			&  + C \left\|\mathbb{1}_{\left\{t\geq 2^{-k}\right\}}s(t)^{-\theta}\left({s'(t)\over s(t)}\right)^\frac{1}{q_0} \left\|A_{2^k}^t\left(f\phi\right)\right\|_{L^{p_0}_x}\right\|^\frac{q_0}{p}_{L^{q_0}_t}
			\\
			&  + C \left\|\mathbb{1}_{\left\{t\geq 2^{-k}\right\}}s(t)^{1-\theta}\left({s'(t)\over s(t)}\right)^\frac{1}{q_1} t\left\|B_{2^k}^t\left(g\phi\right)\right\|_{L^{p_1}_x}\right\|^\frac{q_1}{p}_{L^{q_1}_t}.
		\end{aligned}
	\end{equation}
	
	Next, we set the dependence of $s$ with respect to $t$ so that we may utilize the estimates \eqref{A B norm}. This amounts to optimizing the value of the interpolation parameter $s$ for each value of $t$, which also dictates the value of $0<\theta<1$. That is to say, for given coefficients $c_1,c_2>0$ independent of $k$, we wish to find an optimal value for a function of the form
	\begin{equation}
		c_1\frac{s(t)^{-\theta}}{2^{k\left(\alpha+1-\frac 1{r_0}\right)}t^\lambda}
		\left({s'(t)\over s(t)}\right)^\frac{1}{q_0}
		+ c_2 \frac{s(t)^{1-\theta}}{2^{k\left(\beta+1-\frac 1{r_1}\right)}t^\tau}
		\left({s'(t)\over s(t)}\right)^\frac{1}{q_1}t.
	\end{equation}
	Thus, we define
	\begin{equation}
		\begin{aligned}
			s(t)&=\frac{2^{k\left(\beta+b+1-\frac 1{r_1}\right)}t^{\tau+\frac{1}{q_1}-1}}{2^{k\left(\alpha+a+1-\frac 1{r_0}\right)}t^{\lambda+\frac{1}{q_0}}}\\
			& =
			2^{k\left((\beta+b)-(\alpha+a)+\frac{1}{r_0}-\frac{1}{r_1}\right)}
			t^{\beta-\alpha-1+\frac{1}{r_0}-\frac{1}{r_1}+D\left(\frac{1}{r_1}-\frac{1}{r_0}-\frac{1}{p_1}+\frac{1}{p_0}\right)},
		\end{aligned}
	\end{equation}
	which is admissible since $\beta-\alpha-1+\frac{1}{r_0}-\frac{1}{r_1}+D\left(\frac{1}{r_1}-\frac{1}{r_0}-\frac{1}{p_1}+\frac{1}{p_0}\right)< 0$. Furthermore, in order that the resulting terms $\frac{s(t)^{-\theta}}{t^{\lambda + \frac{1}{q_0}}}$ and $\frac{s(t)^{1-\theta}}{t^{\tau + \frac{1}{q_1}-1}}$ be independent of $t$, we set $0<\theta<1$ so that $(1-\theta)\left(\lambda + \frac{1}{q_0}\right)+\theta\left(\tau+\frac{1}{q_1} -1\right)=0$, which results in
	\begin{equation}
			\theta 
			= \frac{\alpha + 1 -\frac{1}{r_0} + D\left(\frac{1}{r_0}-\frac{1}{p_0}\right)}
			{\alpha + 1 -\frac{1}{r_0} + D\left(\frac{1}{r_0}-\frac{1}{p_0}\right)
			-\beta+\frac{1}{r_1}-D\left(\frac{1}{r_1}-\frac{1}{p_1}\right)}.
	\end{equation}
	To be precise, these choices of parameters yield that
	\begin{equation}
		\begin{aligned}
			s(t)^{-\theta}
			& = 2^{-ks} 2^{k\left(a+\alpha+1-\frac 1{r_0}\right)}t^{\lambda+\frac 1{q_0}},
			\\
			s(t)^{1-\theta}
			& = 2^{-ks} 2^{k\left(b+\beta+1-\frac 1{r_1}\right)}t^{\tau-1+\frac 1 {q_1}},
		\end{aligned}
	\end{equation}
	where $s = (1-\theta)\left(\alpha+a+1-\frac 1{r_0}\right)+\theta\left(\beta+b+1-\frac 1{r_1}\right)$.
	
	We conclude that, in virtue of the identities \eqref{crucial 3} from Proposition \ref{crucial},
	\begin{equation}
		\begin{aligned}
			\left(2^{ks}\left\|\Delta_{2^k}^x\int_{\mathbb{R}^D} f(x,v) \phi(v)dv\right\|_{L^p_x}\right)^p
			&\leq
			C \left(2^{ka} \left\| \Delta_{\left[2^{k-1},2^{k+1}\right]}^x f \right\|_{\widetilde L^{r_0}\left( \mathbb{R}^D_x; B^{\alpha}_{p_0,q_0}\left(\mathbb{R}^D_v\right)\right)}\right)^{q_0}
			\\
			& + C \left(2^{kb} \left\| \Delta_{\left[2^{k-1},2^{k+1}\right]}^x g \right\|_{\widetilde L^{r_1}\left( \mathbb{R}^D_x; B^{\beta}_{p_1,q_1}\left(\mathbb{R}^D_v\right)\right)}\right)^{q_1}.
		\end{aligned}
	\end{equation}
	Hence, summing over $k$, we obtain
	\begin{equation}
		\begin{aligned}
			& \left\| \left\{
			2^{ks} \left\|
			\Delta_{2^k}^x\int_{\mathbb{R}^D} f(x,v)\phi(v) dv
			\right\|_{L^p_x} \right\}_{k=0}^\infty
			\right\|_{\ell^p} \\
			&\hspace{30mm} \leq C
			\left\| \left\{
			2^{ka} \left\| \Delta_{2^k}^x f \right\|_{\widetilde L^{r_0}\left( \mathbb{R}^D_x; B^{\alpha}_{p_0,q_0}\left(\mathbb{R}^D_v\right)\right)} \right\}_{k=-1}^\infty
			\right\|_{\ell^{q_0}}^\frac{q_0}{p} \\
			&\hspace{30mm} +
			C
			\left\|\left\{
			2^{kb} \left\| \Delta_{2^k}^x g \right\|_{\widetilde L^{r_1}\left( \mathbb{R}^D_x; B^{\beta}_{p_1,q_1}\left(\mathbb{R}^D_v\right)\right)}
			\right\}_{k=-1}^\infty\right\|_{\ell^{q_1}}^\frac{q_1}{p}\\
			&\hspace{30mm}\leq
			C \left\| f \right\|_{B^{a,\alpha}_{r_0,p_0,q_0}\left(\mathbb{R}^D_x\times\mathbb{R}^D_v\right)}^\frac{q_0}{p}
			+
			C \left\| g \right\|_{B^{b,\beta}_{r_1,p_1,q_1}\left(\mathbb{R}^D_x\times\mathbb{R}^D_v\right)}^\frac{q_1}{p}.
		\end{aligned}
	\end{equation}
	Now, if the above estimate holds true for any $f$ and $g$, then it must also be valid for all $\lambda f$ and $\lambda g$, where $\lambda>0$, so that
	\begin{equation}
		\begin{aligned}
			&  \left\| \left\{
			2^{ks} \left\|
			\Delta_{2^k}^x\int_{\mathbb{R}^D} f(x,v)\phi(v) dv
			\right\|_{L^p_x} \right\}_{k=0}^\infty
			\right\|_{\ell^p} \\
			&\hspace{30mm}\leq
			C \lambda^{\frac{q_0}{p}-1} \left\| f \right\|_{B^{a,\alpha}_{r_0,p_0,q_0}\left(\mathbb{R}^D_x\times\mathbb{R}^D_v\right)}^\frac{q_0}{p}
			+
			C \lambda^{\frac{q_1}{p}-1} \left\| g \right\|_{B^{b,\beta}_{r_1,p_1,q_1}\left(\mathbb{R}^D_x\times\mathbb{R}^D_v\right)}^\frac{q_1}{p}.
		\end{aligned}
	\end{equation}
	Recalling that $\frac 1p=\frac{1-\theta}{q_0}+\frac{\theta}{q_1}$ and optimizing in $\lambda$ concludes the proof of the main estimate for the high frequencies of the velocity average.
	
	Thus, there only remains to control the low frequencies $\Delta_{0}^x\int_{\mathbb{R}^D} f(x,v)\phi(v) dv$. To this end, if $p\geq r_0$, a straightforward application of Bernstein's inequalities shows that
	\begin{equation}\label{low freq easy}
		\begin{aligned}
			& \left\|\Delta_{0}^x\int_{\mathbb{R}^D} f(x,v)\phi(v) dv\right\|_{L^p_x}\\
			& \leq C
			\left\|\int_{\mathbb{R}^D} \Delta_{0}^xf(x,v)\phi(v) dv\right\|_{L^{r_0}_x} \\
			& = C
			\left\|\int_{\mathbb{R}^D} \Delta_{0}^x \Delta_{0}^v f(x,v)S_2^v\phi(v) dv
			+ \sum_{k=0}^\infty \int_{\mathbb{R}^D} \Delta_{0}^x \Delta_{2^k}^v f(x,v) \Delta_{\left[2^{k-1},2^{k+1}\right]}^v \phi(v) dv\right\|_{L^{r_0}_x} \\
			& \leq C
			\left\|\Delta_{0}^x \Delta_{0}^v f\right\|_{L^{r_0}_xL^{p_0}_v}
			\left\|S_2^v\phi\right\|_{L^{p_0'}_v}
			+ C \sum_{k=0}^\infty
			\left\| \Delta_{0}^x \Delta_{2^k}^v f \right\|_{L^{r_0}_xL^{p_0}_v}
			\left\|\Delta_{\left[2^{k-1},2^{k+1}\right]}^v \phi\right\|_{L^{p_0'}_v} \\
			& \leq C
			\left\| f \right\|_{B^{a,\alpha}_{r_0,p_0,q_0}\left(\mathbb{R}^D_x\times\mathbb{R}^D_v\right)}
			\left\| \phi \right\|_{B^{-\alpha}_{p_0',q_0'}\left(\mathbb{R}^D_v\right)}.
		\end{aligned}
	\end{equation}
	
	Finally, in the case $p < r_0$, the direct use of Bernstein's inequalities as above is not allowed and it is therefore necessary to localize the norm in space in order to carry out the preceding argument. Thus, for any $\chi(x)\in C_0^\infty\left(\mathbb{R}^D\right)$, we employ the standard methods of paradifferential calculus, which are used in the proof of Lemma \ref{paradifferential}, to decompose, following Bony's method,
	\begin{equation}
		\Delta_0^x \left(f(x,v)\chi(x)\right)=
		\Delta_0^x\left[
		\Delta_0^x f S_4^x \chi + \Delta_1^x f S_8^x \chi + \Delta_2^x f S_{16}^x \chi
		+ \sum_{k=2}^\infty \Delta_{2^k}^x f \Delta_{\left[2^{k-2},2^{k+2}\right]}^x \chi
		\right].
	\end{equation}
	Since $\chi$ is rapidly decaying, it then follows that, repeating the estimates from \eqref{low freq easy},
	\begin{equation}
		\begin{aligned}
			& \left\|\Delta_{0}^x\int_{\mathbb{R}^D} f(x,v)\chi(x)\phi(v) dv\right\|_{L^p_x}\\
			& \leq C \left\|\int_{\mathbb{R}^D} \left[
			\Delta_0^x f S_4^x \chi + \Delta_1^x f S_8^x \chi + \Delta_2^x f S_{16}^x \chi
			+ \sum_{k=2}^\infty \Delta_{2^k}^x f \Delta_{\left[2^{k-2},2^{k+2}\right]}^x \chi
			\right]\phi(v) dv
			\right\|_{L^p_x}\\
			& \leq C
			\left\|\int_{\mathbb{R}^D} \Delta_{0}^xf(x,v)\phi(v) dv\right\|_{L^{r_0}_x}
			+ C \sum_{k=0}^\infty 2^{ka}
			\left\|\int_{\mathbb{R}^D} \Delta_{2^k}^xf(x,v)\phi(v) dv\right\|_{L^{r_0}_x} \\
			& \leq C
			\left\| f \right\|_{B^{a,\alpha}_{r_0,p_0,q_0}\left(\mathbb{R}^D_x\times\mathbb{R}^D_v\right)}
			\left\| \phi \right\|_{B^{-\alpha}_{p_0',q_0'}\left(\mathbb{R}^D_v\right)},
		\end{aligned}
	\end{equation}
	which concludes the proof of the theorem.
\end{proof}

\appendix
\section{Some paradifferential calculus}

For the sake of completeness and convenience of the reader, we include below a technical lemma on basic paradifferential calculus.

\begin{lem}\label{paradifferential}
	Let $1\leq p,q,r\leq \infty$, $s\in\mathbb{R}$ and $\phi(v)\in\mathcal{S}\left(\mathbb{R}^D\right)$. Then, for any $f(x,v) \in \widetilde L^r \left( \mathbb{R}^D_x ; B^{s}_{p,q} \left(\mathbb{R}^D_v\right) \right)$, it holds that
	\begin{equation}\label{paradifferential 1}
		\left\| f \phi \right\|_{\widetilde L^r \left( \mathbb{R}^D_x ; B^{s}_{p,q} \left(\mathbb{R}^D_v\right) \right)}
		\leq C
		\left\| f \right\|_{\widetilde L^r \left( \mathbb{R}^D_x ; B^{s}_{p,q} \left(\mathbb{R}^D_v\right) \right)},
	\end{equation}
	where $C>0$ only depends on $\phi$ and on fixed parameters.
\end{lem}

\begin{proof}
	We begin formally with the standard Bony's decomposition (cf. \cite{chemin})
	\begin{equation}\label{Bony}
		f\phi=\left( \Delta_0^v f+\sum_{k=0}^\infty\Delta_{2^k}^vf \right)\left( \Delta_0^v \phi+\sum_{k=0}^\infty\Delta_{2^k}^v\phi \right)
		= T\left(f,\phi\right) + T\left(\phi,f\right) + R\left(f,\phi\right),
	\end{equation}
	where we have denoted the paraproducts
	\begin{equation}
		\begin{gathered}
			T\left(f,\phi\right) = \sum_{k=2}^\infty \Delta_{2^k}^v f S_{2^{k-3}}^v \phi \quad
			\text{and}\quad T\left(\phi,f\right) = \sum_{k=2}^\infty \Delta_{2^k}^v \phi S_{2^{k-3}}^v f,
		\end{gathered}
	\end{equation}
	and the remainder
	\begin{equation}
		\begin{gathered}
			R\left(f,\phi\right)=\Delta_0^v f R_0^v \phi + \sum_{k=0}^\infty \Delta_{2^k}^v f R_{2^k}^v \phi, \\
			\text{where}\quad R_0^v \phi = \Delta_0^v \phi + \sum_{j=0}^1 \Delta_{2^{j}}^v \phi,\quad
			R_1^v \phi = \Delta_0^v \phi + \sum_{j=0}^2 \Delta_{2^{j}}^v \phi,\\
			R_2^v \phi = \Delta_0^v \phi + \sum_{j=0}^3 \Delta_{2^{j}}^v \phi
			\quad \text{and}\quad R_{2^k}^v \phi =\sum_{j=-2}^2\Delta_{2^{k+j}}^v \phi \quad \text{if}\quad k\geq 2.
		\end{gathered}
	\end{equation}
	We then estimate $T\left(f,\phi\right)$, $T\left(\phi,f\right)$ and $R\left(f,\phi\right)$ separately.

	The control of the paraproduct $T\left(f,\phi\right)$ proceeds as follows. First, notice that the support of the Fourier transform in the velocity variable of $S_{k-3}^v \phi$ is contained inside a closed ball of radius $2^{k-2}$ centered at the origin and is thus separated from the support of the Fourier transform of $\Delta_{2^k}^v f$ by a distance of at least $2^{k-2}$. Therefore, the Fourier transform of the general term in the sum of the paraproduct is supported inside an annulus centered at the origin of inner radius $2^{k-2}$ and outer radius $9\cdot 2^{k-2}$. As a consequence, we have that
	\begin{equation}
		\begin{aligned}
			\Delta_{0}^v\left[ \Delta_{2^j}^v f S_{2^{j-3}}^v \phi \right] & \equiv 0
			&& \text{for all}\quad j\geq 2\\
			\text{and}\quad
			\Delta_{2^k}^v\left[ \Delta_{2^j}^v f S_{2^{j-3}}^v \phi \right] & \equiv 0
			&& \text{if}\quad |j-k|\geq 3.
		\end{aligned}
	\end{equation}
	We may then proceed with the estimation
	\begin{equation}\label{paraproduct 1}
		\begin{aligned}
			\left\| T\left(f,\phi\right) \right\|_{\widetilde L^r\left(\mathbb{R}_x^D;B_{p,q}^s\left(\mathbb{R}^D_v\right)\right)}
			& =
			\left\|\left\{2^{ks}\left\| \Delta_{2^k}^v T\left(f,\phi\right) \right\|_{L^r_xL^p_v}\right\}_{k=0}^\infty\right\|_{\ell^q}\\
			& \leq \left\| \left\{ \sum_{\substack{j=2\\|j-k|\leq 2}}^\infty 2^{(k-j)s} 2^{js} \left\| \Delta_{2^j}^v f S_{2^{j-3}}^v \phi \right\|_{L_x^rL^p_v} \right\}_{k=0}^\infty \right\|_{\ell^q}\\
			& \leq C
			\left\| \left\{ 2^{ks} \left\| \Delta_{2^k}^v f S_{2^{k-3}}^v \phi \right\|_{L_x^rL^p_v} \right\}_{k=2}^\infty \right\|_{\ell^q} \\
			& \leq C \left\| \phi \right\|_{L^\infty_v}
			\left\| \left\{ 2^{ks} \left\| \Delta_{2^k}^v f \right\|_{L_x^rL^p_v} \right\}_{k=2}^\infty \right\|_{\ell^q}\\
			& \leq C
			\left\| f \right\|_{\widetilde L^r \left( \mathbb{R}^D_x ; B^{s}_{p,q} \left(\mathbb{R}^D_v\right) \right)}.
		\end{aligned}
	\end{equation}
	
	Regarding the paraproduct $T\left(\phi, f\right)$, we handle it through the following similar calculation, using that $\phi(v)$ is rapidly decaying,
	\begin{equation}\label{paraproduct 2}
		\begin{aligned}
			& \left\| T\left(\phi,f\right) \right\|_{\widetilde L^r\left(\mathbb{R}_x^D;B_{p,q}^s\left(\mathbb{R}^D_v\right)\right)}\\
			& \leq C
			\left\| \left\{ 2^{ks} \left\| \Delta_{2^k}^v \phi S_{2^{k-3}}^v f \right\|_{L_x^rL^p_v} \right\}_{k=2}^\infty \right\|_{\ell^q}\\
			& \leq C
			\left\| \left\{ 2^{ks} \left\| \Delta_{2^k}^v \phi \right\|_{L^\infty_v} \left\| S_{2^{k-3}}^v f \right\|_{L_x^rL^p_v} \right\}_{k=2}^\infty \right\|_{\ell^q}\\
			&\leq C
			\left\| \left\{ 2^{ks} \left\| \Delta_{2^k}^v \phi \right\|_{L^\infty_v}
			\right\}_{k=2}^\infty \right\|_{\ell^q}
			\left\| \Delta_{0}^v f \right\|_{L_x^rL^p_v}
			\\
			& +
			\left\| \left\{ 2^{ks} \left\| \Delta_{2^k}^v \phi \right\|_{L^\infty_v}
			\sum_{j=0}^{k-3} 2^{-j(s\wedge 0)} 2^{j(s\wedge 0)} \left\| \Delta_{2^j}^v f \right\|_{L_x^rL^p_v}
			\right\}_{k=3}^\infty \right\|_{\ell^q}
			\\
			&\leq C
			\left\| \left\{ 2^{ks} \left\| \Delta_{2^k}^v \phi \right\|_{L^\infty_v}
			\right\}_{k=2}^\infty \right\|_{\ell^q}
			\left\| \Delta_{0}^v f \right\|_{L_x^rL^p_v}\\
			& +
			\left\| \left\{ 2^{k(s+1)-k(s\wedge0)} \left\| \Delta_{2^k}^v \phi \right\|_{L^\infty_v}
			\right\}_{k=3}^\infty \right\|_{\ell^q}
			\left\| \left\{ 2^{k(s\wedge 0)} \left\| \Delta_{2^k}^v f \right\|_{L_x^rL^p_v}
			\right\}_{k=0}^\infty \right\|_{\ell^q}
			\\
			& \leq C
			\left\| f \right\|_{\widetilde L^r \left( \mathbb{R}^D_x ; B^{s\wedge 0}_{p,q} \left(\mathbb{R}^D_v\right) \right)}.
		\end{aligned}
	\end{equation}

	Regarding the remainder, we first notice that the Fourier transforms of $\Delta_{0}^v f R_{0}^v \phi$ and $\Delta_{2^k}^v f R_{2^k}^v \phi$, for all $k\in\mathbb{N}$, are supported inside closed balls centered at the origin of respective radii $5$ and $10\cdot2^k$. It follows that
	\begin{equation}
		\begin{aligned}
			\Delta_{2^k}^v\left[ \Delta_{0}^v f R_{0}^v \phi \right] & \equiv 0
			&& \text{if}\quad k\geq 4\\
			\text{and}\quad
			\Delta_{2^k}^v\left[ \Delta_{2^j}^v f R_{2^j}^v \phi \right] & \equiv 0 && \text{if}\quad k\geq j+5.
		\end{aligned}
	\end{equation}
	Therefore, using once again that $\phi\in\mathcal{S}\left(\mathbb{R}^D\right)$,
	\begin{equation}\label{remainder}
		\begin{aligned}
			& \left\| R\left(f,\phi\right) \right\|_{\widetilde L^r\left(\mathbb{R}_x^D;B_{p,q}^s\left(\mathbb{R}^D_v\right)\right)}\\
			& \leq
			\left\| \Delta_{0}^v R\left(f,\phi\right) \right\|_{L^r_xL^p_v}+
			\left\|\left\{2^{ks}\left\| \Delta_{2^k}^v R\left(f,\phi\right) \right\|_{L^r_xL^p_v}\right\}_{k=0}^\infty\right\|_{\ell^q}\\
			& \leq
			C\left(
			\left\| \Delta_{0}^v f R_{0}^v \phi \right\|_{L^r_xL^p_v}
			+\sum_{k=0}^\infty
			\left\| \Delta_{2^k}^v f R_{2^k}^v \phi
			\right\|_{L^r_xL^p_v}
			\right)\\
			& +C
			\left\|\left\{2^{ks}\left\| \sum_{j=k-4}^\infty
			\Delta_{2^j}^v f R_{2^j}^v \phi
			\right\|_{L^r_xL^p_v}\right\}_{k=4}^\infty\right\|_{\ell^q}
			\\
			& \leq
			C\left(
			\left\| \Delta_{0}^v f \right\|_{L^r_xL^p_v} \left\| R_{0}^v \phi \right\|_{L^\infty_v}
			+\sum_{k=0}^\infty
			\left\| \Delta_{2^k}^v f 
			\right\|_{L^r_xL^p_v} \left\| R_{2^k}^v \phi
			\right\|_{L^\infty_v}
			\right)\\
			& +C\sum_{j=0}^\infty
			\left\|
			\Delta_{2^j}^v f
			\right\|_{L^r_xL^p_v}
			\left\|
			R_{2^j}^v \phi
			\right\|_{L^\infty_v}
			\left\|\left\{2^{ks}
			\right\}_{k=4}^{j+4}\right\|_{\ell^q}
			\\
			& \leq
			C
			\left\| \Delta_{0}^v f \right\|_{L^r_xL^p_v} \left\| R_{0}^v \phi \right\|_{L^\infty_v}\\
			& +C
			\left\|\left\{2^{ks}\left\| \Delta_{2^k}^v f 
			\right\|_{L^r_xL^p_v}\right\}_{k=0}^\infty\right\|_{\ell^q}
			\left\|\left\{ 2^{k-k(s\wedge 0)}
			\left\| R_{2^k}^v \phi
			\right\|_{L^\infty_v}\right\}_{k=0}^\infty\right\|_{\ell^{q'}}
			\\
			& \leq C
			\left\| f \right\|_{\widetilde L^r \left( \mathbb{R}^D_x ; B^{s}_{p,q} \left(\mathbb{R}^D_v\right) \right)}.
		\end{aligned}
	\end{equation}
	
	Finally, incorporating \eqref{paraproduct 1}, \eqref{paraproduct 2} and \eqref{remainder} into the decomposition \eqref{Bony}, we easily deduce that the estimate \eqref{paradifferential 1} holds true, which concludes the demonstration.
\end{proof}


\bibliographystyle{plain}


\bibliography{velocity_averaging}



\end{document}